\documentclass[a4paper,10pt]{article}
\usepackage[utf8]{inputenc}
\usepackage{amsmath, amsfonts, amsthm,amssymb,  bbm}
\usepackage{hyperref}
\usepackage{math,mathrsfs}
\usepackage[usenames, dvipsnames]{color}
\usepackage[left=3cm,right=3cm,top=3cm,bottom=3cm]{geometry}
\usepackage{aurical}
\usepackage[T1]{fontenc}
\newtheorem{lemma}{Lemma}[section]
\newtheorem{theorem}[lemma]{Theorem}
\newtheorem{proposition}[lemma]{Proposition}
\newtheorem{corollary}[lemma]{Corollary}
\newtheorem{remark}[lemma]{Remark}
\newtheorem{example}[lemma]{Example}

\newtheorem{definition}[lemma]{Definition}

\newcommand{\sT}{{ \mathscr{T}}}
\newcommand{\sP}{{\mathscr{P}}}
\newcommand{\spazio}[1]{\ell_c^{#1}}
\newcommand{\spazior}[1]{\ell_r^{#1}}
\newcommand{\Br}{B_r}

\newcommand{\di}{{\rm d}}
\newcommand{\bvf}{{\boldsymbol  \varphi}}
\newcommand{\tame}[1]{{\la \abs{#1} \ra}}

\numberwithin{equation}{section}

\renewcommand{\bar}{\overline}

\title{Tame majorant analyticity for   the Birkhoff map of the defocusing Nonlinear Schr\"odinger equation  on the circle}

\author{
A. Maspero
\footnote{
 International School for Advanced Studies (SISSA), Via Bonomea 265, 34136, Trieste, Italy, \newline
 \textit{Email: } \texttt{alberto.maspero@sissa.it}} 
  }

\date{\today}

\begin{document}
	
	\maketitle
	
	\begin{abstract}
	For the defocusing Nonlinear Schr\"odinger equation  on the circle, we construct a Birkhoff map  $\Phi$ which is tame majorant analytic in a neighborhood of the origin.  Roughly speaking, majorant analytic means that  replacing the coefficients of the Taylor expansion of $\Phi$  by their absolute values gives rise to a series (the majorant map) which  is uniformly and absolutely convergent, at least in a small neighborhood. Tame majorant analytic means that the majorant map of $\Phi$ fulfills tame estimates.\\
	The proof is based on a new tame version of the Kuksin-Perelman theorem \cite{kuksinperelman}, which is an infinite dimensional  Vey type theorem. 
 \end{abstract}

\section{Introduction and statement of the main result}
\subsection{Introduction}	
It is well known that the 
cubic defocusing	Nonlinear Schr\"odinger equation (dNLS) on the circle 
\begin{equation}
\label{1dnls0}
\im \dot \varphi = - \partial_{xx} \varphi +2 |\varphi|^2 \varphi \ , \qquad x \in \T:= \R/\Z
\end{equation}
is an integrable system \cite{zakharov_shabat,zakharov_manakov}.
The actual construction of action-angle coordinates is quite complicated, and it has been studied analytically in the last decade  by Gr\'ebert, Kappeler and collaborators in a series of works culminating in \cite{grebert_kappeler}.
 In particular these authors showed that there exists
a globally defined map $\Phi: \varphi \mapsto (z_k, \bar z_k)_{k \in \Z}$, the {\em Birkhoff map}, 
which introduces  {\em Birkhoff coordinates}, namely complex conjugates  canonical coordinates $(z_k, \bar z_k)_{k \in \Z}$,  with the property that the dNLS Hamiltonian, once expressed in such coordinates,  is a real analytic  function of the actions   $I_k := |z_k|^2$ alone.  
As a consequence, in the Birkhoff coordinates the flow \eqref{1dnls0} is conjugated to an infinite chain of nonlinearly coupled oscillators:
\begin{equation}
  \label{dnls.bc}
 \im  \dot z_k = \omega_k(I) z_k\qquad  \forall k \in \Z \ , 
  \end{equation}  
  where the $\omega_k(I)$ are frequencies depending only on the actions $\{I_k\}_{k\in \Z}$. \\
Recently  much effort has been made to understand various analytic properties of the Birkhoff map which are useful in applications.  
 Such properties include for example the $1$-smoothing of the nonlinear part of $\Phi$ \cite{kuksinperelman,kappelerschaad2}, two-sided polynomial estimates on the norm of $\Phi$ \cite{jan1},  extension of $\Phi$ to spaces of low regularity \cite{jan2}.
  
 In this paper we contribute to such analysis by investigating the property of tame majorant  analyticity  of the Birkhoff map. 
 Roughly speaking, an analytic map is  majorant analytic if 
 replacing the coefficients of its Taylor expansion   by their absolute value gives rise to a series (the majorant map) which  is uniformly and absolutely convergent, at least in a small neighborhood.  Then tame majorant analytic means that the majorant map  fulfills tame estimates.
 
  Here we prove that this is indeed true for the Birkhoff map of dNLS, at least  in a small neighborhood of the origin and  in appropriate topologies.
  Our construction of the Birkhoff map is quite different and less explicit than the one of  Gr\'ebert-Kappeler, however  the two Birkhoff maps coincide (up to some normalization) as they are both perturbations of the Fourier transform. 
   
While    Gr\'ebert and Kappeler provide  explicit and globally defined formulas for the action-angle coordinates of dNLS, 
our construction is valid only close to zero and it is based on the Kuksin-Perelman theorem \cite{kuksinperelman}, which is an infinite dimensional  Vey-type  theorem  \cite{vey,eliasson}.
   The Kuksin-Perelman theorem states that given a set of non-canonical coordinates, it is possible,  under certain circumstances, to deform them into canonical Birkhoff coordinates.
   Therefore the main idea of our proof is to construct a starting set of non-canonical coordinates (essentially following the construction of    B\"attig,  Gr\'ebert,  Guillot and Kappeler \cite{battig}), and then show that they   fulfill the assumptions of the Kuksin-Perelman theorem, so that they can be deformed into Birkhoff coordinates.
Actually we need  a little bit more,  since our aim is to construct {\em tame} Birkhoff coordinates. Therefore we prove that the 
starting coordinates are tame majorant analytic, and then we develop a tame version of the Kuksin-Perelman theorem,
 which guarantees that if the starting coordinates are tame-majorant analytic, so are the final Birkhoff coordinates. 
We think that  the tame version of Kuksin-Perelman theorem   could be interesting in itself.

 Majorant analyticity and tameness of vector fields are properties extremely useful in perturbation theory and when one wants to apply Birkhoff normal form techniques. 
Indeed  such properties are closed under composition, generation of flows and solution of homological equations, which are  the typical operations needed in a perturbative scheme.
 This makes tame majorant analyticity an extremely robust tool when investigating stability of solutions. 
  For example, majorant analyticity was  used by  Nikolenko \cite{nikolenko}  to obtain  Poincar\'e normal forms for some dissipative PDEs.
  Similarly,  tame majorant analyticity was exploited 
	by Bambusi and Gr\'ebert \cite{bambusi.grebert}  to develop Birkhoff normal form theory for a wide class of Hamiltonian PDEs, 
and by 
Cong,  Liu and Yuan	\cite{cong} to study long time stability of small KAM tori of NLS with external potential; see also Berti, Biasco and Procesi \cite{bbp13} for applications to KAM theory.

Concerning the  usefulness of tame  properties in  perturbation theory, the idea is essentially the following. Tame estimates are estimates linear in the highest norm, a typical example being  $\norm{u^n}_{H^s} \leq C  \norm{u}_{H^s} \norm{u}_{H^1}^{n-1}$; such estimates  allow to control the size of a  nonlinear term in a high regularity norm by conditions on the size of the function in a lower regularity norm.
In  \cite{bambusi.grebert} this property is exploited to show that, in the algorithm of  Birkhoff normal form,  large parts of the nonlinearity   are actually  very small in size and therefore harmless. 
A different applications of tame estimates is in differentiable Nash-Moser scheme, see e.g. \cite{bbp}; in this case the employ  of tame estimates is one of the necessary ingredients for the convergence of the quadratic scheme.\\
Also our interest in tame majorant analyticity of the Birkhoff map of dNLS  was first motivated by applications:   in the  paper \cite{masp_nls}, we  discuss the stability  of small finite gap solutions of \eqref{1dnls0} when they are considered as solutions of the defocusing NLS on $\T^2$. We first introduce Birkhoff coordinates and then perform  a few steps of Birkhoff normal form.  As in \cite{bambusi.grebert} this requires the {\em majorant analyticity} of the Hamiltonian. 

Furthermore we think that properties of tame majorant analyticity of the Birkhoff map might be useful in the study of long time stability of perturbed dNLS.

\subsection{Main result} 
 
  As usual it is convenient to augment  \eqref{1dnls0} with the conjugated equation for $\bar \varphi$ and  to consider $(\varphi, \bar\varphi)$ as independent variables belonging to the phase space  $L^2_c := L^2(\T, \C) \times L^2(\T, \C)$  with elements $\varphi = (\varphi_1, \varphi_2)$. 
   More generally we denote $H^s_c := H^s(\T, \C) \times H^s(\T, \C)$ for any $s \in \R$. 
 The dNLS Hamiltonian is given by 
 $$
 H_{{\rm NLS}}(\varphi_1, \varphi_2) = \int_\T \Big(\partial_x \varphi_1(x) \, \partial_x \varphi_2(x) + \varphi_1^2(x) \, \varphi_2^2(x) \Big) \, dx
 $$
 and the corresponding Hamiltonian system is 
 \begin{equation}
 \begin{cases}
 \im \dot \varphi_1 = \partial_{\varphi_2} H_{{\rm NLS}} = - \partial_{xx} \varphi_1 + 2 \varphi_2 \varphi_1^2 \\
 \im \dot \varphi_2 = -\partial_{\varphi_1} H_{{\rm NLS}} =  \partial_{xx} \varphi_2 - 2 \varphi_1 \varphi_2^2  
 \end{cases} \ .
 \end{equation}
 In such a way  equation \eqref{1dnls0} is obtained by restricting the system above to the {\em real} invariant subspace
 \begin{equation}
 \label{Hsr}
  H^s_r := \{ (\varphi_1, \varphi_2) \in H^s_c \colon \varphi_2 = \bar \varphi_1 \} 
 \end{equation}
 of states of {\em real} type.  We denote  $L^2_r := H^0_r$ and in such space we introduce the {\em real} scalar product $\la \cdot , \cdot \ra$ and the symplectic form $\Omega_0$  defined for $\bvf_1 \equiv ( \varphi_1, \bar \varphi_1)$  and $\bvf_2 \equiv ( \varphi_2, \bar \varphi_2)$ by
 \begin{equation}
 \label{scalar_product2}
 \la \bvf_1 , \bvf_2 \ra := 2 {\rm \, Re } \int_{\T} \varphi_1(x) \, \bar \varphi_2(x) \, dx \ , \qquad \Omega_0(\bvf_1, \bvf_2) := \la E \bvf_1, \bvf_2 \ra \ ,
 \end{equation}
where  $E := \im$.

It is useful to identify functions $(\varphi_1, \varphi_2)$ with their Fourier coefficients. Thus we denote by $\cF: L^2(\T, \C) \to \ell^2(\Z, \C)$ the Fourier transform and associate to $\varphi_1$ its sequence of Fourier coefficients
  $\{\xi_j \}_{j \in \Z} = \cF(\varphi_1)$ and to $\varphi_2$ its sequence of Fourier coefficients $\{ \eta_{-j} \}_{j \in \Z}= \cF(\varphi_2)$:
 \begin{equation}
 \label{fcvarphi}
 \varphi_1(x) = \sum_{j \in \Z} \xi_j \, e^{-\im 2j \pi x} \ , \qquad  \varphi_2(x) = \sum_{j \in \Z} \eta_j \, e^{\im 2j \pi x} \ . 
 \end{equation}
Clearly
$$
(\varphi_1, \varphi_2) \in L^2_c \quad \Longleftrightarrow \quad (\xi, \eta) \in \ell^2_c := \ell^2(\Z, \C) \times \ell^2(\Z, \C) \ ,
$$ 
and
$$
(\varphi_1, \varphi_2) \in L^2_r \quad \Longleftrightarrow \quad (\xi, \eta) \in \ell^2_r :=\{ (\xi, \eta) \in  \ell^2_c \colon \bar \xi_j = \eta_j \ , \ \ \forall j \in \Z \} \ . 
$$
We endow  $\spazior{2}$ with the {\em real } scalar product $\la \cdot, \cdot \ra$  and symplectic form $\omega_0$ defined for  $\xi^1 \equiv (\xi^1, \bar \xi^1)$ and $\xi^2 \equiv (\xi^2, \bar \xi^2)$ by
\begin{equation}
\label{scalar_productR}
\la \xi^1,\xi^2\ra:= 2 \, {\rm Re}\, \sum_{j \in \Z} \xi^1_j\,\overline{\xi^2_j}\ , 
\qquad 
\omega_0(\xi^1,\xi^2):=\la E \, \xi^1,\xi^2\ra \ ,
\end{equation}
and one has  $\omega_0 := (\cF^{-1})^* \Omega_0$. \\

We are interested  also in more general spaces  which we now introduce.
 It is more convenient to define such spaces in term of the Fourier coefficients $(\xi, \eta)$ of $(\varphi_1, \varphi_2)$.  So for  any real $1 \leq p \leq 2$, $s \geq 0$  define  
\begin{equation}
\label{spaziooo}
\spazio{p, s}:= \left\lbrace (\xi, \eta) \in \spazio{2} \colon  \quad  \norm{(\xi, \eta)}_{p,s}   < \infty \right\rbrace \ , 
\end{equation}
where $\norm{(\xi, \eta)}_{p,s}  := \norm{\xi}_{p,s}  + \norm{\eta}_{p,s} $ and 
\begin{equation}
\label{sole}
\norm{\xi}_{p,s}:= \left(\sum_{j \in \Z} \la j \ra^{ps} \,  |\xi_j|^p \right)^{1/p} \ ; 
\end{equation}
here $\la j \ra := 1 + |j|$.  
Correspondingly $\spazior{p, s} := \left\lbrace (\xi, \bar \xi) \in \spazio{p,s} \right\rbrace$ with the induced norm. 
Note that when $s  = 0$, then the norm \eqref{sole} is simply the $\ell^p$ norm of the Fourier coefficients of $(\varphi_1, \varphi_2)$; therefore the spaces $\spazio{p,s}$ can also be thought as  {\em weighted Fourier Lebesgue spaces}.
Furthermore 
$$
(\varphi_1, \varphi_2) \in H^s_c \quad \Longleftrightarrow \quad (\xi, \eta) \in \spazio{2,s} \ , 
\qquad 
(\varphi_1, \varphi_2) \in H^s_r \quad \Longleftrightarrow \quad (\xi, \eta) \in \spazior{2,s} \ . 
$$ 
We denote by $B^{p,s}(\rho)$ the ball with center $0$ and radius $\rho$ in the topology of $\spazio{p,s}$, 
and by $B_r^{p,s}(\rho)$ the same ball in $\spazior{p, s}$.
 For $s = 0$, we will write simply $\spazio{p} \equiv \spazio{p,0}$ and $B^p(\rho) \equiv B^{p,0}(\rho)$.

In order to state our main theorem we need to introduce the concept of tame majorant analytic map more precisely. 
Given a $\rho >0$, $1 \leq p \leq 2$, let  $F:  B^{p,s}_r(\rho) \to \ell^{p,s'}_r$ be  a real analytic map in a neighborhood of the origin\footnote{here real analytic in a neighborhood of the origin means that there exists an analytic map $\wt F: B^{p,s}(\rho) \to \ell^{p,s'}$ (defined in a  complex ball)  which coincides with $F$ on the real subspace $B^{p,s}(\rho) \cap \spazior{p,s} \equiv B^{p,s}_r$}. 
 Write $F(\xi, \bar \xi) = (F_j(\xi, \bar \xi), \bar{F_j(\xi, \bar \xi)})_{j \in \Z}$ in components 
 and  expand each component $F_j(\xi, \eta)$ into its  uniformly convergent Taylor series in a neighborhood of the origin:
$$
F_j(\xi, \eta) = \sum_{|K| + |L| \geq 0} F_{KL}^j \ \xi^K \eta^{L} \ . 
$$
Define  
$$
\und  F_j(\xi, \eta) := \sum_{|K| + |L| \geq 0} \abs{F_{KL}^j}\  \xi^K \eta^{L} \ 
$$
and the majorant map  $\und{F}$ component-wise by $\und{F}(\xi, \bar \xi) = \big(\und{F_j}(\xi, \bar \xi), {\und{F_j}(\bar \xi, \xi)}\big)_{j \in \Z}$.
Then $F$ will  be said to be {\em majorant analytic}  if  $\exists\,\rho_*>0$ s.t. $\und F$ defines a real analytic map in a neighborhood of the origin mapping $B^{p,s}_r(\rho_*) \to \ell^{p,s'}_r$.\\
Given $0 \leq s \leq s' \leq s''$,  $F$ will we said to be $(p, s, s', s'')$- {\em tame majorant analytic}  if it is majorant analytic $B^{p,s}_r(\rho_*) \to \ell^{p,s}_r$ and furthermore $\und F$  restricts to a real analytic map $B^{p,s}_r(\rho_*)\cap \spazior{p,s'} \to \spazior{p,s''}$ fulfilling
\begin{equation}
\label{tame}
\sup_{\zeta \in B^{p, s}(\rho_*) \cap \spazior{p,s'} } \frac{\norm{\und F(\zeta)}_{p, s''} }{ \norm{\zeta}_{p,s'}} < \infty \ . 
\end{equation}
 Note that, in the estimate \eqref{tame}, the supremum  is taken over $B^{p, s}(\rho_*) \cap \spazior{p,s'}$, namely on all the elements of $\spazior{p,s'}$ which belong to a ball of fix radius in the weaker topology of $\spazior{p,s}$.   
As we will show below (see Remark \ref{rem:tame.c}), \eqref{tame} implies that each polynomial of the Taylor expansion of  $F$ is tame in the sense of polynomial maps.\\

Our main  theorem is the following one:
\begin{theorem}
\label{main} 
There exists  $\rho_0>0$ and a real  analytic and majorant analytic map $\Phi: B^2_r(\rho_0) \to \spazior{2}$ s.t. the following is true:
\begin{itemize}
\item[(i)] $\Phi$ is canonical: $\Phi^*\omega_0 = \omega_0$.
\item[(ii)] The map $\Phi$ is a perturbation of the identity; more precisely $\di\Phi(0,0) = \uno $,  with $\uno$
the identity map.
\item[(iii)] For any reals  $1 \leq p \leq 2$,  $s \geq 1$,   $\exists \, 0<\rho_s<\rho_0$ s.t.  $\Phi - \uno$ restricts to a  $(p, 0, s, s)$- tame majorant analytic  map  $B^{p}(\rho_s) \cap \spazio{p, s} \to  \spazio{p, s}$. Moreover there exists $C >0$, independent of $s$, s.t. 
 $\forall \, 0 < \rho \leq \rho_s$ one has 
$$
 \norm{\und{\Phi - \uno}(\xi, \eta)}_{p, s} \leq  C \, 4^s \rho^2 \norm{(\xi, \eta)}_{p,s} \ ,  \qquad  \forall (\xi, \eta) \in B^{p}(\rho) \cap \spazio{p,s}  \ . 
$$
The same is true for $\Phi^{-1} - \uno$, with different constants.
\item[(iv)] $\Phi$ is a local Birkhoff map in the following sense: for any $(\xi, \bar \xi) \in B^{2}_r(\rho_1)\cap \spazior{2,1}$,
 define 
 $(z_j, \bar z_j) := \Phi_j(\xi, \bar \xi)$. Then the  integrals of motion of dNLS are real analytic functions of the actions
$I_j = |z_j|^2$. In particular, the Hamiltonian $H_{NLS}(\varphi) \equiv  \int_\T \abs{\derx \varphi(x)}^2 dx+ \int_\T \abs{\varphi(x)}^4 dx$, 
the mass $M(\varphi):= \int_\T \abs{\varphi(x)}^2 dx$ and the momentum $P(\varphi):= \int_\T \bar \varphi(x) \im \derx \varphi(x) dx$ have the form
\begin{align}
&\left(H_{NLS}\circ \cF^{-1} \circ \Phi^{-1}\right)(z, \bar z) = h_{nls}(\ldots, I_{-1}, I_0, I_1, \ldots) \ , \\
&\left(M\circ \cF^{-1}\circ \Phi^{-1}\right)(z, \bar z) = \sum_{j \in \Z} I_j \ , \\
&\left(P\circ \cF^{-1}\circ \Phi^{-1}\right)(z, \bar z) = \sum_{j \in \Z}  j I_j \ . 
\end{align}
\end{itemize}
Finally $\Phi - \uno$ is 1-smoothing, in the following sense:
\begin{itemize}
\item[(v)]  For any reals  $1 \leq p \leq 2$, $ s \geq 1 $,  
$\exists \, 0 < \rho_s'<\rho_0$ s.t.  $\Phi - \uno$ restricts to a  $(p, 1, s, s+1)$-tame majorant analytic  map  $B^{p, 1}(\rho_s') \cap \spazio{p, s} \to  \spazio{p, s+1}$. Moreover there exists $C' >0$, independent of $s$,  s.t.   
 $\forall \, 0 < \rho \leq \rho_s'$ one has 
$$
 \norm{\und{\Phi - \uno}(\xi, \eta)}_{p, s+1} \leq  C' 4^s  \rho^2 \norm{(\xi, \eta)}_{p,s} \ ,  \qquad  \forall (\xi, \eta) \in B^{p,1}(\rho) \cap \spazio{p,s}  \ . 
$$
The same is true for $\Phi^{-1} - \uno$, with different constants.
\end{itemize}
\end{theorem}	
	
	The main novelty of Theorem \ref{main}  are the tame majorant analytic properties of the Birkhoff map illustrated in  item $(iii)$ and $(v)$. 
In particular item $(iii)$ shows that $\Phi$ is convergent provided $(\xi, \eta)$ are small  in the low regularity space $\spazio{2}$, despite having  large norm in higher regularity spaces. 
This turns out to be useful in applications (e.g. in perturbation theory), since in such a way one has typically to control only  the low regularity norms of the solution.
Finally  item $(v)$ shows that the nonlinear part of $\Phi$ is  1-smoothing provided the variables  $(\xi, \eta)$  are chosen at least  in  $\spazio{p,1}$. In such a way, one recovers (in a neighborhood of the origin), the 1-smoothing property of the Birkhoff map proved in \cite{kappelerschaad2}.

	We are  actually able to prove convergence  of the Birkhoff map in spaces more general than $\spazio{p,s}$; for example we are able to deal with weighted Fourier Lebesgue spaces where the weight $\la j \ra^{s}$ in \eqref{sole} is replaced by a more general weight $w$, e.g. by an  analytic weight of the form 
\begin{equation}
\label{www}
w_j := \la j \ra^{s} \, e^{a |j|} \ , \qquad a > 0 , \ \ j \in \Z \ ;
\end{equation}
in such a case	the norm is defined by  
$\norm{\xi}_{p,s,a}:=
 \left(\sum_{j \in \Z} \la j \ra^{ps} \, e^{p a |j|} \, |\xi_j|^p \right)^{1/p}$ 
 and the space  by 
 $\spazio{p, s,a}:= 
 \left\lbrace (\xi, \eta) \in \spazio{2} \colon
      \norm{\xi}_{p,s, a}  + \norm{\eta}_{p, s, a }  < \infty \right\rbrace$. 
Then we have the following theorem 
	\begin{theorem}
	\label{main2}
With $\rho_0$ as in Theorem \ref{main},  $\forall \, 1 \leq p \leq 2$,  $s, a  \geq 0$, the map   $\Phi - \uno$ of Theorem \ref{main} restricts to a majorant analytic  map  $B^{p, s, a}(\rho_0) \to  \spazio{p, s, a}$. Morever there exists $C >0$, independent of $s,a$, s.t.  $\forall \, 0 < \rho \leq \rho_0$ one has 
$$
\sup_{\norm{(\xi, \eta)}_{p,s, a} \leq \rho} \norm{\und{\Phi - \uno}(\xi, \eta)}_{p, s, a} \leq  C \rho^3 \ .
$$
The same is true for $\Phi^{-1} - \uno$, with different constants.
	\end{theorem}
In this case we prove just majorant analyticity (and not tameness), but in spaces of analytic functions.	Note that  the domain of (majorant) analyticity of the Birkhoff map  does not shrink to 0 as $s,a$ go to infinity. This is a consequence of an explicit control of every constant  in the proof of the quantitative Kuksin-Perelman theorem \ref{KP}. 
	 Finally we mention that we are able to treat even more general weighted Fourier Lebesgue spaces, giving sufficient conditions for the weight, see Section \ref{sec:app}. \\

An immediate corollary of Theorem \ref{main2} concerns the dNLS dynamics in  $\spazior{p,s,a}$. Recall that the Cauchy problem for \eqref{1dnls0} is well posed in $L^2_r$ \cite{bou}. In Birkhoff coordinates, the flow of \eqref{1dnls0} is given by  
\begin{equation}
\label{bc.ev}
(z_j(t), \bar z_j(t)) = \left( e^{-\im \omega_j(I) t} z_j(0) \ , 
e^{\im \omega_j(I) t} \bar z_j(0) \right) \ ,  \qquad \forall j \in \Z \ , 
\end{equation}
where $\omega_j := \partial_{I_j} H_{NLS}\circ \Phi^{-1}$   is the $j^{th}$ frequency, which depends only on the actions $(I_k)_{k \in \Z}$. 
Then in the original Fourier coordinates $\xi = \cF(\varphi)$, provided $\xi(0)$ is small enough to belong to the domain of the Birkhoff map,   one has 
$(\xi(t), \bar \xi(t)) = \Phi^{-1}\left( z(t), \bar z(t)\right)$, where $z(t) := (z_j(t))_{j \in \Z}$.
Since the norm $\norm{\cdot}_{p,s,a}$ is invariant by the dynamics \eqref{bc.ev},  one gets the following result:
	\begin{corollary}
	\label{cor:norm}
	There exist constants $\rho_*, C_* >0$ s.t. for any $1 \leq p \leq 2$,  $s, a \geq 0$  the following holds true. Consider the solution $\xi(t) = \cF(\varphi(t))$ of \eqref{1dnls0} corresponding to initial data $\xi_0 = \cF(\varphi_0)\in B^{p,s,a}_r(\rho)$, $\rho \leq \rho_*$.  Then one has 
	\begin{equation}
	\label{gsn}
		\sup_{t \in \R}\norm{\xi(t)}_{p,s, a} \leq \rho(1+C_* \rho^2) \ .
	\end{equation}
	\end{corollary}
	
Note that in Corollary \ref{cor:norm} there is no loss of analyticity of the solution (as it happens in \cite{kapp_pos} for example), in the sense that  exponential decay of the initial datum is preserved by the flow. 
The point is that we work only with small initial datum, for which we know that the Birkhoff map and its inverse map $B^{p,s,a}(\rho) \to \spazior{p,s,a}$ with {\em the same } $a$.

Before closing this introduction, we recall some previous works on analytic properties of  the Birkhoff map of infinite dimensional integrable systems. \\
Concerning  majorant analyticity of the Birkhoff map, the first result was  proved by Kuksin and Perelman  in the case of KdV on $\T$ \cite{kuksinperelman}. In particular, these authors proved that in a small neighborhood of the origin in $H^s_r$, $\forall s \geq 0$,  the nonlinear part of the Birkhoff map is both majorant analytic and 1-smoothing. The techniques of this paper were extended by Bambusi and the author \cite{masp_toda} in order to deal with the Toda lattice  with $N$ particles, $N$ arbitrary large. 

Later on, it was proved  in \cite{kappelerschaad}  that the nonlinear part of the Birkhoff map of the KdV on $\T$ is globally 1-smoothing, and the same is true also for the Birkhoff map of the  dNLS  on $\T$ \cite{kappelerschaad2}  (see also  \cite{masp_kdv} for the case of KdV on $\R$). However, none of these papers addresses the question of tameness.

Also the use  of Fourier-Lebesgue spaces (namely spaces with norms like \eqref{sole} with $p \neq 2$) is not new in this context; e.g. in   \cite{masp_kdvf} and  \cite{jan2} the Birkhoff map of KdV and of dNLS were extended to weighted Fourier-Lebesgue spaces in order to study analytic properties of the action-to-frequency map $I \mapsto \omega(I)$.

Concerning tameness properties, recently Kappeler and Montalto \cite{kappelermontalto} constructed  real analytic, canonical coordinates for the
dNLS  on $\T$, which are defined in neighborhoods of families of finite dimensional invariant tori, and which  satisfy  tame estimates. However such coordinates are not Birkhoff coordinates, and the dNLS  Hamiltonian, once expressed in such coordinates,  is in
normal form only up to order three.
On the contrary, our coordinates are well defined only in a neighborhood of the origin, but the dNLS Hamiltonian, written in such coordinates, is in normal form at every order.

Finally we want to comment  on Corollary \ref{cor:norm}, which  shows that  weighted Fourier-Lebesgue  norms of the solution are uniformly bounded in time. 
As a consequence, there is no growth of Sobolev norms. 
The problem of giving  upper bounds  of the form \eqref{gsn}   has been widely studied both for linear time dependent and nonlinear Schr\"odinger equations  (see e.g. \cite{masp_sch, BGMR2} for the linear case, \cite{sohinger, visciglia} for the nonlinear one and references therein). 
In case of linear Schr\"odinger equations quasiperiodic in time, $\im \dot \psi = -\Delta \psi + V(\omega t, x)\psi$,  it is known that the Sobolev norms of the solution can be  uniformly bounded,  provided the frequency vector
$\omega$ is well chosen, see e.g.  \cite{elikuk} for bounded perturbations on $\T^d$ (see also  \cite{BGMR1} for some special  perturbations on $\R^d$, and reference therein).

 In case of dNLS, the inequality \eqref{gsn} is well
known for data in $H^s_r$, $s \in \N$, and can be proved using the conservation laws of the dNLS hierarchy. In $H^s_r$, $s >1 $ real, inequality \eqref{gsn} is proved in \cite{kappelerschaad3}. The novelty of inequality \eqref{gsn} is to treat the case $1 \leq p <2$ and weighted spaces.
We point out that the uniform bound \eqref{gsn} is not a mere  consequence of integrability, but of the stronger property that the Birkhoff map preserves the topology, see Theorem \ref{main}. 
Indeed G\'erard and Grellier proved that  the  cubic Szeg\H{o} equation on $\T$ is integrable \cite{szego,szego2}, and nevertheless there are phenomenons of growth of Sobolev norms  \cite{szego3}.

\vspace{2em}

\subsection{Scheme of the proof}
In order to prove Theorem \ref{main}, we  apply a tame version of the Kuksin-Perelman theorem  \cite{kuksinperelman}
to the dNLS equation.
We recall that the  starting point of the  Kuksin-Perelman theorem 
is to construct  a map  $\zeta \mapsto \Psi(\zeta)$ (not symplectic and locally defined), s.t.   the quantities $|\Psi_j(\zeta)|^2$ are in involution,
 the level sets $|\Psi_j(\zeta)|^2=c_j$ give  a foliation in invariant tori, and $\Psi$ and $\di \Psi^*$ are majorant analytic maps. 
Then Kuksin and Perelman \cite{kuksinperelman} showed that it is possible to deform $\Psi$ into a new map $\Phi$ which is symplectic,  majorant analytic and it is a Birkhoff map, in the sense that $(z, \bar z) := \Phi(\xi, \bar \xi)$ are complex Birkhoff coordinates.

Therefore the first step of our proof is to prove a  tame version of the Kuksin-Perelman theorem, which  tells that if  $\Psi$ (namely the original map)  is  a tame majorant analytic map, so is the  new map  $\Phi$. 
In order to prove such a theorem, we revisit the proof of the Kuksin-Perelman theorem (actually, of the quantitative version of the Kuksin-Perelman theorem proved in \cite{masp_toda}), and prove that the algorithm of construction of $\Phi$ can be made tame, in the sense that at each step of the construction we can control  quantities as in  \eqref{tame} for every object involved.
This turns out to be true since the Kuksin-Perelman algorithm  is based on a combination of some basic operations (like composition of functions, inversion of functions, generation of flows, and solution of a system of equations) which can be made tame.

Then the second step of our proof is to apply the tame Kuksin-Perelman theorem to the dNLS. 
This amounts to  construct the starting map $\zeta \mapsto \Psi(\zeta)$ and to prove that it fulfills the assumptions of the  tame Kuksin-Perelman theorem (in particular, that $\Psi$ is tame majorant analytic). 
 Here we adapt to dNLS the ideas already employed in \cite{kuksinperelman} for the KdV on $\T$ and in \cite{masp_toda} for the Toda lattice
 (see also the pioneering work of  B\"attig, Gr\'ebert, Guillot and Kappeler \cite{battig}). 
 The strategy is to construct $\Psi$ by exploiting the integrable structure of dNLS, and in particular the Lax pair of dNLS. More precisely,  starting from the spectral data of the Lax operator,  one constructs perturbatively  a map $\zeta \mapsto \Psi(\zeta)$ s.t.  the quantities  $|\Psi_j|^2$ equal the spectral gaps $\gamma_j^2$, which are  real analytic functions in involution.
 The main technical challenge is to show that the map $\Psi$ is tame majorant analytic. 
 This is proved by computing explicitly every polynomial in the Taylor expansion of $\Psi$, in order to have a precise formula for a majorant map.

The paper is structured in the following way: in Section 2 we recall the setup of weighted Sobolev spaces and  state the tame Kuksin-Perelman theorem. Its proof is a variant of the proof written in  \cite{masp_toda}, therefore we postpone it to Appendix  \ref{app:KP}. In Section 3 we consider the dNLS equation and construct the map $\Psi$ required by the tame Kuksin-Perelman theorem, and show that it is tame majorant analytic.

\vspace{1em}
\noindent{\bf Acknowledgments.}
We wish to thank Marcel Guardia, Zaher Hani, Emanuele Haus, Thomas Kappeler  and Michela Procesi for  discussions and suggestions.
During the preparation of this work, we were  hosted at Laboratoire Jean Leray, University of Nantes (France) and were supported by ANR-15-CE40-0001-02 ''BEKAM`` of the Agence Nationale de
la Recherche. We  wish to thank  the laboratory of Nantes for the hospitality and the uncountable scientific exchanges.\\
Currently we are partially supported by PRIN 2015 ``Variational methods, with applications to problems in mathematical physics and geometry".

	\section{The tame Kuksin-Perelman theorem}
\label{sec:KP}

We prefer to work in the setting of abstract weighted Fourier-Lebesgue spaces, which we now recall.
 First we define  {\em weight} a function $w: \Z \to \R$ such that  $w_j> 0$ $\forall j \in \Z$.
A weight will be said to be {\em symmetric} if $w_{-j} = w_j$ $ \ \forall j \in \Z$ and {\em sub-multiplicative} if $w_{j+i} \leq w_i \, w_j$ $\ \forall i,j \in \Z$. Given two weights $w$ and $v$, we will say that
$v \leq w$  iff $v_j \leq w_j,$  $\forall j \in \Z$.\\
Given a weight $w$ we define  for any $\R \ni p \geq 1$  the space $\ell^{p,w} $ of complex sequences $\xi = \{ \xi_j\}_{j \in \Z} $ with norm
\begin{equation}
\label{ell_w.norm}
 \norm{\xi}_{p,w}:=\left( \sum_{j \in \Z} w_j^p \, |\xi_j|^p \right)^{1/p}  < \infty .
\end{equation}
 We   denote by $\spazio{p,w}$ the complex Banach space 
 $\spazio{p,w}:=\ell^{p,w} \oplus \ell^{p,w} \ni (\xi, \eta)\equiv \zeta$ endowed with the norm 
 $$
 \norm{\zeta}_{p,w} \equiv \norm{(\xi, \eta)}_{p,w} := \norm{\xi}_{p,w} + \norm{\eta}_{p,w} \ .
  $$
 We denote by $\spazior{p,w}$ the {\em real} subspace of $\spazio{p,w}$ defined by
  \begin{equation}
\label{C^N_w}
\spazior{p,w}:=\left\{(\xi, \eta) \in \spazio{p,w}\, :  \ \eta_j = \overline{\xi}_j \ \ \  \forall \, j \in \Z  \right\} \ . 
  \end{equation}
  We endow such a space with the real scalar product and symplectic form  \eqref{scalar_productR}.
   We will denote by $B^{p,w}(\rho)$ (respectively $\Br^{p,w}(\rho)$) the
 ball in the topology of $\spazio{p,w}$ (respectively $\spazior{p,w}$) with center $0$ and radius $\rho>0$. Clearly if   $w_j = 1 \, \forall j$ one has $\ell^{p,w}_c \equiv \ell^p_c:= \ell^p(\Z, \C) \times \ell^p(\Z, \C)$. In this case we denote the norm simply by  $\norm{\cdot}_{p}$ and the ball of radius $\rho$ by $B^p(\rho)$. Similarly  we  write  $\ell^p_r \equiv \ell^{p,w}_r$ and $B^{p}_r (\rho)\equiv B^{p,w}_r(\rho)$. \\
  As for any $1 \leq p \leq 2$ and weight $w>0$  one has the inclusion $\spazior{p,w}\hookrightarrow \spazior{2} $, the scalar product and the symplectic form \eqref{scalar_productR} are well defined on $\spazior{p,w}$ as well. 
  \begin{remark}
  The space $\spazio{p,s}$ defined in \eqref{spaziooo} coincides with the weighted space $\spazio{p,w}$ choosing the weight $w= \{ \la j \ra^s  \}_{j \in \Z}$.
  \end{remark}

Given a smooth function $F: \spazior{p,w} \to \R$, we denote by $X_F$ the
Hamiltonian vector field of $F$, given by $ X_F = J \nabla F$, where
$J = E^{-1}$.  
For $F, G: \spazior{p,w} \to \R$ we denote by $\{F,G\}$ the
Poisson bracket (with respect to $\omega_0$): $\{F,G\}:= \la \nabla F,
J \nabla G\ra$ (provided it exists).  We say that the
functions $F, G$ \textit{commute} if $\{F,G\} = 0$. \\
  For $\cX, \cY$ Banach spaces, we
 shall write $\cL(\cX,\cY)$ to denote the set of linear and bounded
 operators from $\cX$ to $\cY$. For $\cX=\cY$ we will write just $\cL(\cX)$.\\

\noindent Recall that a map $\wt P^n:(\spazio{p,w})^n\to \cB$, with $\cB$ a  Banach space,
is said to be  $n$-\textit{multilinear} if $\wt P^n(\zeta^{(1)},\ldots,\zeta^{(n)})$ is $\C$-linear in
each variable $\zeta^{(j)}\equiv (\xi^{(j)}, \eta^{(j)})$; a $n$-multilinear map is said to be \textit{bounded} if there
exists a constant $C>0$ such that
$$\norm{\tilde P^n(\zeta^{(1)},\ldots,\zeta^{(n)})}_\cB \leq C \norm{\zeta^{(1)}}_{p,w}\ldots \norm{\zeta^{(n)}}_{p,w}  \ , \quad
\forall \zeta^{(1)},\ldots, \zeta^{(n)} \in \spazio{p,w}.
$$
 Correspondingly its norm is defined by
$$\norm{\wt P^n}:=\sup_{\norm{\zeta^{(1)}}_{p,w}, \cdots, \norm{\zeta^{(n)}}_{p,w}\leq 1}{\norm{\wt P^n(\zeta^{(1)},\cdots,\zeta^{(n)})}_\cB}.$$
A map $P^n:\spazio{p,w}\rightarrow \cB $ is a \textit{homogeneous polynomial} of order $n$ if there exists
a $n$-multilinear map $\wt{P}^n:(\spazio{p,w})^r\rightarrow \cB $ such that
\begin{equation}
\label{polin}
P^n(\zeta)=\wt{P}^n (\zeta,\ldots,\zeta) \ , \quad \forall \zeta\in \spazio{p,w}\ .
\end{equation}
A $n$-homogeneous  polynomial is bounded if it has  finite norm
$$\norm{P^n}:=\sup_{\norm{\zeta}_{p,w}\leq 1}\norm{P^n(\zeta)}_\cB\ .$$

\begin{remark}
Clearly $\norm{P^n}\leq \norm{\wt{P}^n}$. Furthermore one has 
$\norm{\wt{P}^n}\leq e^n\norm{P^n}$  -- cf. \cite{mujica}.
\end{remark}

\begin{remark} It is easy to see that a multilinear map and the corresponding polynomial are continuous (and analytic) if and only if they are bounded.
\end{remark}

A map $F: \spazio{p,w} \rightarrow \cB$ is said to be an \textit{ analytic
  germ} if there exists $\rho>0$ such that $F: B^{p,w}(\rho) \to
\cB$ is analytic. Then $F$ can be written as a power series
absolutely and uniformly convergent in $B^{p,w}(\rho)$:
$F(\zeta)=\sum_{n\geq 0}{F^n(\zeta)}$.  
Here $F^n(\zeta)$ is a homogeneous
polynomial of degree $n$ in the variables $\zeta=(\xi, \eta)$. We will write $F =
O(\zeta^N)$ if in the previous expansion $F^n(\zeta)= 0$ for every $n < N$.  \\
Let $U \subset \spazior{p,w}$ be open.  A map $F: U \to \cB$ is said to
be a {\sl real analytic germ}  on $U$ if for each point $(\xi, \bar \xi) \in U$ there exist a
neighborhood $V$ of $(\xi, \bar \xi)$ in $\spazio{p,w}$ and an analytic germ
 which coincides with $F$ on $U
\cap V$.\\
Let now $F:U\subset \spazio{p,w^1} \to \spazio{p,w^2}$ be an analytic
map. We will say that $F$ {\sl is real for real sequences} if $ F (U
\cap \spazior{p,w^1}) \subseteq \spazior{p,w^2}$, namely $F(\xi,
\eta)=(F_1(\xi, \eta), F_2(\xi, \eta))$ satisfies $\overline{F_1(\xi,
  \bar \xi)} = F_2(\xi, \bar \xi)$. Clearly, the restriction $\left. F
\right|_{U \cap \spazior{p,w^1}}$ is a real analytic map.

\subsection{Majorant analytic maps}

Let $P^n:\spazio{p,w}\to\cB$ be a homogeneous polynomial of order $n$; assume $\cB$ separable and let
$\left\{\bb_m\right\}_{m \in \Z}\subset \cB$ be a basis for the space $\cB$. Expand $P^n$ as follows
\begin{equation}
\label{exp.1}
P^n(\zeta)\equiv P^n(\xi, \eta)=\sum_{\substack{ |K| + |L| = n \\ m \in \Z}} P^{n,m}_{K, L} \  \xi^K \ \eta^L  \ \bb_m ,  
\end{equation}
where $K, L \in \N^\Z_0$, $\N_0 = \N \cup \{ 0 \}$,
 $|K|:= \sum_{j \in \Z} K_j$,
  $\xi \equiv ( \xi_j)_{j \in \Z}$ and $\xi^K := \prod_{j \in \Z} \xi_j^{K_j}$.
\begin{definition}
The modulus of a polynomial $P^n$ is the polynomial $\und{P^n} $
defined by
\begin{equation}
\label{exp.11}
\und{P^n}(\xi, \eta):=\sum_{\substack{ |K| + |L| = n \\ m \in \Z }} \abs{P^{n,m}_{K, L}} \,  \xi^K \eta^L \bb_m \ .
\end{equation}
A polynomial $P^n$ is said to have {\sl bounded modulus} if $\und{P^n}$ is
a bounded polynomial.
\end{definition}

We generalize now the notion of majorant analytic map given in the introduction.
\begin{definition}
\label{def.na} 
An  analytic germ $F: \spazio{p,w} \to \cB$ is said to be  {\sl majorant
  analytic} if there exists $\rho>0$ such that 
\begin{equation}
\label{exp.3}
\und F(\zeta):=\sum_{n\geq 0} \und{F^n}(\zeta)
\end{equation}
is absolutely and uniformly convergent in  $B^{p,w}(\rho)$. In such a case we will write
$F\in \cN_\rho(\spazio{p,w}, \cB)$. $\cN_\rho(\spazio{p,w}, \cB)$ is a Banach space when endowed by the norm
\begin{equation}
\label{Nc.norm}
\left|\und F \right|_{\rho}:=
\sup_{\zeta \in B^{p,w}(\rho)} \| \und{F}(\zeta)\|_{\cB}.
\end{equation}
\end{definition}

A map $F: U \to \cB$ is said to
be  {\sl real majorant
  analytic} on $U$ if for each point $(\xi, \bar \xi) \in U$ there exist a
neighborhood $V$ of $(\xi, \bar \xi)$ in $\spazio{p,w}$ and a   majorant analytic germ which coincides with $F$ on $U
\cap V$.

\begin{remark}
From Cauchy inequality one has that the Taylor polynomials $F^r$ of $F$ 
satisfy
\begin{equation}
\label{exp.4}
\norm{\underline{F}^r(\zeta)}_\cB \leq \left|\und F\right|_{\rho} \frac{\norm{\zeta}^r_{p,w}}{\rho^r} \ ,  \qquad
\forall \zeta \in B^{p,w}(\rho)\ .
\end{equation}
\end{remark}
\begin{remark}
\label{rem.conv}
Since $\forall r \geq 1$ one has 
$\norm{F^r}\leq \norm{\und{F}^r},$
if $F \in \cN_\rho(\spazio{p,w}, \cB)$ then the Taylor series of $F$ is uniformly convergent in $B^{p,w}(\rho)$.
\end{remark}

We will often consider the case $\cB=\spazio{p,w}$; in such a case the basis $\{ \bb_m\}_{m \in \Z}$ coincide with the natural
basis $\be_{2m} :=(e_m,0)$, $\be_{2m+1} := (0,e_m)$ of such a space (where $e_m$ is  the vector in $\C^\Z$ with
all components equal to zero except the $m^{th}$ one which is equal to
$1$). We will consider also the case
$\cB=\cL(\spazio{p,w^1},\spazio{p,w^2})$ (bounded linear operators from
$\spazio{p,w^1}$ to $\spazio{p,w^2}$), where $w^1$ and $w^2$
are weights. Here the chosen basis is $\bb_{jk}=\be_j\otimes \be_k$
(labeled by 2 indexes).

\begin{remark}
\label{Diff}
For $\zeta \equiv (\xi, \eta) \in \spazio{p, w}$, we denote by $|\zeta|$ the vector
of the modulus of the components of $\zeta$: 
$|\zeta| = (|\zeta_j|)_{j \in \Z}$,  $|\zeta_j|:= (|\xi_j|, |\eta_j|) \in \R^2$. If
$F\in\cN_{\rho}(\spazio{p,w},\spazio{p,w})$ then for any $\zeta, \upsilon \in \spazio{p,w} $ one has 
$$
\underline{\di F}(|\zeta|) |\upsilon|\leq
\di \underline F(|\zeta|)|\upsilon|
$$
 (see \cite{kuksinperelman}). Thus  $\forall 0<d<1$, Cauchy estimates imply that $\di F\in
\cN_{(1-d)\rho}(\spazio{p,w}, \cL(\spazio{p,w},\spazio{p,w}))$ with
\begin{equation}
\label{diff.1}
\left|\underline{\di F}\right|_{\rho(1-d)}\leq
\frac{1}{d\rho}\left|{\underline F}\right|_{\rho}\ ,
\end{equation}   
where $\und{ \di F}$ is computed with respect to the basis $\be_j\otimes \be_k$.
\end{remark}
Following Kuksin-Perelman \cite{kuksinperelman} we will need also a further property. 
\begin{definition}
\label{def:arho}
Let $\R \ni \rho >0$ and $\N \ni N \geq 2$. 
 A majorant analytic  germ $F\in\cN_\rho(\spazio{p,w},\spazio{p,w})$ will be said to be of
  class $\cA_{w, \rho}^{N}$ if $F=O(\zeta^N)$ and the map $\zeta  \mapsto \di F(\zeta)^* \in\cN_\rho(\spazio{p,w}, \cL(\spazio{p,w},\spazio{p,w}))$. 
 On $\cA_{w,\rho}^{N}$ we will use the norm
\begin{equation}
\label{nor.arho}
\norm{F}_{\cA_{w,\rho}^{N }}:=\left|\underline F\right|_\rho+{\rho}
\left|\underline{\di F}\right|_\rho+\rho \left|\underline{\di 
  F^*}\right|_\rho .
\end{equation}
\end{definition}


\begin{remark}
\label{rem:norm.in.A}
Assume that for some $ \rho > 0$ the map $F \in \cA_{ w,
  \rho}^{ N}$, $N \geq 2$,  then for every $0 < d \leq \tfrac{1}{2}$ one has
$\abs{\und{F}}_{d\rho} \leq 2 d^N \abs{\und{F}}_{\rho}$ and
$\norm{F}_{\cA_{ w, d\rho}^{N }} \leq 6d^N \norm{F}_{\cA_{ w,
    \rho}^{ N}}$.
\end{remark}
A real majorant analytic germ $F: \Br^{p,w}(\rho)\to \spazior{p,w}$
will be said to be of class $\cN_\rho(\spazior{p,w}, \spazior{p,w})$ (respectively $\mathcal{A}_{ w, \rho}^{N}$) if there
exists a map  of class  $\cN_\rho(\spazio{p,w}, \spazio{p,w})$ (respectively
$\cA_{ w, \rho}^{N}$), which coincides with $F$ on
$\Br^{p,w}(\rho)$, namely on the restriction $\bar\xi_j =\eta_j$, $\forall j \in \Z$.  In this case we will also denote by $\abs{\und{F}}_\rho$ (respectively
$\norm{F}_{\cA_{ w, \rho}^{N }}$) the norm
defined by \eqref{Nc.norm} (respectively \eqref{nor.arho}) of the complex extension of $F$.

\subsection{Tame majorant analytic maps}
We begin with the following definition

\begin{definition}
\label{def.ta} 
Fix $1 \leq p \leq 2$, weights $w^0\leq w^1\leq w^2$ and  $F \in \cN_\rho(\spazio{p, w^0}, \spazio{p, w^0})$.  $F$ is said to be {\em $(p, w^0,w^1, w^2)$-tame majorant analytic} 
 if $\und{F}: B^{p, w^0}(\rho)\cap \spazio{p, w^1} \to \spazio{p,w^2}$ is analytic and 
\begin{equation}
\label{exp.40}
\abs{\und F}_{\rho}^T:= \sup\left\lbrace \frac{\norm{\und F(\zeta)}_{p,w^2}}{\norm{\zeta}_{p,w^1}}  \ 	 \colon \ \zeta \in B^{p, w^0}(\rho) \cap \spazio{p,w^1} \right\rbrace < \infty \ . 
\end{equation}
In such a case we will write
$F\in \cN_{\rho}^T(B^{p,w^0}\cap\spazio{p,w^1}, \spazio{p,w^2})$. We endow such a space with the norm
\begin{equation}
\label{exp.41}
\tame{\und F}_{\rho}:= \abs{\und F}_{\rho}+ \rho \abs{\und F}_{\rho}^T \ ,
\end{equation}
where  here $\abs{\und F}_{\rho} := \sup\left\lbrace \norm{\und F(\zeta)}_{p,w^0}  \ 	 \colon \ 
\norm{\zeta}_{p, w^0} \leq \rho \right\rbrace $.
\end{definition}
\begin{remark}
\label{rem:tame.c}
Let $F \in \cN_{\rho}^T(B^{p,w^0}\cap\spazio{p,w^1}, \spazio{p,w^2})$.
Expand $F$ in Taylor series,  $F = \sum_n F^n$. Then  it follows by Cauchy estimates that each polynomial $F^n$ is $(p, w^0, w^1, w^2)$- tame majorant analytic and 
$$
\norm{\und F^r(\zeta)}_{p,w^2} \leq \frac{\abs{\und F}_{\rho}^T}{\rho^{r-1}} \,  \norm{\zeta}_{p, w^0}^{r-1} \norm{\zeta}_{p,w^1} \ .
$$
Consequently, using also \eqref{exp.4}, one has 
$\tame{\und F^r}_\rho \leq \tame{\und F}_{\rho} \ .$
\end{remark}
 
In case of maps with values in the space $\cL(\spazio{p,w^1}, \spazio{p,w^2})$ we give the following
\begin{definition}
Fix $1 \leq p \leq 2$,  weights $w^0 \leq w^1\leq w^2$ and let $\cG \in \cN_\rho(\spazio{p, w^0}, \, \cL(\spazio{p, w^0}, \spazio{p, w^0}))$. 		\\
 $\cG$ is said to be {\em $(p, w^0,w^1, w^2)$-tame majorant analytic} if $\und{\cG} : B^{p, w^0}(\rho) \cap \spazio{p,w^1} \to \cL(\spazio{p,w^1}, \spazio{p,w^2})$ and 
\begin{equation}
\label{exp.42}
\abs{\und \cG}_{\rho}^T:= \sup \left\lbrace \frac{\norm{\und \cG(\zeta)\upsilon}_{p,w^2}}{\norm{\zeta}_{p,w^1} \norm{\upsilon}_{p, w^0} + \rho\norm{\upsilon}_{p,w^1}} \ \colon \ \zeta, \upsilon \in \spazio{p,w^1}, \, 
\norm{\zeta}_{p,w^0} \leq \rho \right\rbrace 
 < \infty \ . 
\end{equation}
 In such a case we will write
$\cG\in \cN_{\rho}^T(B^{p,w^0}\cap \spazio{p,w^1},\, \cL(\spazio{p,w^1}, \spazio{p,w^2}))$. We endow such a space with the norm
\begin{equation}
\label{exp.5}
\tame{\und \cG}_{\rho}:= \abs{\und \cG}_{\rho} + \rho \abs{\und \cG}_{\rho}^T \ ,
\end{equation}
where here $\abs{\und \cG}_{\rho} := \sup \left\lbrace \norm{\und{\cG}(\zeta)}_{\cL(\spazio{p,w^0})} \ \colon \ \norm{\zeta}_{p,w^0} \leq \rho \right\rbrace$.
\end{definition}

\begin{remark}
\label{cauchy2}
Let $F\in \cN_{\rho}^T(B^{p,w^0}\cap \spazio{p,w^1}, \spazio{p,w^2})$. By Cauchy formula  one has
\begin{equation}
\label{cauchy}
\di F(\zeta) \upsilon = \frac{1}{2\pi 
\im } \oint_{|\lambda|=\epsilon} \frac{F(\zeta + \lambda \upsilon)}{\lambda^2} \di \lambda \ ,
\end{equation}
provided $|\zeta| + \epsilon |\upsilon| \in B^{p,w^0}(\rho)$. It follows that for any $0< d \leq 1/2$, the map 
$\di F\in
\cN_{(1-d)\rho}(B^{p,w^0} \cap \spazio{p,w^1}, \cL(\spazio{p,w^1},\spazio{p,w^2}))$ with
\begin{equation}
\label{diff.2}
\abs{\underline{\di F}}_{\rho(1-d)}^T \leq
\frac{1}{d\rho}\abs{{\underline F}}_{\rho}^T\ .
\end{equation}   
\end{remark}

We extend Definition \ref{def:arho} to deal with tame majorant analytic maps:
\begin{definition}
Let $\R \ni \rho >0$ and $\N \ni N \geq 2$. 
 A map $F \in \cA_{w^0,\rho}^N$ will be said to be of
  class $\sT_{w^0, w^1,\rho}^{w^2, N}$ if $F\in\cN_\rho^T(B^{p,w^0}\cap \spazio{p,w^1}, \spazio{p,w^2})$ 
  and the map 
  $\zeta  \mapsto \di F(\zeta)^* \in
  \cN_\rho^T(B^{p,w^0}\cap\spazio{p,w^1}, \cL(\spazio{p,w^1},\spazio{p,w^2}))$. 
 On $\sT_{ w^0, w^1, \rho}^{w^2, N}$ we will use the norm
\begin{equation}
\label{nor.trho}
\norm{F}_{\sT_{ w^0, w^1, \rho}^{w^2, N }}:= \tame{\underline F}_\rho+{\rho}
\tame{\underline{\di F}}_\rho+\rho \tame{\underline{\di 
  F^*}}_\rho .
\end{equation}
\end{definition}

\begin{remark}
\label{rem:sT.tame}
Let $F \in  \sT_{w^0, w^1,\rho}^{w^2, N}$. Then
$$
\norm{\und F(\zeta)}_{p, w^2} \leq \frac{\norm{F}_{\sT_{ w^0, w^1, \rho}^{w^2, N }}}{\rho} \, \norm{\zeta}_{p, w^1} \ , \qquad
\forall \zeta \in B^{p, w^0}(\rho) \cap \spazio{p,w^1} \ . 
$$
\end{remark}

\begin{remark}
\label{rem:norm.in.T}
Assume that for some $ \rho > 0$ the map $F \in \sT_{ w^0,
w^1,  \rho}^{w^2, N}$, $N \geq 2$,  then for every $0 < d \leq \tfrac{1}{2}$ one has
$\tame{\und{F}}_{d\rho} \leq 2 d^N \tame{\und{F}}_{\rho}$ and
$\norm{F}_{\sT_{ w^0,w^1, d\rho}^{w^2, N }} \leq 6d^N \norm{F}_{\sT_{ w^0,
w^1,    \rho}^{w^2, N}}$.
\end{remark}

\subsection{The tame Kuksin-Perelman theorem}

We are now able to state a tame version of the Kuksin-Perelman theorem.

Fix  $\rho>0$, $1 \leq p \leq 2$ and let $\Psi: \Br^{p, w^0}(\rho) \to \spazior{p,w^0} $.  Write
$\Psi$ component-wise, $\Psi=\left\{(\Psi_j, \overline{\Psi}_j )\right\}_{j\in \Z}$, and
consider the foliation defined by the functions
$\left\{\left|\Psi_j(\xi, \bar \xi)\right|^2/2\right\}_{j\in \Z}$.  Given
$\xi \equiv (\xi, \bar \xi) \in\spazior{p,w^0}$ we define the leaf through $\xi$ by
\begin{equation}
\label{leaf.1}
\cT_\xi:=\left\{(\upsilon, \bar \upsilon) \in \spazior{p,w^0} : \ \frac{|\Psi_j(\upsilon, \bar \upsilon)|^2}{2}=\frac{|\Psi_j(\xi, \bar \xi)|^2}{2}\  
\ ,\ \forall 
j \in \Z \right\}\ .
\end{equation}
Let $\cT=\bigcup_{\xi \in \spazior{p,w^0}} \cT_\xi$ be the collection of all the leaves of the foliation.  
We 
denote by $T_\xi\cT$ the tangent space to $\cT_\xi$ at the point $\xi \in \spazior{p,w^0}$. 
Next we define the function $I = \{ I_j \}_{j\in \Z}$ by
\begin{equation}
\label{Ij}
I_j(\xi)\equiv I_j(\xi, \bar \xi):=\frac{|\xi_j|^2}{2} \qquad \forall j \in \Z \ .
\end{equation}
The foliation they define will be denoted by $\cT^{(0)}$. 
\begin{remark}
$\Psi$ maps the foliation $\cT$ into the foliation $\cT^{(0)}$, namely $\cT^{(0)} = \Psi(\cT)$.
\end{remark}
We state now the tame  Kuksin-Perelman theorem:
\begin{theorem}[Tame Kuksin-Perelman theorem]
\label{KP} 
Let $1 \leq p \leq 2$ be real.
 Let $w^0, w^1$ and $w^2$ be weights with $w^0 \leq w^1 \leq w^2$. Consider the
 space $\spazior{p,w^0}$ endowed with the symplectic form $\omega_0$
 defined in \eqref{scalar_productR}. Let $\rho >0$ and assume $\Psi:
 \Br^{p,w^0}(\rho) \to \spazior{p,w^0}$, $\Psi= \uno + \Psi^0$ and $\Psi^0
 \in \sT_{w^0, w^1, \rho}^{w^2, N}$, $\N \ni N \geq 2$. Define
\begin{equation}
\label{th}
\epsilon_1:=\norm{\Psi^0}_{\sT_{ w^0, w^1, \rho}^{w^2,  N}}\ .
\end{equation}
Assume that the functionals $\lbrace
\frac{1}{2}\left|\Psi_j(\xi, \bar \xi)\right|^2\rbrace_{ j\in \Z}$ pairwise
commute with respect to the symplectic form $\omega_0$, and that
$\rho$ is so small that
\begin{equation}
\epsilon_1<2^{-60}\rho. 
\label{th.1}
\end{equation}
Then there exists a real majorant analytic map
$\widetilde{\Psi}:\Br^{{p,w^0}}(\ta \rho)\to \spazior {p,w^0}$, $\ta=2^{-120}$, 
with the following properties:
\begin{itemize}
\item[i)] $\widetilde{\Psi}^*\omega_0=\omega_0$, so that the coordinates
  $(z, \bar z) :=\widetilde\Psi(\xi, \bar \xi) $ are canonical;
\item[ii)] the functionals $\left\{ \frac{1}{2}\left|\widetilde\Psi_j(\xi, \bar \xi)\right|^2\right\}_{ j \in \Z }$ pairwise commute with respect to the symplectic form $\omega_0$;
\item[iii)] $\cT^{(0)}=\widetilde{\Psi}(\cT),$ namely the foliation
  defined by $\Psi$ coincides with the foliation defined by
  $\widetilde{\Psi}$;
\item[iv)] $\widetilde{\Psi} = \uno + \widetilde{\Psi}^0$ with 
$\widetilde{\Psi}^0 \in\sT_{w^0, w^1, \rho}^{w^2, N} $ and furthermore $\norm{\widetilde{\Psi}^0}_{\sT_{w^0, w^1, \ta\rho}^{w^2, N}} \leq 2^{17} \epsilon_1 $. 
\item[v)] The inverse map $\widetilde\Psi^{-1} $ is real majorant analytic  
$B^{p, w^0}_r(\widetilde\ta  \rho) \to \spazior{p,w^0}$, $\widetilde \ta = 2^{-130}$, $\widetilde\Psi^{-1}- \uno  \in \sT_{w^0, w^1, \widetilde\ta\rho}^{w^2, N}$ with the quantitative estimate  
$\norm{\widetilde{\Psi}^{-1} - \uno}_{\sT_{w^0, w^1, \widetilde\ta\rho}^{w^2, N}} \leq 2^{18} \epsilon_1 $. 
\end{itemize}
Finally the theorem holds true also if the class $ \sT_{w^0, w^1, \rho}^{w^2, N}$ is replaced by the class $\cA_{w^0, \rho}^N$.
\label{teokuksinperelman}
\end{theorem}

\noindent The novelty of Theorem  \ref{KP} is to prove that  the Birkhoff map is {\em tame majorant analytic}, provided  the initial map $\Psi - \uno $ is tame majorant analytic.
The proof of Theorem \ref{KP} is actually a variant of the results of \cite{kuksinperelman,masp_toda}, so we postpone it to Appendix \ref{app:KP}.

The following corollary  is an immediate application of Theorem \ref{KP} and shows that $\wt\Psi$ is a Birkhoff map:
\begin{corollary}
\label{cor.KP}
Let $H: \spazior{p, w^1} \to \R$ be a real analytic Hamiltonian
function. Let $\Psi$  be as in Theorem \ref{KP} and assume that for
every $j \in \Z$, $\abs{\Psi_j(\xi, \bar \xi)}^2$ is an integral of motion for
$H$, i.e.
\begin{equation}
\label{ass.1}
\lbrace H, |\Psi_j|^2 \rbrace = 0 \qquad \forall\,j \in \Z \ .
\end{equation}
Then the coordinates $(z_j,\bar z_j)$ defined by $(z_j, \bar z_j):=\widetilde
\Psi_j(\xi, \bar \xi)$ are complex Birkhoff coordinates for $H$, namely canonical
conjugated coordinates in which the Hamiltonian depends only on
$|z_j|^2/2 $.
\end{corollary}
\begin{proof}
By assumption, $\Psi$ is analytic as a map $B^{p,w^0}(\rho) \cap \spazio{p, w^1} \to \spazio{p,w^1}$, therefore the composition $H\circ \Phi^{-1}$ is well defined and real analytic as  a map $B^{p,w^0}_r(\rho) \cap \spazior{p, w^1} \to \R$.
Thus it admits a convergent Taylor expansion of the form 
\begin{equation}
\label{h.tay}
\left(H\circ  \Phi^{-1}\right)(z, \bar z) = \sum_{r \geq 2 \atop |\alpha| + |\beta| = r} H^{r}_{\alpha, \beta} \, z^\alpha \, \bar z^\beta \ .
\end{equation}
Arguing as in  \cite[Corollary 2.13]{masp_toda} one shows that \eqref{ass.1} implies that in each monomial of the r.h.s. of \eqref{h.tay}, one has  $\alpha = \beta$.
\end{proof}

\section{Application to dNLS on $\T$}
\label{sec:app}
In order to construct a tame Birkhoff map for dNLS, we wish to apply the Kuksin-Perelman theorem \ref{KP}. This requires to be able to construct the starting map $\Psi$ and to verify that such a map is tame majorant analytic.
To construct  $\Psi$ we will exploit  the integrable structure of dNLS,  following the ideas already employed in the case of  the KdV \cite{kuksinperelman} and the Toda lattice \cite{masp_toda} (see also  \cite{battig}).
To prove tame majorant analyticity,  we will expand $\Psi$ in Taylor series $\Psi = \sum_{n \in \N} \Psi^n$, compute each polynomial $\Psi^n$  and prove that it belongs to $\sT_{w^0, w^1, \rho}^{w^2, n}$  for some   weights $w^0 \leq w^1 \leq w^2$. 
We are able to state sufficient conditions for the choice of weights  $w^0 \leq w^1 \leq w^2$ to use. Such conditions depend only on some  arithmetic property that we state now.
To do so,   we need a bit of preparation.\\ 
Given $n \geq 3$ odd, $k_1, \ldots, k_n \in \Z$, we define  the function $\ff_n :\Z^n \times \Z \to \R$ by
\begin{equation}
\label{def:ff}
\ff_n(k_1, \ldots, k_n; j) :=  \uno_{\{k_1 + \ldots + k_n  = j\}}\,\prod_{m=1\atop m\,  {\rm odd}}^{n-1} \frac{1}{\la k_1 + \ldots + k_m - j\ra } \, \cdot \, \frac{1}{\la k_1 + \ldots + k_{m+1} \ra } \ ,
\end{equation}
where here $\uno_S$ is the indicator function on the set $S$.
Given  an integer  $1 \leq r \leq n$ we define $\fg_{n,r}:\Z^n \times \Z \to \R$ by
\begin{equation}
\label{def:gg}
 \fg_{n,r}(k_1, \ldots, k_n ; j) := \ff_n(k_1, \ldots, k_{r-1}, j , k_{r+1}, \ldots, k_n ; k_r) \ . 
\end{equation}
Note that $\fg_{n,r}$  is  $\ff_n$ with   indexes $k_r$ and $j$ switched. 
Its explicit expression  is given in Appendix \ref{app:weight}.

The key point, as we shall see below,  is that $\ff_n$ bounds the kernel of the polynomial $\und{\Psi^n}$, while the $\fg_{n,r}$'s bound the kernel of $[\und{\di \Psi^n}]^*$.  For example, it turns out that  boundedness of $\und{\Psi^n}$ and $[\und{\di \Psi^n}]^*$ as maps $B^{p}(\rho)\to \spazio{p}$ respectively $B^{p}(\rho)\to \cL(\spazio{p})$ are implied by the following  summability properties of  $\ff_n$
and $\fg_{n,r}$:
\begin{lemma}
\label{f.g.lp}
 Let $1 \leq p \leq 2$ be real. Let   $p'$ s.t.  $\frac{1}{p} + \frac{1}{p'} = 1$ and define
 $$
R_* := \left(\sum_{k \in \Z} \frac{1}{\la k \ra^{p'}}\right)^{1/p'} \ .
$$
	Then  for every $n \geq 3$, $n$ odd
	\begin{align}
	\label{w10}
\sup_{j\in \Z}  \  \, \norm{ \ff_n(\cdot; j)  }_{\ell^{p'}(\Z^n)} \leq   R_*^{n-1} \ , 
	\end{align}
	and 
\begin{equation}
\label{w20}
\max_{1 \leq r \leq n} \ 
\sup_{j\in \Z} \  
\norm{\fg_{n,r}(\cdot ; j)}_{\ell^{p'}(\Z^n)}
 \leq  R_*^{n-1} \ .
\end{equation}
\end{lemma} 
The lemma is proved in  Appendix \ref{app:weight.1}.\\

In  a similar way, we will show below (see  Lemma \ref{lem:psi00} and Lemma \ref{lem:dPsi.1})  that the maps $\Psi^n$  and $[\di \Psi^n]^*$ are $(p, \tu, \tv, \tw)$-tame majorant analytic if the weights $\tu \leq \tv \leq \tw$  fulfill the following  property:
	\begin{itemize}
	\item[(W$)_p$] Let $\tu \leq \tv \leq \tw$ be symmetric and sub-multiplicative weights. Let $1 \leq p \leq 2$ and   $\frac{1}{p} + \frac{1}{p'} = 1$.   
	There exist $R_0>0, R_1\geq R_*$,   s.t. for every $n \geq 3$, $n$ odd
	\begin{align}
	\label{w1}
\sup_{j\in \Z}  \ \tw_j \, \norm{ \frac{\ff_n(\cdot; j)}{\sum_{l=1}^n \t v_{k_l} \ \prod_{m \neq l} \tu_{k_m}}  }_{\ell^{p'}(\Z^n)} \leq  R_0 \, R_1^{n-1} \ , 
	\end{align}
	and 
\begin{equation}
\label{w2}
\max_{1 \leq r \leq n} \ 
\sup_{j\in \Z} \  \tw_j \,
\norm{\frac{\fg_{n,r}(\cdot ; j)}{\sum_{l=1}^n \t v_{k_l}  \ \prod_{m \neq l} \tu_{k_m} }}_{\ell^{p'}(\Z^n)}
 \leq  R_0 \, R_1^{n-1} \ .
\end{equation}	
	\end{itemize}
	
	We give some examples of weights fulfilling condition {\rm (W$)_p$}:
%
%
%
%
\begin{proposition}
\label{prop:sw}
Let  $1 \leq p \leq 2$. Then the following holds true:
\begin{itemize}
\item[(i)] For any $s \geq 0$, $a\geq 0$ and $0 < b \leq 1$, the weights $\tu=\tv = \tw = \{ \la j\ra^s\, e^{a |j|^b }\}_{j \in \Z}$  fulfill  {\rm (W$)_p$} with   constants
$$R_0 = 1 \ , \qquad  R_1 = \left(\sum_{k \in \Z} \frac{1}{\la k\ra^{p'} }\right)^{1/p'}  \ . $$
\item[(ii)] For any  $s \geq 1 $, the weights 
$\tu = \{1\}_{j \in \Z}$, 
$\tv =\tw = \{\la j \ra^{s}\}_{j \in \Z}$ fulfill  {\rm (W$)_p$} with   constants 
$$R_0 = 1 \ , \qquad   R_1 = 2^s \left(\sum_{k \in \Z} \frac{1}{\la k\ra^{p'} }\right)^{1/p'} \ . $$
\item[(iii)] For any  $s \geq 1$, the weights 
$\tu = \{\la j \ra\}_{j \in \Z}$, 
$\tv = \{\la j \ra^{s} \}_{j \in \Z}$ and $\tw = \{\la j \ra^{s+1}\}_{j \in \Z}$ fulfill  {\rm (W$)_p$} with   constants 
$$R_0 = 1 \ , \qquad   R_1 = 2^{s+2} \left(\sum_{k \in \Z} \frac{1}{\la k\ra^{p'} }\right)^{1/p'} \ . $$
\end{itemize}
\end{proposition}

The proof of the Proposition is postponed in Appendix \ref{app:weight}.	\\

The main result of this section  is the following theorem:
\begin{theorem}
\label{thm:psi}
Fix $1 \leq p \leq 2$. There exist $\tC,  \varrho_* >0$ and an   analytic map  $\Psi: B^{p}(\varrho_*) \to \spazio{p}$ s.t. the following is true: 
\begin{itemize}
\item[(i)] The quantities   $\abs{\Psi_j(\xi, \bar \xi)}^2$ are in involution $\forall (\xi, \bar \xi) \in B^{p}_r(\varrho_*)$; more precisely   $\abs{\Psi_j(\xi, \bar \xi)}^2 = \gamma_j^2(\varphi, \bar \varphi)$,   where $\gamma_j$ is the $j^{th}$ spectral gap (see \eqref{sp.gap}) and $\varphi = \cF^{-1}(\xi)$.  
\item[(ii)]  For any  $\tu \leq \tv \leq \tw$   weights fulfilling {\rm (W$)_p$} and 
\begin{equation}
\label{rho.size}
0 < \rho  \leq \min\left( \varrho_*, \ \frac{\varrho_*}{R_1} \right)
\end{equation}
the restriction of $\Psi$ to $B^{p,\tu}(\rho)$ is analytic as a map    $B^{p,\tu}(\rho)  \to \spazio{p,\tu} $; its nonlinear part $\Psi - \uno$ is $(p, \tu, \tv, \tw)$-tame majorant analytic, $\Psi- \uno \in \sT_{ \tu, \tv, \rho}^{\tw, 3}$ with the quantitative estimate
 \begin{equation}
 \label{normPsi}
 \norm{\Psi- \uno}_{\sT_{\tu,  \tv, \rho}^{\tw, 3}} \leq   \tC  \, R_0 \, R_1^2  \, \rho^3 \ .
 \end{equation}
 \item[(iii)] $\Psi$ is real for real sequences, i.e. $\Psi\colon B^{p}_r(\varrho_*) \to \spazior{p}$.
 \end{itemize}
\end{theorem}	

	\begin{remark}
The constants $\tC, \varrho_*$	in Theorem \ref{thm:psi} {\em do not} depend on the regularity of $\zeta$, and it is possible to compute them, see Section \ref{subsec.proof}.
\end{remark}
Before proving Theorem \ref{thm:psi}, we show how  Theorem \ref{thm:psi} and  the Kuksin-Perelman theorem \ref{KP} imply Theorem \ref{main}.\\

\begin{proof}[Proof of Theorem \ref{main}]
Let $1 \leq p \leq 2$. First take $\tu = \tv = \tw = \{ 1 \}_{j \in \Z}$.  By Proposition \ref{prop:sw} such weights fulfill (W$)_p$ with $R_0 = 1$ and $R_1 = \left(\sum_{k \in \Z} \frac{1}{\la k\ra^{p'} }\right)^{1/p'}\leq 2$.  
For $\rho \leq 2^{-1} \varrho_*$, 
$\Psi- \uno \in \sT_{\tu, \tv, \rho}^{\tw, 3}$ with
$$
\epsilon_1 := \norm{\Psi - \uno}_{\sT_{\tu, \tv, \rho}^{\tw, 3}} \leq 4 \tC  \rho^3  \ . 
$$ 
Hence condition \eqref{th.1} is satisfied if
$\rho \leq \min(2^{-31} \tC^{-1/2}, 2^{-1} \varrho_*) =: \varrho_1$.
Applying the tame Kuksin-Perelman theorem,   there exists a 
map $\wt\Psi: B^{p}_r(\ta \varrho_1) \to \ell^{p}_r$,
 $\ta = 2^{-120}$,  which fulfills $i)$--$v)$ of Theorem \ref{KP}. 
In particular $\wt \Psi - \uno \in \cA_{\tu, \ta \rho}^{ 3}$ for any $\rho \leq \varrho_1$ with 
$$
\norm{\wt \Psi - \uno}_{\cA_{\tu , \ta \rho}^{ 3}} \leq
2^{17} \epsilon_1 \leq  2^{19}\, \tC \, \rho^3 \ . 
$$
Denote now  $\Phi := \left.\wt\Psi\right|_{B^{p}_r(\ta \rho)}$.
Such map is majorant analytic as a map $B^{p}(\rho_0) \to \spazio{p}$ for $\rho_0 \equiv  \ta  \varrho_1 $, and fulfills $(i)$ and $(ii)$ of Theorem \ref{main}.\\
We prove now $(iii)$. Take $\tu= \{ 1 \}_{j \in \Z}$, $\tv = \tw =  \{\la j \ra^{s}\}_{j \in \Z}$  with $s \geq 1$.
 By Proposition \ref{prop:sw}(ii) these weights fulfill  {\rm (W$)_p$} with   constants $R_0 = 1$ and $R_1 = 2^s \left(\sum_{k \in \Z} \frac{1}{\la k\ra^{p'} }\right)^{1/p'}\leq 2^{s+1}$.
Therefore for $0 < \rho < \varrho_* 2^{-1 - s}$, one has that
$\Psi - \uno \in \sT_{\tu, \tv, \rho}^{\tw, 3}$ and 
$\epsilon_1\equiv \norm{\Psi - \uno}_{\sT_{\tu, \tv, \rho}^{\tw, 3}} \leq \tC 2^{2s+2} \rho^3$. Thus, condition  \eqref{th.1} is fulfilled provided   $\rho <\min(2^{-31-s} \tC^{-1/2}, 2^{-1} \varrho_*)$, so one applies the Kuksin Perelman  theorem obtaining 
that $\wt \Psi - \uno \in 
\sT_{\tu, \tv, \ta \rho}^{\tw, 3}$. Then with $\rho' = \ta \rho$, 
$$
\norm{\Phi - \uno}_{\sT_{\tu , \tv, \rho'}^{ \tw, 3}} \equiv 
\norm{\wt \Psi - \uno}_{\sT_{\tu , \tv,  \ta\rho}^{\tw, 3}} \leq 
\tC^\prime 4^{s} \rho'^3  \ .
$$
This estimate and Remark  \ref{rem:sT.tame} implies $(iii)$.
Item $(v)$ is proved analogously using the weights 
$\tu = \{\la j \ra\}_{j \in \Z}$, 
$\tv = \{\la j \ra^{s} \}_{j \in \Z}$ and $\tw = \{\la j \ra^{s+1}\}_{j \in \Z}$.

Item $(iii)$ is a consequence of   Corollary \ref{cor.KP}. The explicit form of the mass and the momentum in Birkhoff coordinates follows by the result of \cite{grebert_kappeler} and the  remark that,  despite the integrating Birkhoff map is not necessarily unique, it is unique the normal form.
\end{proof}

\begin{proof}[Proof of Theorem \ref{main2} and Corollary \ref{cor.KP}]
Theorem \ref{main2}  follows with the same arguments employed in the proof of Theorem \ref{main},  using  that the weights  $\tu = \tv = \tw =  \{\la j \ra^{s} \, e^{a |j|}\}_{j \in \Z}$ fulfill  (W$)_p$ with constants $R_0$ and $R_1$ which {\em do not} depend on $s,a$.\\
The proof of Corollary \ref{cor.KP} follows as in  \cite[Corollary 1.6]{masp_toda}.
\end{proof}

\subsection{Proof of Theorem \ref{thm:psi}}
\label{subsec.proof}
As we already mentioned, in  order to construct the map $\Psi$
we will exploit the integrable structure of the dNLS, which we now recall.
It is well known that   dNLS on $\T$ admits a Lax pair formulation, where the Lax operator $L$ is the Zakharov-Shabat differential operator given for any $ (\varphi_1, \varphi_2) \in L^2_c $ by
\begin{equation}
\label{lax}
L(\varphi_1, \varphi_2) = 
\im 
\begin{pmatrix}
1 & 0 \\
0 & -1
\end{pmatrix} \partial_x + 
\begin{pmatrix}
0 & \varphi_1 \\
 \varphi_2 & 0
\end{pmatrix}  \ . 
\end{equation}
We consider \eqref{lax} as  an operator on the space 
 $$ \cY := L^2(\R/2\Z, \C) \times L^2(\R/2\Z, \C) $$
of functions  with periodic boundary conditions on the interval $[0, 2]$ (twice the  periodicity of $\varphi_1, \varphi_2$).
Often we will denoted by $ \begin{pmatrix} f_1 \\ f_2 \end{pmatrix}$ the elements of $\cY$. The space $\cY$ is  equipped with the complex scalar product
\begin{equation}
\label{scalar_productcY}
\left( \begin{pmatrix} u_1 \\ u_2 \end{pmatrix} \ , \begin{pmatrix} v_1 \\ v_2 \end{pmatrix} \right)_\cY := \int_0^2  (u_1(x)\, \bar{v}_1(x) + u_2(x)\bar v_2 (x)) \, dx \ .  
\end{equation}
The standard theory of Lax pairs guarantees that the eigenvalues of \eqref{lax}  are infinitely many commuting constants of motion.
More precisely it is well known \cite{grebert_kappeler} that there exists $U \subset L^2_c$ a complex neighborhood of $L^2_r$ in $L^2_c$ s.t. $\forall (\varphi_1, \varphi_2) \in U$ the spectrum of \eqref{lax} is given by a sequence of complex numbers (lexicographycally ordered) 
$$
\cdots \sleq \lambda_0^-(\varphi_1, \varphi_2) \sleq \lambda_0^+(\varphi_1, \varphi_2)
 \sleq
  \lambda_1^-(\varphi_1, \varphi_2) \sleq 
  \lambda_1^+(\varphi_1, \varphi_2) < \cdots
$$
The $\{\lambda_j^\pm(\varphi_1, \varphi_2)\}_{j \in \Z}$ are not analytic as functions of $(\varphi_1, \varphi_2)$,  thus one prefers to use, rather than the eigenvalues, the {\em spectral gaps}
\begin{equation}
\label{sp.gap}
\gamma_j(\varphi_1, \varphi_2) := \lambda_j^+(\varphi_1, \varphi_2) - \lambda_j^-(\varphi_1, \varphi_2)  \ , \qquad 
j \in \Z
\end{equation}
which are known to be  real analytic commuting constants of motion.
As we already mentioned,  the map $\Psi$ that we will construct has the property that, for {\em real} $\zeta$, $\abs{\Psi_j(\xi, \bar \xi)}^2 = \gamma_j^2(\varphi, \bar\varphi)$ $\, \forall j \in \Z$, where $\xi = \cF(\varphi)$.  \\

Before starting the construction of $\Psi$,  it is useful to state some properties of the Lax operator which will be used in the following. First decompose $L$ as the sum $L = L_0 + V$, where 
\begin{equation}
\label{L0}
L_0 := \im 
\begin{pmatrix}
1 & 0 \\
0 & -1
\end{pmatrix} \partial_x \ , \qquad V(\varphi_1, \varphi_2) := \begin{pmatrix}
0 & \varphi_1 \\
 \varphi_2 & 0
\end{pmatrix} \ . 
\end{equation}

By \eqref{fcvarphi}, we identify $(\varphi_1, \varphi_2) \in L^2_c$ with  $\zeta = (\xi, \eta) \in \spazio{2}$, thus from now on we will write $ V(\zeta) \equiv V(\varphi_1, \varphi_2) $.
		The following properties are trivially verified:
	\begin{itemize}
		\item[(H1)] {\bf Involution $\imath$}:  denote by $\imath$ the bounded, antilinear  operator $\cY \to \cY$ defined by
		$$
		\imath \begin{pmatrix} u_1 \\ u_2 \end{pmatrix} = \begin{pmatrix} \bar u_2 \\  \bar u_1 \end{pmatrix} \ .
		$$
		Then $\forall f, g \in \cY$, $\forall \lambda \in \C$ one has
		\begin{align*}
\imath^2 f = f \ , \qquad 		
		\bar{(f, g)}_\cY = (\imath f, \imath g)_\cY \ ,  \qquad 
		\imath (\lambda f) = \bar \lambda \  \imath f \ . 
		\end{align*}
		Furthermore 
		\begin{equation}
	        \label{i.comm} 
		\imath \, L_0 = L_0 \, \imath \ , \quad \imath \, V(\zeta) = V(\zeta)^* \, \imath \ .
		\end{equation}	

	\item[(H2)] {\bf Spectrum of $L_0$}: $L_0$ is a selfadjoint  operator with domain $D(L_0)$  dense in $\cY$. Its  spectrum is discrete,  $\sigma(L_0) = \left\{ \lambda_{j}^0 \right\}_{j \in \Z}$,    and each  eigenvalue $\lambda_j^0 \equiv  \pi j$ has  multiplicity 2. Remark that 
	\begin{equation}
	\label{spectrum}
	 \inf_{i\neq j} \abs{\lambda_{j}^0 - \lambda_{i}^0} = \pi  \ . 
	\end{equation}
		\end{itemize}

\begin{itemize}
	\item[(H3)] {\bf Eigenfunctions of $L_0$}: 	for any  $j \in \Z$ we denote by   $f_{j0}^+, f_{j0}^- \in \cY$  the eigenfunctions corresponding to the eigenvalue $\lambda_j^0$ given by
	$$
	f_{j0}^- := \frac{1}{\sqrt 2} \, \begin{pmatrix} 0 \\ e^{\im  \pi j x} \end{pmatrix} \ , 
	\qquad f_{j0}^+ := \frac{1}{\sqrt 2} \, \begin{pmatrix}  e^{-\im  \pi j x}  \\ 0\end{pmatrix} \ . 
	$$
	They  	 fulfill
	\begin{equation}
	\label{i.eig} 
		\imath f_{j0}^- = f_{j0}^+ \ , \qquad \imath f_{j0}^+ = f_{j0}^- \ . 
	\end{equation}
	We denote by $E_{j0}:= {\rm Vect}(f_{j0}^-, f_{j0}^+)$ the vectorial space  spanned by $f_{j0}^\pm$.		\\
	 The vectors $\{f_{j0}^\sigma\}_{j \in  \Z, \sigma \in \pm}$ form a   basis for  $\cY$.
	\end{itemize}
	
%
	
	\begin{itemize}
	\item[(H4)] {\bf Perturbation $V(\zeta)$:} for any $\zeta \in \spazio{2}$,  the operator $V(\zeta)$ has domain $D(V(\zeta)) \supset D(L_0)$.
	\begin{itemize}
	\item[(H4a)] for any {\em real} $\zeta \in\spazior{2}$,   the operator $V(\zeta)$  is symmetric (w.r.t. the scalar product \eqref{scalar_productcY}) on its domain.
	\item[(H4b)] for any $i_1, i_2 \in \Z$ we have that 
	\begin{align*}
	& (V(\zeta) \, f_{i_1 0 }^{-} \ , f_{i_2 0}^{+} )_\cY  = \xi_j  \ , \quad \mbox{ if } i_1 + i_2 = 2j \\
	& (V(\zeta) \, f_{i_1 0 }^{+} \ , f_{i_2 0}^{- })_\cY  = \eta_j \ , \quad \mbox{ if } i_1 + i_2 = 2j \\
        & (V(\zeta) \, f_{i_1 0 }^{\sigma_1} \ , f_{i_2 0}^{\sigma_2} )_\cY  = 0 \ , \quad \mbox{ otherwise }
	\end{align*}
	\end{itemize}	
\end{itemize}

Now take  $(\xi, \eta) \in \spazio{2}$ with a sufficiently small norm. We will  construct  perturbatively the spectral data of    $L(\zeta)$ (defined in  \eqref{lax}) starting from the spectral data of $L_0$  (defined in \eqref{L0}) which is given in  (H2) and (H3). 
We start with a preliminary result:
\begin{lemma}
\label{key.pert}
For any $\lambda \in \C \setminus \sigma(L_0)$
the map $\zeta \mapsto V(\zeta) (L_0 - \lambda)^{-1}$ is analytic as a map 
$\spazio{2} \to \cL(\cY) $ and  fulfills
\begin{equation}
\label{pert.res0}
 \norm{V(\zeta) \, (L_0 - \lambda)^{-1}}_{\cL(\cY)} \leq 2 \, c(\lambda) \, \norm{\zeta}_2  \ , 
\end{equation}
where
$$
c(\lambda) := \left(\sum_{n \in \Z} \frac{1}{|\pi n - \lambda|^2}\right)^{1/2} \ . 
$$
\end{lemma}
\begin{proof}
The map $\zeta \mapsto V(\zeta) (L_0 - \lambda)^{-1}$ is a $\C$-linear map, so it is analytic iff it is bounded. To compute its norm, take $f \in \cY $ and write 
$$
f = \left( f_1, f_2 \right) = \frac{1}{\sqrt 2} \, \left( \sum_{k \in \Z } \alpha_k \, e^{-\im \pi k x} \ , \ \ \sum_k \beta_k \, e^{\im \pi k x} \right) \ ,
$$
so that  
$$
(L_0 - \lambda)^{-1} f \equiv (\wt f_1, \wt f_2) = \frac{1}{\sqrt 2} \, \left( \sum_{k} \frac{\alpha_k}{\pi k - \lambda} \,  e^{-\im \pi j x} \ , \ \ \sum_k \frac{\beta_k}{\pi k - \lambda} \, e^{ \im \pi k x} \right) \ .
$$
Now remark that $\wt f_1, \wt f_2 \in L^\infty[0,2]$ and
$$
\norm{\wt f_1}_{L^\infty[0,2]} \leq \frac{1}{\sqrt 2}\left( \sum_k \abs{\alpha_k}^2 \right)^{1/2} \, \left(\sum_k \frac{1}{\abs{\pi k - \lambda}^2}\right)^{1/2} \equiv c(\lambda) \, \norm{f_1}_{L^2[0,2]}  \ 
$$
(clearly  the same bound holds also for $\wt f_2$). 
Thus 
\begin{align*}
\norm{V(\zeta) \, (L_0 - \lambda)^{-1} f}_{\cY} & = \left(\norm{\varphi_2 \wt f_1}_{L^2[0,2]}^2 + \norm{\varphi_1 \wt f_2}_{L^2[0,2]}^2\right)^{1/2}  \\
&\leq \sqrt{2} \norm{(\varphi_1, \varphi_2)}_{L^2_c} \left(  \norm{\wt f_1}_{L^\infty[0,2]}^2+ 
 \norm{\wt f_2}_{L^\infty[0,2]}^2\right)^{1/2}
  \leq 2\, c(\lambda) \,  \norm{\zeta}_2 \norm{f}_{\cY}
\end{align*}
which is the claimed inequality.
\end{proof}

Now it is sufficient to apply classical Kato perturbation theory \cite{kato} to get the following:
\begin{lemma}
\label{lem:P.ana}
Let 
 $0 < \rho < \frac{1}{8}  \ . $
Then  for any $\norm{\zeta}_2 \leq \rho$   the following holds true: 
\begin{itemize}
\item[(i)] For any $j \in \Z$, let 
\begin{equation}
	\label{gamma.def}
	 \Gamma_j := \left\{\lambda \in \C : \ |\lambda - \lambda_j^0| = \pi/2 \right\} \ . 
	\end{equation}
Then $\Gamma_j \subset \rho(L(\zeta))$, the resolvent set of $L(\zeta)$.
\item[(ii)] For any $j \in \Z$, define the projector $P_j(\zeta)$ and the subspace $E_j(\zeta) \subseteq \cY$ as 
\begin{equation}
\label{proj.def}
P_j(\zeta) := - \frac{1}{2\pi \im } \oint_{\Gamma_j} (L(\zeta) - \lambda)^{-1} \, \di \lambda \ , \qquad E_j(\zeta) := {\rm Ran }\, P_j(\zeta)  \ , \quad j \in \Z
\end{equation}
where $\Gamma_j$ is counterclockwise oriented. Then $\zeta \mapsto P_j(\zeta)$ is analytic as a map  $B^2(\rho)  \to \cL(\cY)$. 
\item[(iii)] For any $j \in \Z$,  define the transformation operator $U_j(\zeta)$ as 
$$
U_j(\zeta) := (\uno - (P_j(\zeta) - P_{j0})^2)^{-1/2} P_j(\zeta) \ , 
$$	
where $P_{j0} := P_j(0)$. Then $\zeta \mapsto U_j(\zeta)$ is analytic as a map $B^2(\rho) \to \cY$  and ${\rm Ran } \, U_j(\zeta) \equiv E_j(\zeta)$.\\
 Furthermore for $\zeta$ {\em real}
	\begin{align}
	&\norm{U_j(\zeta) f}_\cY = \norm{f}_\cY \ , \qquad \forall f \in E_{j}^0 \ , \\
	\label{U.comm}
	&[\imath, U_j(\zeta)] = 0 \ . 
	\end{align}

\end{itemize}
\end{lemma}
\begin{proof}
By Lemma \ref{key.pert}, for any $\lambda \in \Gamma_j$, the map $\zeta \mapsto V(\zeta) (L_0 - \lambda)^{-1}$ is analytic as  a map $\spazio{2} \to \cL(\cY)$ and 
\begin{equation}
\label{VL0}
\sup_{\lambda \in \Gamma_j} \norm{V(\zeta) \, (L_0 - \lambda)^{-1}}_{\cL(\cY)} \leq  \sup_{\lambda \in \Gamma_j }2 \, c(\lambda) \, \norm{\zeta}_2 \leq 4\, \norm{\zeta}_2
\end{equation}
where in the last step we used the explicit formula for $\Gamma_j$ to estimate $c(\lambda)$.
It follows that for 
 $\norm{\zeta}_2 \leq \rho < \frac{1}{8}$, the perturbed resolvent $L(\zeta) - \lambda$ is well defined by Neumann series  and fulfills the estimate
\begin{equation}
\label{pert.res}
\sup_{\lambda \in \Gamma_j} \norm{(L(\zeta) - \lambda)^{-1}}_{\cL(\cY)} \leq \frac{4}{\pi} \ . 
\end{equation}
Thus for $\norm{\zeta}_2 \leq \rho < 1/8 $ the projector $P_j(\zeta)$ in \eqref{proj.def} is well defined and  analytic $\forall j \in \Z$. By the resolvent identity
$$
P_j(\zeta) - P_j(0) =  \frac{1}{2 \pi \im } \oint_{\Gamma_j} \left( L(\zeta) - \lambda\right)^{-1} \, V(\zeta) \, \left( L_0 - \lambda\right)^{-1} \, \di \lambda  \ , \qquad \forall j \in \Z 
$$
which together with \eqref{VL0} and \eqref{pert.res} gives the estimate
$$
\norm{P_j(\zeta) - P_j(0) }_{\cL(\cY)} \leq 8 \norm{\zeta}_2 < 1
$$
provided $\norm{\zeta}_2 \leq \rho < 1/8$. Hence also  $U_j(\zeta)$ is defined by  Neumann series.

We prove now \eqref{U.comm}. First we claim that $\imath P_j(\zeta) = P_j(\zeta)^* \imath$ for every $\zeta$ sufficiently small. 
This follows by a direct computation using that by    (H2), (H4a)
$\imath (L- \lambda)^{-1} = [(L- \lambda)^{-1}]^* \imath$.
Since for $\zeta$  real $L(\zeta)$ is self-adjoint, one has  $P_j(\zeta)^* = P_j(\zeta)$, hence $[\imath,  P_j(\zeta)] =0$.  \eqref{U.comm} follows easily.
\end{proof}
	For  any $j \in \Z$ let us set now 
	\begin{align}
	\label{def.f}
	f_j^\pm(\zeta) : = U_j(\zeta) f_{j0}^\pm \in E_{j}(\zeta)  \ . 
	\end{align}
	Remark that the  $f_j^\pm(\zeta)$'s do not need to be eigenvectors, but they span the eigenspace $E_j(\zeta)$ and are analytic as functions of $\zeta$.\\
		Finally for any $j \in \Z$ let us define
	\begin{align}
	\label{def.z}
	z_j(\zeta) := \left( (L(\zeta)-\lambda_{j}^0) \, f_j^-(\zeta) \ , \imath f_j^-(\zeta) \right)_\cY \ , \quad w_j(\zeta) := \left( (L(\zeta)-\lambda_{j}^0) \, f_j^+(\zeta) \ , \imath f_j^+(\zeta) \right)_\cY \ . 
	\end{align}
  Such coordinates fulfill the following properties:
	\begin{lemma}
	\label{lem:z.ana}
For any $\rho < \frac{1}{8}$ the following holds true:
\begin{itemize}
\item[(i)] $ \forall \, j \in \Z$  the map $B^2(\rho) \to \C^2$, $\zeta \mapsto (z_j(\zeta), w_j(\zeta))$ is analytic.
\item[(ii)] $ \forall \, j \in \Z$,  for any {\em real } $\zeta \in B_r^2(\rho)$  one has $\overline{z_j(\zeta)} = w_j(\zeta)$.
\item[(iii)]$ \forall \, j \in \Z$, for any {\em real } $\zeta \in B_r^2(\rho)$ one has  $	\abs{z_j(\zeta)}^2 = (\lambda_{j}^+(\zeta) - \lambda_{j}^-(\zeta))^2$.
	\end{itemize}
	\end{lemma}
	\begin{proof}
	(i) Define $A:\cY\times \cY \to \C$, $(f,g) \mapsto A(f,g):=(f, \imath g)_\cY$. By (H1) $A$ is a $\C$-multilinear continuous map, hence it is an analytic in each components  (see e.g. \cite{mujica}).\\
	Then $ z_j(\zeta) =  A\Big((L(\zeta)- \lambda_j^0)U_j(\zeta) f_{j0}^-, U_j(\zeta) f_{j0}^- \Big)_\cY$ is composition of analytic maps  and hence it is analytic. Analogous for $\zeta \mapsto w_j(\zeta)$.\\
(ii) We claim that   for  $\zeta$ real  
\begin{equation}
\label{i.f}
\imath f_j^-(\zeta) = f_j^+(\zeta) \ , \qquad \imath f_j^+(\zeta) = f_j^-(\zeta) \ , \qquad \forall j \in \Z \ . 
\end{equation}
This follows from the following chain of equalities (which hold for   $\zeta$ real and sufficiently small) 
$$
f_j^+(\zeta) \stackrel{\eqref{def.f}}{=} U_j(\zeta) \, f_{j0}^+ \stackrel{\eqref{i.eig}}{=} U_j(\zeta) \, \imath f_{j0}^- 
\stackrel{\eqref{U.comm}}{=} \imath \,  U_j(\zeta) \,  f_{j0}^- \stackrel{\eqref{def.f}}{=} \imath f_j^-(\zeta) \ .
$$
Thus for $\zeta$ real  and $\norm{\zeta}_2 \leq \rho$ one has  
\begin{align*}
\bar{z_j(\zeta) } & = \bar{\left( (L(\zeta)-\lambda_{j}^0) \, f_j^-(\zeta) \ , \imath f_j^-(\zeta) \right)_\cY} = \left( \imath  f_j^-(\zeta) \ , (L(\zeta)-\lambda_{j}^0) \, f_j^-(\zeta) \right)_\cY \\
& = \left( (L(\zeta)-\lambda_{j}^0) \, \imath f_j^-(\zeta) \ ,  f_j^-(\zeta) \right)_\cY \stackrel{\eqref{i.f}}{=} \left( (L(\zeta)-\lambda_{j}^0) \, f_j^+(\zeta) \ , \imath f_j^+(\zeta) \right)_\cY = w_j(\zeta) \ , 
\end{align*}
where in the third equality we used that for $\zeta$ real  $L(\zeta) - \lambda_{j}^0$ is self-adjoint.\\
(iii) By Lemma \ref{lem:P.ana}, for $\zeta$ real and $\norm{\zeta}_2 \leq \rho$ the operator  $U_j(\zeta)\vert_{E_j^0}$ is unitary. Since $f_{j0}^+$, $f_{j0}^-$ are orthogonal in $\cY$, the vectors  $f_j^+(\zeta)$, $f_j^-(\zeta)$ are orthogonal as well, thus form a basis for $E_j(\zeta)$. 	Let $M_j(\zeta)$ be the matrix of the self-adjoint operator $L(\zeta)-\lambda_j^0$ on this basis. One has
	$$
	M_j(\zeta) = 
	\begin{bmatrix}
	\left( (L(\zeta)-\lambda_{j0}) \, f_j^-(\zeta) \ , f_j^-(\zeta) \right)_\cY & \left( (L(\zeta)-\lambda_{j0}) \, f_j^-(\zeta) \ , f_j^+(\zeta) \right)_\cY \\
	\left( (L(\zeta)-\lambda_{j0}) \, f_j^+(\zeta) \ , f_j^-(\zeta) \right)_\cY & \left( (L(\zeta)-\lambda_{j0}) \, f_j^+(\zeta) \ , f_j^+(\zeta) \right)_\cY
	\end{bmatrix} = 
	\begin{bmatrix}
	a_j^1 & b_j \\
	\bar b_j & a_j^2
	\end{bmatrix}
	$$
	for some $a^1_j, a^2_j, b_j \in \C$. We show now that $a_j^1 = a_j^2 \in \R$.   Indeed using the self-adjointness of $L(\zeta)- \lambda_j^0$, 
	\begin{align*}
&  \bar{\left( (L(\zeta)-\lambda_{j}^0) \, f_j^-(\zeta) \ ,  f_j^-(\zeta) \right)_\cY} = \left( (L(\zeta)-\lambda_{j}^0) \, f_j^-(\zeta) \ ,  f_j^-(\zeta) \right)_\cY
\end{align*}
	while using \eqref{i.comm}, \eqref{i.f}
	\begin{align*}
&  \bar{\left( (L(\zeta)-\lambda_{j}^0) \, f_j^-(\zeta) \ ,  f_j^-(\zeta) \right)_\cY}  = \left( (L(\zeta)-\lambda_{j}^0) \, \imath f_j^-(\zeta) \ ,  \imath f_j^-(\zeta) \right)_\cY = \left( (L(\zeta)-\lambda_{j}^0) \, f_j^+(\zeta) \ , f_j^+(\zeta) \right)_\cY \ .
\end{align*}
        The eigenvalues of $M_j(\zeta)$ are the eigenvalues of $\left.L(\zeta) - \lambda_{j}^0\right|_{E_j(\zeta)}$, i.e. $\lambda_j^\pm(\zeta)- \lambda_j^0$. Then 
        $$
        (\lambda_j^+(\zeta) - \lambda_j^-(\zeta))^2 = ({\rm Tr }\, M_j(\zeta) )^2 - 4 {\rm det }\, M_j(\zeta)  = |b_j|^2 \ . 
        $$
        Now remark that by \eqref{i.f} one has  $z_j(\zeta) \equiv b_j$, $\forall j \in \Z$. 
        	\end{proof}
        	\vspace{1em}
	Since the maps $\zeta \mapsto z_j(\zeta)$, $\zeta \mapsto w_j(\zeta)$ are analytic, we can expand them in their absolutely and uniformly convergent Taylor series
\begin{equation}
\label{z.formal} 
 z_j(\zeta) = \sum_{n=1}^\infty Z_j^n(\zeta) \ , \quad  w_j(\zeta) = \sum_{n=1}^\infty W_j^n(\zeta) \ , \qquad j \in \Z \ ,
\end{equation}
 where $Z_j^n$ and $W_j^n$ are homogeneous polynomials of  degree $n$ in $\zeta$.
	By a direct computation  the first terms in the Taylor series are given by	
\begin{equation}
\label{taylor.exp}
\begin{aligned}
& z_j(\zeta) = \left( V(\zeta) f_{j0}^-, f_{j0}^+ \right) + 
 \left( V(\zeta) \left(L_0 - \lambda_{j}^0 \right)^{-1} 
 \left(\uno - P_{j0}\right) V(\zeta) f_{j0}^-,  f_{j0}^+ \right) +
  h.o.t. \ \\
 &  w_j(\zeta) = \left( V(\zeta) f_{j0}^+, f_{j0}^- \right) + 
 \left( V(\zeta) \left(L_0 - \lambda_{j}^0 \right)^{-1} 
 \left(\uno - P_{j0}\right) V(\zeta) f_{j0}^+,  f_{j0}^- \right) +
  h.o.t. 
  \end{aligned}
\end{equation}
Using (H4b)  one verifies  that 
\begin{equation}
\label{lin.term}
Z_j^1(\zeta) =  \xi_j \ , \qquad  W_j^1(\zeta) =  \eta_j \ , \qquad 
Z_j^2(\zeta) = W_j^2(\zeta) = 0 \ .
\end{equation}
The expression for the general homogeneous term $Z_j^n(\zeta)$ is much more involved and is given in the following proposition:
\begin{proposition}
\label{lem:exp}
For $\norm{\zeta}_2 \leq \rho < 1/8$,  the homogeneous polynomials $Z_j^n(\zeta)$ and $W_j^n(\zeta)$ are given $\forall j \in \Z$, $\forall n \in \N$, $n \geq 3$   by 
\begin{equation}
\label{Zjn.for}
\begin{aligned}
&Z_j^n(\zeta) = \sum_{k_1 +  \ldots  + k_n = j  }
\cK_{j}^n(k_1, \ldots, k_n)  \ \xi_{k_1} \, \eta_{-k_2} \,  \xi_{k_3} \, \cdots \eta_{-k_{n-1}}\, \xi_{k_n}  \\
&W_j^n(\zeta) = \sum_{k_1 +  \ldots  + k_n = j  }
\cK_{j}^n(k_1, \ldots, k_n)  \ \eta_{k_1} \, \xi_{-k_2} \,  \eta_{k_3} \, \cdots \xi_{-k_{n-1}}\, \eta_{k_n}   
\end{aligned}
\qquad \ , \qquad 
\mbox{ if }  n \mbox{ odd}
\end{equation}
 and by 
 $$Z_j^n(\zeta) = W_j^n(\zeta)= 0 \ , \qquad \qquad \qquad 
 \mbox{ if }  n \mbox{ even} \ . $$ 
The kernel $\cK_{j}^n(k_1, \ldots, k_n)$ has support in
$ k_1 + \ldots +  k_n = j $ and 
there exist $\tK_0, \, \tK_1 >0$ s.t. 
\begin{equation}
\label{eq:K.est} 
\abs{\cK_{j}^n(k_1, \ldots, k_n)} \leq \tK_0 \cdot \tK_1^{n-1} \  \prod_{m=1\atop m\,  {\rm odd}}^{n-1} \frac{1}{\la k_1 + \ldots + k_m - j\ra } \, \cdot \, \frac{1}{\la k_1 + \ldots + k_{m+1} \ra } \ . 
\end{equation}
\end{proposition}
The proof of the lemma, being quite technical, is postponed in Appendix \ref{app.exp}.
 
 \begin{remark}
 \label{rem:est}
 With  $\ff_n$ defined in \eqref{def:ff}, one has 
 \begin{equation}
 \label{eq:K.est1}
\abs{\cK_{j}^n(k_1, \ldots, k_n)} \leq \tK_0 \cdot \tK_1^{n-1} \   \ff_n(k_1, \ldots, k_n; j)
 \end{equation}
 \end{remark}

We are finally ready to define the map $\Psi$ of Theorem \ref{thm:psi}. 
First for $\zeta \in B^2(\rho)$, $\rho < 1/8$, let 
$$
Z(\zeta):= \left( z_j(\zeta) \right)_{j \in \Z} \ , \  \ \ \ W(\zeta):= \left( w_j(\zeta) \right)_{j \in \Z} \ . 
$$
	Now for any $\zeta \in B^2(\rho)$ we define the map $\Psi: B^2(\rho) \to \C^\Z \times \C^\Z$ by
	\begin{equation}
	\label{def.Psi}
\Psi(\zeta) :=   \Big( Z(\zeta) \ , \ W(\zeta) \Big) \ . 
	\end{equation}
In the rest of the section we show that   $\Psi$ fulfills the properties claimed in Theorem \ref{thm:psi}.

First note that, at least formally,  by   \eqref{lin.term} $\Psi = \uno + \Psi_3$, where  $\Psi_3(\zeta) := (Z_3(\zeta) , W_3(\zeta))$ and $Z_3:= Z - \uno$,   $W_3 := W - \uno$ are $O(\zeta^3)$.
 
In the next proposition we  show that,  provided $\tu \leq \tv\leq \tw$ are  weights fulfilling (W$)_p$, $\Psi_3$  is analytic and tame majorant analytic as a map  $B^{p,\tu}(\rho)\cap \spazio{p,\tv} \to \spazio{p,\tw}$ (in the sense of Definition \ref{def.na}).  
	\begin{lemma}
	\label{lem:psi00}
 Fix $1 \leq p \leq 2$ and  let $\tu \leq \tv \leq \tw$ be  weights fulfilling {\rm (W$)_p$}  with constants $R_0, R_1$. Then 	for any  
  $$
 0 <  \rho < \min\left( \frac{1}{8}, \ \frac{1}{ 8 \, \tK_1 \, R_1} \right)
  $$
    the map $\zeta \mapsto \Psi_3(\zeta)$ is analytic and tame  majorant analytic, $\Psi_3 \in \cN_\rho^T(B^{p,\tu}\cap \spazio{p,\tv}; \spazio{p,\tw})$ and 
    \begin{equation}
    \label{norm.psi3}
    \tame{\und{\Psi_3}}_\rho \leq 2^5\, \tK_0\, \tK_1^2 \,  R_0 \, R_1^2 \, \rho^3 \ . 
    \end{equation}
	\end{lemma}
	\begin{proof}
By formally Taylor expanding the map $\Psi_3$, one has that $\Psi_3 = \sum_{n\geq 3} \Psi^n$, where 
$$
\Psi^n(\zeta) :=   \Big( Z^n(\zeta) \ , \ W^n(\zeta) \Big) \ , \qquad  Z^n(\zeta):= \left(Z_j^n(\zeta) \right)_{j \in \Z} \ , \qquad  W^n(\zeta):= \left( W_j^n(\zeta) \right)_{j \in \Z}  \ .
$$ 
It is sufficient to show that  
$
\und{\Psi^n}(\zeta) :=   \Big( \und{Z^n}(\zeta) \ , \ \und{W^n}(\zeta) \Big) $
fulfills
		\begin{align}
		 \label{Z.mod.est} 
	&\norm{\und{Z^n}(\zeta)}_{p,\tw} +
\norm{\und{W^n}(\zeta)}_{p,\tw} \leq 
2 \, \tK_0 \,   R_0 \, \left(  2 \, \tK_1 \,R_1 \right)^{n-1} \, \norm{\zeta}_{p,\tu}^{n-1}\norm{\zeta}_{p,\tv}  \ , \qquad 	
	 \forall n \geq 3 \ . \\
	  \label{Z.mod.est2}
	 & \norm{\und{Z^n}(\zeta)}_{p,\tu} +
\norm{\und{W^n}(\zeta)}_{p,\tu} \leq 
2 \, \tK_0 \,   R_0 \, \left(  2 \, \tK_1 \,R_1 \right)^{n-1} \, \norm{\zeta}_{p,\tu}^n  \ , \qquad 	
	 \forall n \geq 3 \ . 
	 \end{align}
	Indeed \eqref{Z.mod.est}  implies that $\und{\Psi_3} = \sum_{n\geq 3} \und{\Psi^n}$  fulfills 
$$
\norm{\und{\Psi_3}(\zeta)}_{p,\tw}  \leq \sum_{n \geq 3}\norm{\und{\Psi^n}(\zeta)}_{p,\tw} 
\leq 
2 \, \tK_0 \, R_0 \, \norm{\zeta}_{p,\tv} \sum_{n \geq 3} \left(2 \, \tK_1\, R_1 \right)^{n-1} \, \norm{\zeta}_{p,\tu}^{n-1}   \ ,
$$
from which one deduces that
$$
\abs{\und{\Psi_3}}_{\rho}^T \equiv
 \sup_{\norm{\zeta}_{p,\tu}\leq \rho} 
 \frac{\norm{\und{\Psi_3}(\zeta)}_{p,\tw} }{\norm{\zeta}_{p,\tv}} \leq 2^4 \, \tK_0 \, \tK_1^2 \, R_0 \, R_1^2 \rho^2 \ .
$$
Analogously, using \eqref{Z.mod.est2} one proves that
$ \abs{\und{\Psi_3}}_{\rho}\equiv  \sup_{\norm{\zeta}_{p,\tu}\leq \rho} \norm{\und{\Psi_3}(\zeta)}_{p,\tu} \leq  2^4 \, \tK_0 \, \tK_1^2 \, R_0 \, R_1^2  \rho^3$.
Estimate \eqref{norm.psi3} follows by the definition of $\tame{\und{\Psi_3}}_\rho :=  \abs{\und{\Psi_3}}_\rho + \rho \abs{\und{\Psi_3}}_\rho^T$.
Note that  by Remark \ref{rem.conv} one has  that $\Psi_3$ is analytic.
		
	To prove \eqref{Z.mod.est}, \eqref{Z.mod.est2} we use the explicit formulas for $Z^n_j$ and $W^n_j$ given in Proposition \ref{lem:exp}.  We  perform the computations only for $Z^n$, since for $W^n$ they are identical.\\
	We begin by proving \eqref{Z.mod.est}.
Multiplying and dividing the kernel of $Z_j^n$ by $\sum_{l=1}^n \tv_{k_l} \prod_{m\neq l} \tu_{k_m}$ and using Cauchy-Schwartz  one gets for $1/p + 1/p'=1$
	\begin{align}
	\notag
|\und{Z_{j}^n}(\zeta)|  
&
\leq
 \norm{\frac{\cK_{j}^n({\bf k})}{\sum_{l=1}^n \tv_{k_l} \, \prod_{m\neq l} \tu_{k_m}} }_{\ell^{p'}(\Z^n)} 
 \cdot
\sum_{l=1}^n \beta_{j,l}   \ , \\
\notag
& \qquad \beta_{j,l} := \left(   \sum_{k_1 + \ldots + k_n = j  }  \, 
 \tv_{k_l}^p |\zeta_{k_l}|^p \, \prod_{m\neq l} \tu_{k_m}^p |\zeta_{k_m}|^p   \right)^{1/p} \ ;
\end{align}
by  Remark \ref{rem:est}  
\begin{align}
	\notag
|\und{Z_{j}^n}(\zeta)|  
&
\leq   \tK_0  \cdot \tK_1^{n-1} 
 \norm{\frac{\ff_n(\cdot ;j)}{\sum_{l=1}^n \tv_{k_l} \, \prod_{m\neq l} \tu_{k_m}} }_{\ell^{p'}(\Z^n)}
 \cdot
 \sum_{l=1}^n \beta_{j,l}  \ .
 \end{align}
Now we have
\begin{align*}
	\norm{\und{Z^n(|\zeta|)}}_{p,\tw} 
	& \leq \tK_0  \cdot \tK_1^{n-1} 
\sup_{j \in \Z}\ \left[ \tw_j \norm{\frac{\ff_n(\cdot ;j)}{\sum_{l=1}^n \tv_{k_l} \, \prod_{m\neq l} \tu_{k_m}} }_{\ell^{p'}(\Z^n)} \right] \norm{ \sum_{l=1}^n
\beta_{j,l}}_{\ell^p_j}	\\
& \leq \tK_0 R_0  \cdot (\tK_1 R_1)^{n-1} 
\sum_{l=1}^n \norm{ 
\beta_{j,l}}_{\ell^p_j}	\leq \tK_0 R_0  \cdot (\tK_1 R_1)^{n-1} 
 \cdot n \, \norm{\zeta}_{p,\tv} \, \norm{\zeta}_{p,\tu}^{n-1}  \\
& \leq \tK_0 R_0 \cdot (2\,\tK_1 R_1)^{n-1} \cdot  \norm{\zeta}_{p,\tv} \, \norm{\zeta}_{p,\tu}^{n-1} 
	\end{align*}	
	where to pass to the second line we used \eqref{w1}, while in the third inequality we used that
$$
\norm{ 
\beta_{j,l}}_{\ell^p_j}	\equiv \left(\sum_j |\beta_{j,l}|^p \right)^{1/p} \leq \left(\sum_j \sum_{k_1 + \cdots + k_m = j}  \tv_{k_l}^p |\zeta_{k_l}|^p \, \prod_{m\neq l} \tu_{k_m}^p |\zeta_{k_m}|^p  \right)^{1/p} \leq \norm{\zeta}_{p,\tv} \, \norm{\zeta}_{p,\tu}^{n-1}  \ .
$$

	To prove  \eqref{Z.mod.est2} it is enough to repeat the computations above with $\tw = \tv = \tu$ and use \eqref{w10}. \\
\end{proof}
Next  we  study the operator $\di \Psi_3(\zeta)^t$, which is the  transposed of $\di \Psi_3(\zeta)$ w.r.t. the (complex)  scalar product:
$$
\Big( \di \Psi_3(\zeta)(\xi^1 , \eta^1),  \bar{(\xi^2, \eta^2)} \Big) = \Big( (\xi^1 , \eta^1), \bar{\di \Psi_3(\zeta)^t (\xi^2, \eta^2)}  \Big) \ , \quad (\xi^1 , \eta^1), (\xi^2 , \eta^2) \in \spazio{2} \ .
$$
We prove the following  result:
\begin{lemma}
\label{lem:dPsi.1}
With the same assumptions of Lemma \ref{lem:psi00}, let 
  \begin{equation}
  \label{rho.size2}
   0 < \rho < \min\left( \frac{1}{8}, \ \frac{1}{4\, \tK_1 \, R_1} \right) \ .
  \end{equation}
Then     the following holds true:  
\begin{itemize}
\item[(i)] For any  $n \geq 3$,  $\zeta \mapsto \und{\di \Psi^n}(\zeta)^t \in \cN_{\rho}^T(B^{p,\tu}\cap \spazio{p,\tv}; \cL(\spazio{p,\tv}, \spazio{p,\tw}))$ and fulfills 
\begin{equation}
\label{d.Psi.est00}
\tame{[\und{\di \Psi^n}]^t }_{\rho} 
  \leq  2^5 \, \tK_0  \,  R_0  \left(2\, \tK_1  \, R_1\right)^{n-1} \, \rho^{n-1} \ , \quad \forall n \geq 3 \ .
\end{equation}
\item[(ii)] The map  $\di \Psi_3^*  \in \cN_\rho^T(B^{p,\tu}\cap \spazio{p,\tv}, \, \cL(\spazio{p,\tv}, \spazio{p,\tw}))$  and moreover
 \begin{equation}
 \label{normPsi2}
 \tame{\und{\di \Psi_3^*}}_\rho 
  \leq 
  2^{8} \,\tK_0 \, \tK_1^2 \, R_0 \,   R_1^2 \, \rho^2 \ .
 \end{equation}
\end{itemize}
\end{lemma}
\begin{proof}
We will prove that for any $n \geq 3$
\begin{align}
\label{d.Psi.est0}
 & \norm{\und{\di \Psi^n}(\zeta)^t \upsilon}_{p,\tw} 
  \leq  2^4 \, \tK_0 \, R_0 \, \left(2 \, \tK_1 \, R_1\right)^{n-1} \, \left(\norm{\upsilon}_{p,\tv} \norm{\zeta}_{p,\tu}  + \norm{\upsilon}_{p,\tu} \norm{\zeta}_{p,\tv}\right) \, \norm{\zeta}_{p,\tu}^{n-2} \ , \\
  \label{d.Psi.est1}
&  \norm{\und{\di \Psi^n}(\zeta)^t \upsilon}_{p,\tu} 
  \leq  2^4 \, \tK_0 \, R_0 \, \left(2 \, \tK_1 \, R_1\right)^{n-1} \, \norm{\upsilon}_{p,\tu} \, \norm{\zeta}_{p,\tu}^{n-1} \ ,
\end{align}
from which item $(i)$ follows.\\
  The $j^{th}$-component of $\di \Psi^n(\zeta)^t \upsilon$ is given by  
$[\di \Psi^n(\zeta)^t \upsilon]_j  = \Big( A^n_j(\zeta) \upsilon \ , \   B^n_j(\zeta) \upsilon \Big)$ where,  letting  $\upsilon = (\wt\xi, \wt \eta)$, 
$$
 A^n_j(\zeta) (\wt\xi, \wt\eta) :=  \sum_\ell \frac{\partial Z^n_\ell(\zeta)}{\partial \xi_j} \wt\xi_\ell + 
\frac{\partial W^n_\ell(\zeta)}{\partial \xi_j} \wt\eta_\ell \ , \qquad 
  B^n_j(\zeta) (\wt\xi, \wt\eta) := 
 \sum_\ell \frac{\partial Z^n_\ell(\zeta)}{\partial \eta_j} \wt\xi_\ell + 
\frac{\partial W^n_\ell(\zeta)}{\partial \eta_j} \wt\eta_\ell \ .
$$
 To compute such terms explicitly we use  Proposition \ref{lem:exp}. One has for example (we compute only   $A^n_j(\zeta) (\wt\xi, \wt\eta) $, the other is analogous) 
\begin{align*}
A^n_j(\zeta) &(\wt\xi, \wt\eta)  = I_j + II_j \ , \\
& I_j := \sum_{\ell \in \Z} \frac{\partial Z^n_\ell(\zeta)}{\partial \xi_j} \wt\xi_\ell =   \sum_{r=1 \atop r \ {\rm odd}}^{n} \sum_{\tk \in \fS^{n,r}_{-j} }
\cK_{k_r}^n(k_1,  \ldots, k_{r-1}, j, k_{r+1}, \ldots, k_n)  \  \wt\xi_{k_r} \cdot   \frac{\xi_{k_1} \eta_{-k_2} \ldots \xi_{k_n}}{\xi_{k_r}}\\
&II_j :=\sum_{\ell \in \Z} \frac{\partial W^n_\ell(\zeta)}{\partial \xi_j} \wt\eta_\ell  = 
\sum_{r=1 \atop r \ {\rm even}}^{n} \sum_{\tk \in \fS^{n,r}_{j} }
\cK_{k_r}^n(k_1,  \ldots, k_{r-1}, -j, k_{r+1}, \ldots, k_n)  \  \wt\eta_{k_r} \cdot   \frac{\eta_{k_1} \xi_{-k_2} \ldots \eta_{k_n}}{\xi_{k_r}}
\end{align*}
where ${\bf k} := (k_1, \ldots, k_n) \in \Z^n$ and $\fS^{n,r}_{a}$  is the set defined by  
\begin{equation}
\label{set.2}
\fS^{n,r}_{a}:= \Big\{{\bf k} \in \Z^n \colon \sum_{i=1 \atop i \neq r}^n k_i  -k_r = a  \Big\} \ . 
\end{equation}
Then $\abs{\und{A^n_j}(\zeta) (\wt\xi, \wt\eta)} \leq \und{I_j}+\und{II_j}$, where
\begin{align*}
\und{I_j} := \sum_{r=1 \atop r \ {\rm odd}}^{n} \sum_{\tk \in \fS^{n,r}_{-j} }
\abs{\cK_{k_r}^n(k_1,  \ldots, k_{r-1}, j, k_{r+1}, \ldots, k_n) } \  |\wt\xi_{k_r}| \cdot   \frac{|\xi_{k_1}| \, |\eta_{-k_2}| \,  \ldots |\xi_{k_n}|}{|\xi_{k_r}|}  \ , \\ 
\und{II_j}:=  \sum_{r=1 \atop r \ {\rm even}}^{n} \sum_{\tk \in \fS^{n,r}_{j} }
\abs{\cK_{k_r}^n(k_1,  \ldots, k_{r-1}, -j, k_{r+1}, \ldots, k_n)}  \  |\wt\eta_{k_r}| \cdot   \frac{|\eta_{k_1}| |\xi_{-k_2}| \ldots |\eta_{k_n}|}{|\xi_{k_r}|} \ .
\end{align*}
To estimate $\abs{\und{I_j}}$, we first multiply and divide the kernel of $\und{I_j}$ by $\sum_{l=1}^n \tv_{k_l} \prod_{m \neq l} \tu_{k_m}$, then we use by Cauchy-Schwartz to obtain
\begin{align*}
|\und{I_j}| & \leq
 \sum_{r=1 \atop r \ {\rm odd}}^n 
\norm{\frac{\cK_{k_r}^n(k_1,  \ldots, k_{r-1},   j, k_{r+1}, \ldots, k_n)}{\sum_{l=1}^n \tv_{k_l} \prod_{m \neq l} \tu_{k_m}}}_{\ell^{p'}(\Z^n)} \, 
 \sum_{l=1}^n \beta_{j,l,r} \\
&  \beta_{j,l,r} := 
\left(\sum_{\tk \in \fS^{n,r}_{-j} } 
 \frac{ |\wt\xi_{k_r}|^p \, |\xi_{k_1}|^p \, |\eta_{-k_2}|^p \,  \ldots |\xi_{k_n}|^p}{|\xi_{k_r}|^p} 
\tv_{k_l}^p \prod_{m \neq l} \tu_{k_m}^p
 \right)^{1/p} ; 
\end{align*}
then since $\abs{\cK_{k_r}^n(k_1,  \ldots, k_{r-1},  \pm j, k_{r+1}, \ldots, k_n) } \leq \tK_0 \cdot \tK_1^{n-1} \, \fg_{n,r}({\bf k};  \pm j)$ for any ${\bf k} \in \fS^{n,r}_{\mp j}$ one gets
$$
|\und{I_j}| \leq 
 \tK_0 \cdot \tK_1^{n-1} 
  \sum_{r=1 \atop r \ {\rm odd}}^n 
  \norm{\frac{\fg_{n,r}(\cdot; j)}{\sum_{l=1}^n \tv_{k_l} \prod_{m \neq l} \tu_{k_m}}}_{\ell^{p'}(\Z^n)} \, 
 \sum_{l=1}^n \beta_{j,l,r}
  \ .
$$
Finally 
\begin{align*}
\left( \sum_{j \in \Z} \tw_j^p \und{I_j}^p \right)^{1/p}
&\leq 
\tK_0 \cdot \tK_1^{n-1} 
\max_{1 \leq r \leq n} 
\sup_{j \in \Z}
\, \tw_j   \norm{\frac{\fg_{n,r}(\cdot; j)}{\sum_{l=1}^n \tv_{k_l} \prod_{m \neq l} \tu_{k_m}}}_{\ell^{p'}(\Z^n)} \
\norm{\sum_{r}\sum_l \beta_{j,l,r}}_{\ell^p_j} \\
& \leq 
\tK_0 R_0 \cdot (\tK_1 R_1)^{n-1} \sum_{r,l} \norm{\beta_{j,l,r}}_{\ell^p_j}\\
& \leq
4\, \tK_0 R_0 \cdot (2 \tK_1 R_1)^{n-1}  \left( 
\norm{\upsilon}_{p,\tu} 
\norm{\zeta}_{p,\tv} +
 \norm{\upsilon}_{p,\tv}
\norm{\zeta}_{p,\tu}
\right) \norm{\zeta}_{p,\tu}^{n-2} 
\end{align*}
where to pass to the second inequality we used \eqref{w2}, while to pass to the third inequality we used the explicit expression of $\beta_{j,l,r}$ and the inequality $\sum_{l,r} 1 \leq 4 \cdot 2^{n-1}$.\\
Clearly  $\left(\sum_j \tw_j^p |\und{II_j}|^p\right)^{1/p}$ is bounded by the same quantity.    It follows that 
\begin{align}
\notag
\norm{\und{A^n}(|\zeta|) |\upsilon|}_{p,\tw} 
& \leq \norm{(\und{I_j})_{j \in \Z}}_{p,\tw}  
+ \norm{(\und{II_j})_{j \in \Z}}_{p,\tw} \\
\label{A.est}
& \leq 
 2^3 \, \tK_0\,  R_0 \, \left(2 \,\tK_1\, R_1\right)^{n-1}\,   \norm{\zeta}_{p,\tu}^{n-2}  \left(  \norm{\upsilon}_{p,\tv} \norm{\zeta}_{p,\tu} + 
  \norm{\upsilon}_{p,\tu} \norm{\zeta}_{p,\tv}\right) \ . 
\end{align}
One verifies that $\norm{ \und{B^n}(\zeta) \upsilon}_{p,\tw}$ admits the same bound \eqref{A.est}. Thus  estimate \eqref{d.Psi.est0}  follows.\\
The estimate for \eqref{d.Psi.est1} is obtained similarly
 putting $\tw = \tv = \tu$. Note that in such a case it is enough to use the sub-multiplicative property of the weight and estimate \eqref{w20}.

We prove now $(ii)$. By the results of item $(i)$ we need just to check that $\zeta  \mapsto \di \Psi(\zeta)^* \in\cN_\rho^T(B^{p,\tu}\cap\spazio{p,\tv}, \cL(\spazio{p,\tv},\spazio{p,\tw}))$. 
However this follows from the identity
$$
\und{\di \Psi_3(|\zeta|)^*|\upsilon|} =  \und{\di \Psi_3(|\zeta|)^t }|\upsilon|
$$
and the previous result. Finally 
$ \tame{\und{\di \Psi_3^t}}_{\rho}
 \leq \sum_{n \geq 3}\tame{\und{\di \Psi_n^t}}_{\rho}
 \leq 
 2^{8} \,\tK_0  R_0 \, \tK_1^2  R_1^2 \, \rho^2. $
\end{proof}

%
%
%
%

Then one obtains immediately the following 
\begin{lemma}
\label{lem:dPsi.2}
 Fix $1 \leq p \leq 2$ and  let $\tu \leq \tv \leq \tw$ be weights fulfilling {\rm (W$)_p$}  with constants $R_0, R_1$. Then 	for any  
  $$
 0 < \rho < \min\left( \frac{1}{2^{4}} , \ \frac{1}{ 8 \, \tK_1 \,  R_1} \right)
  $$    the map  $\Psi_3  \in \sT_{ \tu, \tv, \rho}^{\tw, 3}$ and moreover
 \begin{equation*}
 \norm{\Psi_3}_{ \sT_{ \tu, \tv, \rho}^{\tw, 3}} \leq 2^{10} \,\tK_0 \, R_0 \, \, \tK_1^2 \, R_1^2 \, \rho'^3 \ .
 \end{equation*}
\end{lemma}
\begin{proof} 
Let $\rho' = \rho/2$. Then by Lemma \ref{lem:psi00}, Lemma \ref{lem:dPsi.1} and Cauchy estimates \eqref{diff.1},\eqref{diff.2}
$$
\norm{\Psi_3}_{ \sT_{ \tu, \tv, \rho'}^{\tw, 3}} \leq  2 \tame{\und{\Psi_3}}_{2\rho'} + \rho' \tame{\und{\di \Psi_3^*}}_{\rho'} \leq 2^{10} \,\tK_0 \, R_0 \, \, \tK_1^2 \, R_1^2 \, \rho'^3  \ , 
$$
then we denote again $\rho' \equiv \rho$.
\end{proof}

We conclude the section with the proof of  Theorem \ref{thm:psi}.

\begin{proof}[Proof of Theorem \ref{thm:psi}.]
 Fix $1 \leq p \leq 2$. By Proposition \ref{prop:sw}$(i)$ the weights $\tu = \tv  = \tw = \{ 1 \}_{j \in \Z}$  fulfill {\rm (W$)_p$}  with $R_0 = 1$ and $1 \leq R_1 \equiv  \left(\sum_{k \in \Z} \frac{1}{\la k\ra^{p'} }\right)^{1/p'} \leq 2$. 
 Then  by Lemma \ref{lem:dPsi.2},  $\Psi - \uno  \in \sT_{\tu, \tv, \rho}^{\tw, 3}$ for any   
$\rho < \min(2^{-4}, 2^{-4} \tK_1^{-1})\equiv \varrho_* $. 
Item $(i)$ of Theorem \ref{thm:psi} follows by Lemma \ref{lem:z.ana} $(iii)$. 
 Item $(ii)$ of Theorem \ref{thm:psi} follows by Lemma \ref{lem:dPsi.2}, with $\tC = 2^{10}\, \tK_0 \,\tK_1^2$.
Item $(iii)$ follows by  Lemma \ref{lem:z.ana}(ii). 
\end{proof}

\appendix

\section{Proof of tame Kuksin-Perelman theorem}
\label{app:KP}

 \subsection{Properties of tame majorant analytic functions}
 \label{sec:prop}
In this section we show that the class of tame majorant analytic maps is closed under
several operations like composition, inversion and flow-generation,
and provide new quantitative estimates which will be used during the
proof of Theorem \ref{KP}.  In the rest of the section denote by $S:=\sum_{n=1}^\infty 1/n^2$ and by
\begin{equation}
\label{mu.def} 
\mu:=1/e^2 32  S \approx 0.0025737 > 2^{-10} \ .
\end{equation}
\begin{lemma}
\label{lem:fn.tame}
Let  $ F \in
  \cN_{\rho}^T(B^{p,w^0}\cap \spazio{p,w^1}, \spazio{p,w^{2}})$. Write  $F = \sum_{n \geq 0} F^n$ and   denote by $\wt F^n$ the symmetric multilinear map associated to $F^n$. Then each multilinear map  $\wt F^n$  is a $(p,w^0,w^1, w^2)$- {\em tame modulus map} in the sense that 
  \begin{equation}
\label{m.tame}
\norm{ \und{\wt F^n} (\zeta^{(1)}, \ldots, \zeta^{(n)})}_{p,w^2}
 \leq e^n \frac{ \abs{\und F}_\rho^T}{\rho^{n-1}}  \,
  \frac{1}{n} \sum_{l=1}^n \norm{\zeta^{(l)}}_{p,w^1} \prod_{m \neq l}\norm{\zeta^{(m)}}_{p,w^0} 
  \ , \qquad \forall \zeta^{(1)}, \ldots, 
  \zeta^{(n)} \in \spazio{p,w^1} \ .
\end{equation}
\end{lemma}
\begin{proof}
By Cauchy formula
$$
\und{\wt F^n}(\zeta^{(1)}, \ldots , \zeta^{(n)}) = \frac{1}{(2\pi \im )^n n!} \oint_{|\lambda_1| = \epsilon_1} \cdots 
 \oint_{|\lambda_n| = \epsilon_n} \frac{\und{F}(\lambda_1 \zeta^{(1)} + \cdots + \lambda_n \zeta^{(n)})}{\lambda_1^2 \cdots \lambda_n^2} \di \lambda_1 \cdots \di \lambda_n \ ;
$$
such  formula is well defined provided $\sum_j \lambda_j \,  \zeta^{(j)} \in B^{p, w^0}(\rho)$: this is true e.g. choosing $\epsilon_i =  \rho/n \norm{\zeta^{(i)}}_{p,w^0}$ $\, \forall 1 \leq i \leq n$. Then
\begin{align*}
\norm{ \und{\wt F^n} (\zeta^{(1)}, \ldots, \zeta^{(n)})}_{p,w^2}
 & \leq \frac{\abs{\und F}_\rho^T}{n! \epsilon_1 \cdots \epsilon_n} \sum_j \epsilon_j \norm{\zeta^{(j)}}_{p, w^1} 
 \leq  
  \frac{\abs{\und F}_\rho^T}{n!} 
  \sum_j \frac{\norm{\zeta^{(j)}}_{p,w^1}}{\prod_{l \neq j} \epsilon_l} \\
 & \leq \frac{\abs{\und F}_\rho^T}{\rho^{n-1}} \frac{n^n}{n!} \frac{1}{n}\sum_j\norm{\zeta^{(j)}}_{p, w^1} \prod_{l \neq j} \norm{\zeta^{(l)}}_{p,w^0}
\end{align*}
and the claimed estimate follows.
\end{proof}

 \begin{lemma}
\label{rem:pro}
Let $1 \leq p \leq 2$.
Let $w^0 \leq w^1 \leq w^2 \leq w^{3}$ be weights. 
\begin{itemize}
\item[(i)] Let $F \in \cN_\rho^T(B^{p,w^0}\cap\spazio{p,w^1},\, \spazio{p,w^2})$ and $\cG\in \cN_{\rho}^T(B^{p,w^0}\cap\spazio{p,w^1},\, \cL(\spazio{p,w^1}, \spazio{p,w^2}))$.
Define $H(\zeta) = \cG(\zeta) F(\zeta)$. Then  $H \in \cN_\rho^T(B^{p,w^0}\cap\spazio{p,w^1},\, \spazio{p,w^2})$ and 
 $\tame{\und H}_\rho \leq \tame{\und \cG }_\rho \tame{ \und F}_\rho$.
\item[(ii)] Let $\cH, \cG\in \cN_{\rho}^T(B^{p,w^0}\cap\spazio{p,w^1},\, \cL(\spazio{p,w^1}, \spazio{p,w^2}))$. Define $\cI(\zeta)\upsilon := \cH(\zeta)\cG(\zeta)\upsilon$ for $\upsilon \in \spazio{p,w^1}$. Then $\cI \in\cN_{\rho}^T(B^{p,w^0}\cap\spazio{p,w^1},\, \cL(\spazio{p,w^1}, \spazio{p,w^2}))$ and  
 $\tame{\und{\cI }}_\rho \leq \tame{\und \cH }_\rho \tame{ \und \cG}_\rho$.
\end{itemize}
\end{lemma}
\begin{proof}
(i) One has that $\und{\cH(|\zeta|)} \leq \und{\cG}(|\zeta|) \und F(|\zeta|)$, which implies immediately that
 $\abs{\und{H }}_\rho \leq \abs{\und \cG }_\rho \abs{\und F}_\rho$
 and $\abs{\und{H }}_\rho^T \leq \abs{\und{\cG}}_\rho^T
  \left( \abs{\und F}_\rho + \rho \abs{\und F}_\rho^T \right)$. The claim follows.\\
(ii) One has that $\und{\cH(\zeta)\cG(\zeta)}\upsilon \leq \und{\cH}(|\zeta|) \und{\cG}(|\zeta|)|\upsilon|$, then the claim follows as above.
\end{proof}

 \begin{lemma}
 \label{FGinN}
Let $1 \leq p \leq 2$.
Let $w^0 \leq w^1 \leq w^2 \leq w^{3}$ be weights. 
\begin{itemize}
\item[(i)]  Let $G\in\cN_{\rho}(\spazio{p,w^0}, \spazio{p,w^1})$ with $\abs{\und{G}}_{\rho}\leq
  \sigma$ and $ F \in
  \mathcal{N}_{\sigma}(\spazio{p,w^1}, \spazio{p,w^2})$. Then $F \circ G$ belongs to $\mathcal{N}_\rho(\spazio{p,w^0},\spazio{p,w^2})$ and
$\abs{\und{F \circ G}}_\rho\leq \abs{\und{F}}_\sigma.$
\item[(ii)] Let $G\in\cN_{\rho}^T(B^{p,w^0}\cap \spazio{p,w^1}, \spazio{p,w^2})$ with $\tame{\und{G}}_{\rho}\leq
  \sigma$ and $ F \in
  \cN_{\sigma}^T(B^{p,w^0}\cap\spazio{p,w^2}, \spazio{p,w^{3}})$. Then $F \circ G \in \cN_\rho^T(B^{p,w^0}\cap\spazio{p,w^1},\spazio{p,w^{3}})$ and
$\tame{\und{F \circ G}}_\rho\leq \tame{\und{F}}_\sigma $.
\item[(iii)] Let $G\in\cN_{\rho}^T(B^{p,w^0}\cap\spazio{p,w^1}, \spazio{p,w^2})$ with $\tame{\und{G}}_{\rho}\leq
  \sigma$ 
  and 
  $\cE \in \cN_\sigma^T(B^{p,w^0}\cap\spazio{p,w^1}, \cL(\spazio{p,w^1}, \spazio{p,w^3}))$. Define $\cH(\zeta) = \cE(G(\zeta))$. Then $\cH \in \cN_\rho^T(B^{p,w^0}\cap\spazio{p,w^1}, \cL(\spazio{p,w^1}, \spazio{p,w^3}))$ and
$\tame{\und \cH}_\rho \leq \tame{\und \cE}_{\sigma}$.
\end{itemize}
\end{lemma}
\begin{proof}
(i) It follows as in  \cite[Lemma A.1]{masp_toda}.\\
(ii) Recall that $\und{F \circ G}(|\zeta|) \leq \und F \circ \und G(|\zeta|)$ (cf. \cite{kuksinperelman}). Provided $\abs{\und{G}}_{\rho}\leq
  \sigma$ one has 
$\norm{\und F(\und G(|\zeta|))}_{p,w^3} \leq \abs{\und F}_\sigma^T \abs{\und G}_{\rho}^T \norm{\zeta}_{p,w^1}$,
hence $\abs{\und {F\circ G}}_\rho^T \leq \abs{\und F \circ \und G}_\rho^T \leq \abs{\und F}_\sigma^T \abs{\und G}_\rho^T$.
 Thus
$$
\tame{\und{F \circ G}}_\rho \equiv \abs{\und{F \circ G}}_\rho + \rho \abs{\und{F \circ G}}_\rho^T
 \leq \abs{\und{F}}_\sigma + \abs{\und F}_\sigma^T \abs{\und G}_{\rho}^T \rho
  \leq 
\abs{\und{F}}_\sigma + \abs{\und F}_\sigma^T \sigma \equiv \tame{\und F}_\sigma \ . 
$$
(iii) First $\und \cH(|\zeta|)|\upsilon| \leq \und{ \cE}(\und G(|\zeta|))|\upsilon|$, which implies 
 $\abs{\und \cH}_\rho \leq \abs{\und \cE}_\sigma$.
  Furthermore, using also  $\abs{\und G}_\rho^T \leq \sigma/\rho$, 
\begin{align*}
\norm{\und \cH(|\zeta|)|\upsilon|}_{p,w^3}
& \leq \abs{\und \cE}_\sigma^T \left( \norm{\und G(|\zeta|)}_{p,w^1} \norm{\upsilon}_{p, w^0} + \sigma \norm{\upsilon}_{p,w^1}\right)
 \leq
 \abs{\und \cE}_\sigma^T  
  \left(\abs{\und G}_\rho^T \norm{\zeta}_{p,w^1} \norm{\upsilon}_{p,w^0} + \sigma \norm{\upsilon}_{p,w^1}\right) \\
  & \leq \abs{\und \cE}_\sigma^T  \frac{\sigma}{\rho} 
   \left( \norm{\zeta}_{p,w^1} \norm{\upsilon}_{p,w^0} + \rho \norm{\upsilon}_{p,w^1}\right)
\end{align*}
therefore $\abs{\und \cH}_\rho^T \leq \abs{\und \cE}_\sigma^T  \frac{\sigma}{\rho}$. Finally
$
\tame{\und \cH}_\rho  = \abs{\und \cH}_\rho + \rho \abs{\und \cH}_\rho^T \leq  \abs{\und \cE}_\sigma + \sigma \abs{\und \cE}_\sigma^T \equiv \tame{\und \cE}_{\sigma}.
$
\end{proof}

\begin{lemma}
\label{Ginverso}
Fix  $1 \leq p \leq 2$ and weights $w^0 \leq w^1 \leq w^2$.
\begin{itemize}
\item[(i)] Let  $F \in \mathcal{N}_\rho(\spazio{p,w^0}, \, \spazio{p,w^0})$,  $F=O(\zeta^N)$ for some $N\geq 2$ and $\abs{\und{F}}_\rho \leq \rho/e$. 
Then the map $\uno + F$ is invertible in $B^{p,w^0}(\mu \rho)$, $\mu$ as in \eqref{mu.def}. 
Moreover 
there exists  $G\in\cN_{\mu\rho}(\spazio{p,w^0}, \, \spazio{p,w^0})$, 
$G=O(\zeta^N)$, such that 
  $(\uno+F)^{-1}=\uno-G$, and  
\begin{equation}
\label{G.inv.est}
\abs{\und{G}}_{\mu\rho}\leq
\frac{\abs{\und{F}}_\rho}{8} .
\end{equation}
\item[(ii)] Assume that  $F \in \mathcal{N}_\rho^T(B^{p,w^0}\cap\spazio{p,w^1}, \, \spazio{p,w^2})$, $F=O(\zeta^N)$ for some $N\geq 2$ and 
 $\tame{\und{F}}_\rho \leq \rho/e$, then $G\in\cN_{\mu\rho}^T(B^{p,w^0}\cap\spazio{p,w^1}, \, \spazio{p,w^2})$ and
\begin{equation}
\label{G.inv.est2}
 \tame{\und{G}}_{\mu\rho}\leq
\frac{\tame{\und{F}}_\rho}{8} \ .
\end{equation}
\end{itemize}
\end{lemma}
\begin{proof}
Item (i) follows as in  \cite[Lemma A.2]{masp_toda}. We prove item (ii). Following the scheme of \cite{masp_toda}, $G$ is given by the power series $G = \sum_{n \geq 2} G^n$, where the homogeneous polynomial $G^n = 0$ for $1 < n < N$ and 
\begin{equation}
\label{recursive_inverse}
{G}^n(\zeta)=
\sum_{r=2}^{n}\sum_{k_1+\cdots+k_r=n}{\tilde{F}}^r\Big({G}^{k_1}(\zeta),\cdots,{G}^{k_r}(\zeta)\Big), \qquad \forall \, n \geq N \ .
\end{equation}
In the formula above  $k_1, \ldots, k_r \in \N$, and we wrote $F = \sum_{r \geq N} F^r$, where $
F^r$ is a homogeneous polynomial of degree $r$ and $\tilde{F}^r$ is its
associated  multilinear map (see \eqref{polin}). Moreover we write
$G^1(\zeta) := \zeta$.  We show now that the formal series $G = \sum_{n \geq N}
G^n$ with $G^n$ defined by \eqref{recursive_inverse} is a tame majorant analytic map. Note that
\begin{equation}
 \und{G^n}(|\zeta|)\leq \sum_{r=N}^{n}\sum_{k_1+\cdots+k_r=n}
 \und{\tilde{F}^r}\Big(\und{G^{k_1}}(|\zeta|),\ldots,\und{G^{k_r}}(|\zeta|)\Big).
\end{equation}
We prove that there exists a constant  $A>0$ such that
\begin{align}
&\norm{\und{G^n}(|\zeta|)}_{p,w^2}\leq \frac{\abs{\und{F}}_\rho^T}{8Sn^2}A^{n-1}\norm{\zeta}_{p,w^0}^{n-1}\norm{\zeta}_{p,w^1}, \qquad \forall n \geq N \ ,
\label{induzioneserie1}
 \\
&\norm{\und{G^n}(|\zeta|)}_{p,w^0}\leq \frac{\abs{\und{F}}_\rho}{8Sn^2}A^n\norm{\zeta}_{p,w^0}^{n} \ , \qquad \forall n \geq N 
\label{induzioneserie2}
\end{align}
The proof is by induction on $n$.
 For $n=N$, by \eqref{recursive_inverse} it follows that   $G^N(\zeta) = \tilde{F}^N(\zeta, \ldots, \zeta)$. 
 Using also Lemma \ref{lem:fn.tame} one has
$$
\norm{\und{G^N}(|\zeta|)}_{p,w^2} \leq  e^N\frac{\abs{\und{F}}_\rho^T}{\rho^{N-1}} \norm{\zeta}^{N-1}_{p,w^0} \norm{\zeta}_{p,w^1} \ ,
$$
thus   \eqref{induzioneserie1} holds for $n=N$ with   $A =  e^2 32 S / \rho$. The proof of \eqref{induzioneserie2} is analogous, and we skip it.
We prove now the inductive step $n-1 \leadsto n$. Assume therefore that \eqref{induzioneserie1}, \eqref{induzioneserie2} hold up to order $n-1$. Then one has
\begin{align*}
 \norm{\und{G^n}(|\zeta|)}_{p, w^2}
 &
 \leq\sum_{r=N}^{n}
 \sum_{k_1+\cdots+k_r=n}
 \frac{\abs{\und F}_\rho^T}{\rho^{r-1}} \frac{e^r}{r} \sum_{\ell=1}^r 
 \norm{\und{G^{k_\ell}}(|\zeta|)}_{p,w^2} \prod_{m \neq \ell} \norm{\und{G^{k_m}}(|\zeta|)}_{p,w^0}
 \\ 
 &\leq \frac{\abs{\und F}_\rho^T}{8S} A^{n-1} \norm{\zeta}_{p,w^0}^{n-1} \norm{\zeta}_{p,w^1} 
 e \abs{\und{F}}_\rho^T
 \sum_{r=N}^{n} 
 \left( \frac{e \, \abs{\und{F}}_\rho}{8S \rho}\right)^{r-1}
 \sum_{k_1+\cdots+k_r=n} 
 \frac{1}{ k_1^2 \cdots k_r^2}
\\
 &\leq\frac{\abs{\und{F}}_\rho^T}{8 S n^2}A^{n-1}
 \norm{\zeta}_{p,w^0}^{n-1}\norm{\zeta}_{p, w^1} e \abs{\und{F}}_\rho^T \sum_{r=1}^\infty{\left(\frac{e\abs{\und{F}}_\rho}{2\rho}\right)^r}\\
 &\leq
 \frac{\abs{\und{F}}_\rho^T}{8Sn^2}A^{n-1}\norm{\zeta}_{p,w^0}^{n-1}\norm{\zeta}_{p, w^1} 
\end{align*}
where in the first inequality we used  $w^1 \leq w^2$, in the second the inductive assumption and  in the last  we used  the hypothesis $\abs{\und{F}}_\rho + \rho\abs{\und{F}}_\rho^T \leq \rho/e$. 
Finally to pass from the second to the third line we used the following standard inequality (see e.g. \cite[Lemma A.5]{masp_toda})
\begin{equation}
n^2\sum_{k_1+\cdots+k_r=n}\frac{1}{ k_1^2 \cdots k_r^2}\leq (4S)^{r-1}, \qquad n \geq 1 \ .
\label{disserieconv2}
\end{equation}
In such a way  we proved \eqref{induzioneserie1}. The proof of \eqref{induzioneserie2} is analogous (for details see \cite{masp_toda}). 
Finally from \eqref{induzioneserie1} one has 
$\abs{\und G}_{\mu\rho}^T \leq \sum\limits_{n\geq N} \frac{\abs{\und{F}}_\rho^T}{8Sn^2}(A\mu\rho)^{n-1} \leq \frac{\abs{\und{F}}_\rho^T}{8} $ 
 choosing $\mu \rho= 1/A =\rho/e^2 32 S$.
\end{proof}

Next we have closedness of the
classes $\cA_{w^0, \rho}^N$ and $\sT_{w^0,w^1, \rho}^{w^2, N}$ under different operations. 
\begin{lemma}
\label{FGinA}
Fix $1 \leq p \leq 2$, weights  $w^0 \leq w^1\leq w^2$, $\N \ni N\geq 2$ and  let $\mu$ be as in \eqref{mu.def}. Then the following holds true:
\begin{itemize}
	\item[(i)] 
	If $F \in \sT_{w^0,w^1,\rho}^{w^2, N}$ and
          $G\in\sT_{w^0,w^1,\mu\rho}^{w^2, N }$
           with
          $\norm{G}_{\sT_{w^0,w^1,\mu\rho}^{w^2, N}}<\tfrac{\rho\mu}{e}$, then  
          $H:= F(\zeta + G(\zeta))\in \sT_{w^0,w^1, \mu\rho}^{w^2, N}$ and
$$
\norm{H}_{\sT_{w^0,w^1, \mu\rho}^{w^2, N}} \leq 2 \norm{F}_{\sT_{w^0,w^1, \rho}^{w^2, N}} \ .
$$
\item[(ii)] 
If  $F \in \sT_{w^0,w^1, \rho}^{w^2, N}$ and $\|F\|_{\sT_{w^0,w^1,\rho}^{w^2, N}}\leq\rho/e$, then $(\uno + F)^{-1} = \uno + G$ with  $G\in \sT_{w^0,w^1,\mu\rho}^{w^2, N}$ with 
\begin{equation}
\label{sti.inv.22}
\norm{G}_{\sT_{w^0,w^1,\mu\rho}^{w^2, N}}\leq 2\norm{F}_{\sT_{w^0,w^1,\rho}^{w^2, N}} \ .
\end{equation}
\item[(iii)] 
If  $F \in \sT_{w^0,w^1,\rho}^{w^2, N}$, then  $H(\zeta) := \di F(\zeta) \zeta$ is in the class $\sT_{w^0,w^1,\mu\rho}^{w^2, N}$ and
$$\norm{H}_{\sT_{w^0,w^1,\mu\rho}^{w^2, N}} \leq 2 \norm{F}_{\sT_{w^0,w^1,\rho}^{w^2, N}} \ . $$
\item[(iv)]
If   $F^0, G^0 \in \sT_{w^0,w^1,\rho}^{w^2, N}$ with $\norm{F^0}_{\sT_{w^0,w^1,\rho}^{w^2, N}} \leq \tfrac{\rho}{e}$, then $H_0:= \di G^0(\zeta)^*(F^0(\zeta)) \in \sT_{w^0,w^1,\rho/2}^{w^2, N}$
 and
 $$
  \norm{H_0}_{\sT_{w^0,w^1,\rho/2}^{w^2, N}} \leq  \norm{G^0}_{\sT_{w^0,w^1,\rho}^{w^2, N}}  \, \frac{4}{\rho}  \norm{F^0}_{\sT_{w^0,w^1,\rho/2}^{w^2, N}} \leq 2 \norm{G^0}_{\sT_{w^0,w^1,\rho}^{w^2, N}}.
  $$
  \item[(v)] If   $F^0, G^0 \in \sT_{w^0,w^1,\rho}^{w^2, N}$ with $\norm{F^0}_{\sT_{w^0,w^1,\rho}^{w^2, N}} \leq \tfrac{\rho}{e}$, then $H_0:= \di G^0(\zeta) F^0(\zeta) \in \sT_{w^0,w^1,\rho/2}^{w^2, N}$
 and
 $$
  \norm{H_0}_{\sT_{w^0,w^1,\rho/2}^{w^2, N}} \leq  \norm{G^0}_{\sT_{w^0,w^1,\rho}^{w^2, N}} \, \frac{4}{\rho} \norm{F^0}_{\sT_{w^0,w^1,\rho/2}^{w^2, N}} \leq 2 \norm{G^0}_{\sT_{w^0,w^1,\rho}^{w^2, N}}.
  $$
\end{itemize}
Finally all the results hold true replacing everywhere $\sT_{w^0,w^1,\rho}^{w^2, N}$ with  $\cA_{w^0,\rho}^{N}$.
\end{lemma}
\begin{proof} Again we prove only the tame estimates, since the 
results for the class $\cA_{w^0,\rho}^{N}$ are already proved in    \cite[Lemma A.3]{masp_toda}.
\begin{enumerate}
\item[$(i)$] 
One has $\uno + G \in \cN_{\mu\rho}^T(B^{p,w^0}\cap \spazio{p,w^1}, \spazio{p,w^1})$  with  $\tame{\uno + \und G}_{\mu\rho} \leq \rho$, then by Lemma \ref{FGinN}  
\begin{equation*}
\tame{\und{H}}_{\mu\rho} \leq \tame{\und{F}}_{\rho} \ . 
\end{equation*}
Now 
 $\di H(\zeta) = \di F(\zeta+ G(\zeta)) (\uno + \di G(\zeta))$  therefore  
 \begin{equation*}
 \und{\di H}(|\zeta|) \leq \und{\di F}(|\zeta| + \und{G}(|\zeta|)) + \und{\di F}(|\zeta| + \und{G}(|\zeta|)) \und{\di G}(|\zeta|) \ .
 \end{equation*}
Then exploiting Lemma \ref{FGinN} (iii) and Lemma \ref{rem:pro} 
 $$
 \tame{\und{\di H}}_{\mu\rho} \leq  \tame{\und{\di F}}_{\rho}(1 +  \tame{\und{\di G}}_{\mu\rho}) \ .
 $$
Finally the adjoint $\di H(\zeta)^* = (\uno + \di G(\zeta)^*) \di F(\zeta + G(\zeta))^*$, thus analogously one finds 
\begin{align*}
\tame{\und{\di H^*}}_{\mu\rho} \leq  \tame{\und{\di F^*}}_{\rho}(1 +  \tame{\und{\di G^*}}_{\mu\rho})\  .
\end{align*}
 Therefore since $\mu\rho(\tame{\und{\di G}}_{\mu\rho}+\tame{\und{\di G^*}}_{\mu\rho})\leq \norm{G}_{\sT_{w^0,w^1, \mu\rho}^{w^2, N}}  \leq \mu\rho/e$, one has
\begin{align*}
     \norm{H}_{\sT_{w^0,w^1, \mu\rho}^{w^2, N}} & \leq
  \tame{\und{ F}}_{\mu\rho} 
  +
\tame{\und{\di F}}_{\rho}(\mu\rho +  \mu\rho\tame{\und{\di G}}_{\mu\rho})
   +
    \tame{\und{\di F^*}}_{\rho}(\mu\rho +  \mu\rho\tame{\und{\di G^*}}_{\mu\rho}) \\
    & \leq 2 \norm{F}_{\sT_{w^0,w^1,\rho}^{w^2, N}}.
    \end{align*}
\item[$(ii)$] By Lemma \ref{Ginverso} we know that
$\tame{\und G}_{\mu \rho} \leq \frac{\tame{\und F}_\rho}{8}$. 
Differentiating the identity $G = F\circ(\uno - G)$ one gets 
$$
\di G(\zeta)= [\uno+ \di F(\zeta-G(\zeta))]^{-1} \di F(\zeta-G(\zeta)) , 
$$
thus by Lemma \ref{FGinN}(iii), Lemma \ref{rem:pro}
 and the assumption $\norm{F}_{\sT_{w^0,w^1, \rho}^{w^2, N}} \leq \rho/e $, it follows that 
$$
\tame{\und{\di G}}_{\mu \rho} \leq \tame{\und{\di F}}_{\rho} \, \sum_{k \geq 0}  \tame{\und{\di F}}_{\rho}^k \leq \frac{e}{e-1}\tame{\und{\di F}}_{\rho} \ .
$$
Analogously $\tame{\und{\di G^*}}_{\mu \rho} \leq \frac{e}{e-1}\tame{\und{\di F^*}}_{\rho} $.
Estimate \eqref{sti.inv.22} then follows.
\item[$(iii)$]  Clearly $\tame{\und H}_\rho \leq \tame{\und{\di F}}_\rho$. Then  $\di H(\zeta)\upsilon = \di F(\zeta)\upsilon + \di^2F(\zeta)(\upsilon,\zeta),$ and the claim follows  arguing as in item $(i)$ and using Cauchy estimate (see also the proof of $(iv)$).
\item[$(iv)$] 
Consider first $\und{H_0}(|\zeta|)$. By Lemma \ref{rem:pro}$(i)$, $\tame{\und{H_0}}_\rho \leq \tame{ \und{\di G^0}^* }_\rho \tame{ \und{F^0}}_\rho$.
Consider now $\di H_0(\zeta)\upsilon = \di G^0(\zeta)^* \di F^0(\zeta)\upsilon  + \di_\zeta (\di G^0(\zeta)^* U)\upsilon$, $U=F^0(\zeta)$. One has
$$
\und{\di H_0}(|\zeta|)|\upsilon| \leq \und{\di G^0}(|\zeta|)^* \und{\di F^0}(|\zeta|)|\upsilon|  + \di_{|\zeta|} (\und{\di G^0}(|\zeta|)^* \und U)|\upsilon|  \ , \qquad \und{U} = \und F^0(|\zeta|) \ .
$$
Using also the Cauchy estimates \eqref{diff.2},
one has
  $\tame{\und{\di H_0}}_{\rho/2} \leq \frac{2}{\rho} \tame{\und{\di G^0}^*}_\rho \norm{F^0}_{\sT_{w^0,w^1,\rho/2}^{w^2, N}}$.\\
Finally  in order to estimate $\und{\di H_0}(|\zeta|)^*$ remark that (see \cite[Lemma 3.6]{kuksinperelman}) 
$
\di H_0(\zeta)^*\upsilon = \di F^0(\zeta)^*  \di G^0(\zeta)\upsilon + \di_\zeta (\di G^0(\zeta)^* U) \upsilon, 
$
thus
$$
\und{\di H_0}(|\zeta|)^*|\upsilon| \leq \und{\di F^0}(|\zeta|)^*  \und{\di G^0}(|\zeta|)|\upsilon| + \di_{|\zeta|} (\und{\di G^0}(|\zeta|)^*\und{U} ) \, |\upsilon| \ .
$$
The $\tame{\cdot}_{\rho/2} $ norm of first term in the r.h.s. is estimated by $\tame{\und{\di F^0}^*}_{\rho/2} \tame{\und{\di G^0}}_{\rho/2}$. To estimate the second term we use Cauchy formula \eqref{cauchy}, obtaining that its $\tame{\cdot}_{\rho/2} $ norm  is controlled by $\frac{2}{\rho}\tame{\und{\di G^0}^*}_\rho \tame{\und F^0}_{\rho/2}$.
  The claim follows.
\item[$(v)$]  The proof is similar to $(iv)$, and we skip it.
\end{enumerate}\end{proof}

Finally  we analyze the flow generated by a vector field of class
$\sT_{w^0,w^1,\rho}^{w^2, N}$.  Given a time dependent vector field $Y_t(v)$,
consider the differential equation
\begin{equation}
\label{ode}
\begin{cases}
\dot{u}(t)=Y_t(u(t))\\
u(0)=\zeta \in \spazio{p,w^0} 
\end{cases} \ . 
\end{equation}
We will denote by $\phi^t(\zeta)$ the corresponding flow map whose
existence and properties are given in the next lemma. 

\begin{lemma}
\label{flussoinA}
 Assume that  $\left[0,1\right]\ni t\mapsto
Y_t\in\sT_{w^0,w^1,\rho}^{w^2, N} $ is continuous and $\sup_{t \in
  [0,1]}\norm{Y_t}_{\sT_{w^0,w^1,\rho}^{w^2, N}}\leq \rho/e$;  then for each
$t\in \left[0,1\right]$, $\phi^t-\uno \in \sT_{w^0,w^1,\mu\rho}^{w^2, N}$ and
\begin{equation}
\label{sti.flu}
\norm{\phi^t-\uno}_{\sT_{w^0,w^1,\mu\rho}^{w^2, N}}\leq
2\sup_{t\in[0,1]}\norm{Y_t}_{\sT_{w^0,w^1,\rho}^{w^2, N}}\ .
\end{equation}
The same holds true if the class  $\sT_{w^0,w^1,\mu\rho}^{w^2, N}$ is replaced everywhere  by $\cA_{w^0,\mu\rho}^{N}$.
\end{lemma}
\begin{proof}
The claim for the class $\cA_{w^0,\mu\rho}^{N}$  follows as in \cite{masp_toda}, thus we consider only the tame class.\\ 
We look for a solution $u(t,\zeta)= \sum_{j \geq 1} u^j(t,\zeta)$ in power
series of $\zeta$, with $u^j(t,\zeta)$ a homogeneous polynomial of degree $j$ in $\zeta$. Expanding the vector field
$Y_t(\zeta)=\sum_{r \geq N}Y_t^r(\zeta)$ in Taylor series, one obtains the recursive formula for the solution
\begin{equation}
\label{ric.1}
u^1(t,\zeta)=\zeta,\qquad
u^n(t,\zeta)=\sum_{r=2}^{n}\sum_{k_1+\cdots+k_r=n}\int_0^t{\tilde{Y}^r_s(u^{k_1}(s,\zeta),\ldots,u^{k_r}(s,\zeta)) \;ds}\qquad
\forall n\geq 2,
\end{equation}
where $\tilde{Y}^r_s$ is the multilinear map associated to $Y^r_s$ (see \eqref{polin}).
Arguing as in the proof of \eqref{Ginverso} one gets that $u^n(t, \zeta) =0$ if $1 < n < N$, while 
\begin{align}
\label{ode.rec.est1}
&\norm{\underline u^n(t,\zeta)}_{p, w^2} \leq \frac{\sup_{t \in
    [0,1]}\abs{\und {Y_t}}_\rho^T}{8Sn^2}A^{n-1}\norm{\zeta}_{p, w^0}^{n-1}\norm{\zeta}_{p, w^1} \qquad \forall n\geq
N, \\
\label{ode.rec.est2}
& \norm{\underline u^n(t,\zeta)}_{p, w^0} \leq \frac{\sup_{t \in
    [0,1]}\abs{\und {Y_t}}_\rho}{8Sn^2}A^n\norm{\zeta}_{p,w^0}^n \qquad \forall n\geq
N,
\end{align} 
with $A= \tfrac{e^2}{\rho} 32 S$,  
from which it follows that $\tame{\und{\phi^t-\uno}}_{\mu\rho} \leq \sup_{t \in [0,1]}\tame{\und{Y_t}}_\rho/8$.
		We come to the estimate of the
differential of $u(t,\zeta)$ and of its adjoint.
To do this, remark that $\di u(t, \zeta)\upsilon$ is the solution of the linearized equation 
$$
\dot w(t)  = \di Y_t(u(t, \zeta)) w(t) \ , \qquad w(0)= \upsilon \ ,
$$
whose solution can be written by Picard iteration as
$$
\di u(t, \zeta)\upsilon = \upsilon + \sum_{n=1}^\infty
\int\limits_0^{t}\cdots \int\limits_0^{t_{n-1}} \di Y_{t_1}(u(t_1, \zeta)) \cdots \di Y_{t_n}(u(t_n, \zeta))\upsilon \di t_n \ldots \di t_1 \ .
$$
The series is absolutely and uniformly convergent for $\norm{\zeta}_{p,w^0}$ sufficiently small; moreover
\begin{equation}
\label{ric.20}
(\und{\di u}(t, |\zeta|) - \uno )|\upsilon| \leq \sum_{n=1}^\infty
\int\limits_0^{t}\cdots \int\limits_0^{t_{n-1}} \und{\di Y_{t_1}}(\und{u}(t_1, |\zeta|)) \cdots \und{\di Y_{t_n}}(\und u(t_n, |\zeta|))|\upsilon| \, \di t_n \ldots \di t_1 \ ,
\end{equation}
so one has
$$
\tame{\und{\di u}-\uno}_{\mu\rho}\leq  \sum_{n=1}^\infty \left(\sup_{t \in [0,1]}\tame{\und{\di Y_t}}_\rho\right)^n \leq  \frac{e}{e-1}\sup_{t \in [0,1]}\tame{\und{\di Y_t}}_\rho
$$
 Finally one has to
estimate $[\underline{\di u^n}]^*$, but this is done by simply taking the adjoint of \eqref{ric.20} and estimate the r.h.s. using the bounds on $\sup_{t \in [0,1]}\tame{\und{\di Y_t^*}}_\rho$.
\end{proof}

\subsection{Proof Theorem \ref{KP}}

The  map $\wt\Psi$ of Theorem \ref{KP} 
will be constructed in  two steps  based on the 
Darboux theorem. Such a theorem states that in order to construct a
coordinate transformation $\psi$ transforming the closed
nondegenerate form $\Omega_1$ into a closed nondegenerate form
$\Omega_0$, then it is convenient to look for $\psi$ as the time 1
flow $\psi^t$ of a time-dependent vector field $Y^t$. To construct
$Y^t$ one defines $\Omega_t:= \Omega_0 + t(\Omega_1 - \Omega_0)$ and
imposes that
$$
0 = \left.\tfrac{d}{d t}\right|_{t=0} \psi^{t*} \Omega_t = \psi^{t*} \left(\cL_{Y^t}\Omega_t + \Omega_1 - \Omega_0 \right) = \psi^{t*} \left( \di (Y^t\lrcorner \Omega_t )+ \di (\alpha_1 - \alpha_0) \right)
$$
where $\alpha_1, \alpha_0$ are potential forms for $\Omega_1$ and $\Omega_0$ (namely $\di\alpha_i = \Omega_i$, $i=0,1$) and $\cL_{Y^t}$ is the Lie derivative of $Y^t$. Then one gets
\begin{equation}
\label{darboux.eq}
Y^t\lrcorner \Omega_t  + \alpha_1 - \alpha_0 = \di f
\end{equation}
for each $f$ smooth; then, if $\Omega_t$ is nondegenerate, this
defines $Y^t$. If $Y^t$ generates a flow $\psi^t$ defined up to
time 1, the map $\psi:= \left. \psi^t\right|_{t=1}$ satisfies
$\psi^*\Omega_1 = \Omega_0$ (see also \cite{masp_sol} for another application of Darboux theorem in a different model).

{\bf 
We prove the Kuksin-Perelman theorem  only in the case $\Psi^0 \in \sT_{w^0, w^1, \rho}^{w^2, N}$. In case  $\Psi^0 \in \cA_{w^0, \rho}^{N}$, then it is sufficient to replace in the following the class $\sT_{w^0, w^1, \rho}^{w^2, N}$ with the class $\cA_{w^0, \rho}^{N}$ (thanks to the results of Section \ref{sec:prop}).
} Actually the proof is exactly as the one of \cite{masp_toda}; therefore we only state the main lemmas with the new classes $\sT_{w^0, w^1, \rho}^{w^2, N}$,  while for the proofs we refer the reader to the corresponding  proofs in  \cite{masp_toda}.

From now on we will work always in the real subspace $\spazior{p,w}$, so abusing notation, we will write just $\xi \equiv (\xi, \bar \xi)$. 
Correspondingly, for a map $F:B^{p,w}_r(\rho) \to \spazior{p,w}$  we will write simply $\upsilon = F(\xi)$ rather than $(\upsilon, \bar \upsilon) \equiv F(\xi, \bar \xi)$. Furthermore  we fix a real $1 \leq p \leq 2$ and weights $w^0 \leq w^1 \leq w^2$. 
We will  meet maps in  $\sT_{w^0, w^1, \rho}^{w^2, N}$; to simplify notation, we will write only $\sT_\rho$  if nothing else is specified.\\

We recall the setup of \cite{masp_toda}. A non-constant symplectic form $\Omega$
is represented through a
linear skew-symmetric invertible operator $E_\Omega$ as follows:
\begin{equation}
\label{symp.1}
\Omega(\upsilon)(\xi^{(1)};\xi^{(2)})=\langle E_\Omega(\upsilon) \xi^{(1)}; \xi^{(2)} \rangle\ ,\quad \forall
\xi^{(1)}, \xi^{(2)}\in T_\upsilon\spazior{p,w}\simeq\spazior{p,w} , \ 
\forall \upsilon \in \spazior{p,w}.
\end{equation}
Here the scalar product is the one defined in \eqref{scalar_productR}.
We denote by $\{F,G\}_{\Omega}$ the
Poisson bracket with respect to $\Omega$ : $\{F, G\}_\Omega:=
\la \nabla F, J_\Omega \nabla G\ra$, $J_\Omega:=E_\Omega^{-1}$. 
Similarly we will represent   $1$-forms through the vector field $A$ such that
\begin{equation}
\label{1form.1}
\alpha(\upsilon)(\xi) = \langle A(\upsilon), \xi \rangle, \quad \forall \upsilon \in T_\upsilon\spazior{p,w}, \ \forall \upsilon \in \spazior{p,w} \ .
\end{equation}
Consider the Hamiltonian vector fields $X^0_{-I_l}$ of
the functions $I_l\equiv \frac{|\xi_l|^2}{2}$ through the symplectic form $\omega_0$; they are
given by
\begin{equation}
[X^0_{-I_l}(\xi)]_j= \delta_{l,j} \,  \im \xi_l ,  \quad \forall\, l,j \in \Z \ , \ \forall \xi \in \spazior{p,w} \ . 
\label{X_I_l}
\end{equation}
  For every $l \in \Z $ the corresponding flow
$\phi_l^t \equiv \phi_{X^0_{-I_l}}^t$ is given  by
$$ \phi^t_l(\xi) = \left( \cdots, \xi_{l-1}, e^{\im t}\xi_l, \xi_{l+1},
  \cdots \right) \ .$$ 
Remark that the map $\phi_l^t$ is linear in $\xi$, $2\pi$ periodic in $t$
and its adjoint satisfies
 $(\phi_l^t)^*=\phi_l^{-t}$.\\
Given a $k$-form $\alpha$ on $\spazior{p,w}$ $(k\geq0)$, we define its average by 
\begin{align}
\label{ave.1}
M_l\alpha(\xi)=\frac{1}{2\pi}\int_0^{2\pi}{((\phi_l^t)^*\alpha)(\xi)dt},
\quad l \in \Z \ , \qquad \mbox{and} \qquad
M\alpha(\xi)=\int_{\T^\infty}[(\phi^\theta)^*\alpha]\, d\theta
\end{align}
where $\T^\infty$ is the infinite dimensional torus, the map
$\phi^\theta=(\cdots \circ \phi_{-1}^{\theta_{-1}} \circ\phi_0^{\theta_0}\circ \phi_2^{\theta_2}\cdots)$ and
$d\theta$ is the Haar measure on $\T^\infty$.

\begin{remark}
\label{ave.J}
In the particular cases of 1 and 2-forms it is useful to compute the
average in term of the representations \eqref{symp.1} and
\eqref{1form.1}. Thus, for $\upsilon$, $\xi^{(1)}, \xi^{(2)} \in
\spazior{p,w}$, if
$$ \alpha(\upsilon)\xi^{(1)}=\langle A(\upsilon); \xi^{(1)}\rangle\ ,\quad
\omega(\upsilon)(\xi^{(1)},\xi^{(2)})=\langle E(\upsilon)\,\xi^{(1)};\xi^{(2)} \rangle\ ,
$$
one has 
\begin{equation}
(M\alpha)(\upsilon)\xi^{(1)}=\langle (MA)(\upsilon);\, \xi^{(1)}\rangle \ ,\quad \text{with}\quad
MA (\upsilon)=\int_{\T^\infty}{\phi^{-\theta}
  A(\phi^\theta(\upsilon))} 
\;d\theta 
\label{MA}
\end{equation}
and
\begin{equation}
(M\omega)(\upsilon)(\xi^{(1)},\xi^{(2)})=\langle (ME)(\upsilon)\xi^{(1)};\,\xi^{(2)}
\rangle\ ,\ \text{with}\ ME(\upsilon)=\int_{\T^\infty}{\phi^{-\theta}
  E(\phi^\theta(\upsilon))\phi^\theta} 
\; d\theta .
\label{ME}
\end{equation}
\end{remark}
\begin{remark}
\label{rem.M}
 The operator $M$ commutes with the differential operator $\di$ and the rotations $\phi^\theta$. In particular  $MA(\upsilon)$ and $ME(\upsilon)$ as in  \eqref{MA}, \eqref{ME} satisfy 
$$
\phi^\theta MA(\upsilon) = MA(\phi^\theta \upsilon), \quad 
\phi^\theta ME(\upsilon)\xi = ME(\phi^\theta \upsilon)\phi^\theta \xi, \qquad \forall \, \theta \in \T^\infty  \ .
$$
\end{remark}

\begin{remark} By condition \eqref{th.1} and \eqref{mu.def}, one has
\begin{equation}
\label{eps.mu}
\epsilon_1 < \mu^6 \rho \ . 
\end{equation}
\end{remark}
Define $\omega_1:=(\Psi^{-1})^*\omega_0$, and let $E_{\omega_1}$ be the  operator representing the symplectic form $\omega_1$. 
\begin{lemma}
\label{lem.analy} Let $\Phi:=\Psi^{-1}$ and  $\omega_1$ be as above.
Assume that 
$\epsilon_1\leq \rho/e$. Then the following holds:
\begin{itemize}
\item[(i)] $E_{\omega_1}=\im +\Upsilon_{\omega_1}$, with $\Upsilon_{\omega_1}\in\cN_{\mu\rho}^T(B^{p,w^0}\cap\spazior{p,w^1}, \cL(\spazior{p,w^1}, \spazior{p,w^2}))$  and
\begin{equation}
\label{JM}
\tame{\und{\Upsilon_{\omega_1}}}_{\mu\rho}\leq \frac{8 \epsilon_1}{\mu\rho}\ .
\end{equation}
Furthermore $\Upsilon_{\omega_1}$ is antisymmetric, $\Upsilon_{\omega_1}(\xi)^* = - \Upsilon_{\omega_1}(\xi)$.
\item[(ii)] Define 
\begin{equation}
\label{WM}
W_{\omega_1}(\xi):=\int_0^1 \Upsilon_{\omega_1}(t\xi)t\xi \;dt \ ,
\end{equation}
then $W_{\omega_1}\in\sT_{\mu^3\rho}$ and $\norm{W_{\omega_1}}_{\sT_{\mu^3\rho}}\leq 8 \epsilon_1 $. Moreover the 1-form $\alpha_{W_{\omega_1}}:=\langle W_{\omega_1}; .\rangle$ satisfies
$
\di \alpha_{W_{\omega_1}}=\omega_1-\omega_0\ .
$
\end{itemize}
\end{lemma}
\begin{proof}
The proof follows  \cite[Lemma 2.16]{masp_toda}, using the results of Lemma \ref{FGinA} and \ref{FGinN}.
\end{proof}

\begin{remark}
\label{rem.me}
One has $M\alpha_{\omega_1} - \alpha_0 = M\alpha_{W_{\omega_1}} = \la M W_{\omega_1}, \cdot \ra$ and  $\norm{MW_{\omega_1}}_{\sT_{\mu^3\rho}}
\leq\norm{W_{\omega_1}}_{\sT_{\mu^3\rho}} $.
\end{remark}

We are ready now for the first step.
\begin{lemma}There exists a map $\hat \psi: \Br^{{p,w^0}}(\mu^5 \rho) \to \spazior{p,w^0}$ such that $(\uno-\hat\psi)\in\sT_{\mu^5\rho} $ and
\begin{equation}
\label{hat.sti.3}
\norm{\uno-\hat\psi}_{\sT_{\mu^5\rho}}\leq 2^5\epsilon_1\ .
\end{equation}
Moreover $\hat{\psi}$ satisfies the following properties:
\begin{enumerate}
\item[(i)] $\hat \psi$ commutes with the rotations $\phi^\theta$, namely $\phi^\theta \hat\psi(\xi) = \hat\psi(\phi^\theta \xi)$ for every $\theta \in \T^\infty $.
\item[(ii)] Denote $\hat{\omega}_1 := \hat{\psi}^* \omega_1$, then  $M \hat{\omega}_1 = \omega_0$. 
\end{enumerate} 
\label{map.phi1}
\end{lemma}
\begin{proof}
It follows  as in \cite[Lemma 2.18]{masp_toda}. We apply the  Darboux procedure described at the beginning of this section with $\Omega_0 = \omega_0$ and $\Omega_1 = M\omega_1$. Then $\Omega_t$ is represented by the operator $\hat{E}_{\omega_1}^t := \left(\im + t(ME_{\omega_1} -\im )\right) $. Write equation \eqref{darboux.eq}, with $f \equiv 0$,  in terms of the operators defining the symplectic forms, getting the equation 
$\hat{E}_{\omega_1}^t \hat Y^t = - MW_{\omega_1}$ (see also Remark \ref{rem.me}). This equation can be solved by inverting the operator $\hat{E}_{\omega_1}^t$ by Neumann series:
\begin{equation}
\label{equazione3.1}
\hat Y^t:=-(\im +tM\Upsilon_{\omega_1})^{-1}
MW_{\omega_1}\ .
\end{equation}
In order to estimate it, we expand the r.h.s. of \eqref{equazione3.1} in Neumann series and estimate each piece. First note that for any $G \in \sT_{\mu^3 \rho}$,  by Lemma \ref{FGinA}, one has
\begin{align*}
\norm{\Upsilon_{\omega_1} G}_{\sT_{\mu^3\rho/2}}
 \leq 
\frac{1}{2}\norm{G}_{\sT_{\mu^3 \rho/2}}  \ ;
\end{align*}
then by induction one has that $\norm{[\Upsilon_{\omega_1}]^k MW_{\omega_1}}_{\sT_{\mu^3\rho/2}} \leq 2^{-k}\norm{MW_{\omega_1}}_{\sT_{\mu^3 \rho}} $.
Therefore the Neumann series converges, $\hat{Y}^t$ is of class $\sT_{\mu^4\rho}$ and fulfills
\begin{equation}
\label{sti.hat.1}
\sup_{t \in [0,1]}\norm{\hat Y^t}_{\sT_{\mu^4\rho}}\leq 2 \norm{MW_{\omega_1}}_{\sT_{\mu^3\rho}} \leq 2^4\epsilon_1 \ .
\end{equation}
By Lemma \ref{flussoinA} the vector field $\hat Y^t$ generates a flow $\hat{\psi}^t: \Br^{{p, w^0}}(\mu^5 \rho) \to \spazior{p,w^0} $ such that $\hat{\psi}^t- \uno $ is of class $\sT_{\mu^5\rho}$ and satisfies 
$$
\norm{\hat \psi^t - \uno}_{\sT_{\mu^5\rho}} \leq 
2 \sup_{t \in [0,1]}\norm{\hat Y^t}_{\sT_{\mu^4\rho}}\leq 2^5\epsilon_1 .$$
Therefore the map $\hat \psi \equiv \hat \psi^t\vert_{t=1}$ exists, satisfies the claimed estimate \eqref{hat.sti.3} and furthermore $\hat \psi^* M\omega_1 = \omega_0$.
\newline
Item $(i)$ is a geometric property and it follows exactly as in 
\cite[Lemma 2.18]{masp_toda}.
\end{proof}

The analytic properties of the symplectic form  $\hat \omega_1$ can be studied in the same way as in  Lemma \ref{lem.analy}; we get therefore the following corollary:
\begin{corollary}
\label{cor.e}
Denote by $E_{\hat{\omega}_1}$ the symplectic operator describing $\hat{\omega}_1=\hat\psi ^*\omega_1$. Then 
\begin{itemize}
\item[(i)] $E_{\hat{\omega}_1}= \im +\Upsilon_{\hat{\omega}_1}$, with $\Upsilon_{\hat{\omega}_1}\in
\cN_{\mu^5\rho}^T(B^{p,w^0}\cap \spazior{p,w^1}, \cL(\spazior{p,w^1}, \spazior{p,w^2}))$ and
$\tame{\und{\Upsilon_{\hat{\omega}_1}}}_{\mu^5\rho}\leq 2^7\frac{\epsilon_1}{\mu\rho}.$
\item[(ii)] Define 
$W(\xi):=\int_0^1 \Upsilon_{\hat{\omega}_1}(t\xi)t\xi \;dt$, then $W\in\sT_{\mu^7\rho}$ and
$\norm{W}_{\sT_{\mu^7\rho}}\leq 2^7\epsilon_1 .$
\end{itemize}
Furthermore the 1-form $\alpha_{W}:= \langle W, . \rangle $ satisfies 
 $\di \alpha_{W} = \hat{\omega}_1 - \omega_0.$
\end{corollary}
Finally we will need also some analytic and geometric properties of the map 
\begin{equation}
\label{checkPsi}
\check{\Psi}:= \hat{\psi}^{-1} \circ \Psi.
\end{equation}
The functions $\{\check{\Psi}(\xi)\}_{j \in \Z}$ forms a new set of coordinates in a suitable neighborhood of the origin whose properties are given by the following corollary:
\begin{corollary}
\label{cor.psi.1}
The map $\check{\Psi}:\Br^{{p,w^0}}(\mu^8\rho) \to \spazior{p,w^0}$, defined in \eqref{checkPsi},  satisfies the following properties:
\begin{itemize} 
\item[(i)] $\di \check{\Psi}(0) = \uno$ and $\check\Psi^0:= \check \Psi- \uno \in \sT_{\mu^8 \rho}$ with $ \norm{\check \Psi^0}_{\sT_{\mu^8 \rho}} \leq 2^8 \epsilon_1$.
\item[(ii)]  $\cT^{(0)} = \check{\Psi}(\cT)$, namely  the foliation defined by  $\check{\Psi}$ coincides with the foliation defined by $\Psi$.
\item[(iii)] The functionals $\lbrace \frac{1}{2}\left|\check{\Psi}_j\right|^2\rbrace_{ j\in \Z}$ pairwise commute with respect to the symplectic form $\omega_0$.
\end{itemize}
\end{corollary}
\begin{proof}
 As in \cite[Corollary 2.20]{masp_toda}.
\end{proof}

\vspace{1em} The second step consists in transforming $\hat{\omega}_1$
into the symplectic form $\omega_0$ while preserving the functions
$I_l$. In order to perform this transformation, we apply once more the
Darboux procedure with $\Omega_1 = \hat{\omega}_1$ and $\Omega_0 =
\omega_0$. However, we require each leaf of the foliation to be
invariant under the transformation. In practice, we look for a change
of coordinates $\psi$ satisfying
\begin{align}
	&\psi^*\Omega_1=\Omega_0 \ , \label{fi.1}\\
	&I_l(\psi(\xi))=I_l(\xi), \quad \forall \, l \in \Z  \ . \label{fi.2}
\end{align}
In order to fulfill the second equation, we take advantage of the arbitrariness of $f$ in equation \eqref{darboux.eq}. 
It turns out that if $f$ satisfies the set of differential equations given by
 \begin{equation}
\label{fi.3}
\di f(X^0_{-I_l})-(\alpha_1-\alpha_0)(X^0_{-I_l})=0, \quad \forall \, l \in \Z \ ,
\end{equation}
then equation \eqref{fi.2} is satisfied (as it will be proved below).
Here $\alpha_1$ is the potential form of $\hat\omega_1$ and is given by
$\alpha_1 := \alpha_0 + \alpha_W$, 
where $\alpha_W$ is defined in Corollary
\ref{cor.e}.  However,  \eqref{fi.3} is essentially a system of equations for the potential of a 1-form on a torus, so there is a solvability condition. In Lemma \ref{moser} below we will prove that the system \eqref{fi.3} has a solution if the following conditions are satisfied:
\begin{align}
\label{comp.1}
& \di (\alpha_1-\alpha_0)\arrowvert_{T\cT^{(0)}}=0 \ ,
\\
\label{comp.2}
& M(\alpha_1-\alpha_0)\arrowvert_{T\cT^{(0)}}=0\ .
\end{align}

In order to show that these two conditions are fulfilled, we need a preliminary result.
First, for
$\xi\in \spazior{p,w}$ fixed, define the symplectic orthogonal of $T_\xi\cT^{(0)}$ with respect to the form $\omega^t := \omega_0 + t(\hat\omega_1 - \omega_0) $ by
\begin{equation}
\label{fol.1}
(T_\xi)^{\angle_t}:=\left\{h \in\spazior{p,w}:
\omega^t(\xi)(u,h)=0 \ \ \  \forall u \in T_\xi\cT^{(0)}\right\}.
\end{equation}

\begin{lemma}
\label{tv}
(i) For  $\xi \in \Br^{{p,w^0}}(\mu^5 \rho)$, one has $T_\xi\cT^{(0)} = (T_\xi\cT^{(0)})^{\angle_t}$.\\
(ii) The solvability conditions   \eqref{comp.1}, \eqref{comp.2}
are fulfilled. 
\end{lemma}
\begin{proof} 
It follows with the same arguments of  \cite[Lemma 2.21, Lemma 2.22]{masp_toda}.
\end{proof}
We show now that the system  \eqref{fi.3} can be solved and its solution has good analytic properties:
\begin{lemma}
\label{moser}
 If conditions \eqref{comp.1} and \eqref{comp.2} are fulfilled, then equation \eqref{fi.3} has a solution $f$. Moreover, denoting $h_j:=(\alpha_1-\alpha_0)(X^0_{-I_j})$, the solution $f$  is given by the explicit formula $f(\xi)=\sum_{k= 0}^\infty f_k(\xi) $, 
\begin{equation}
\begin{aligned}
f_0(\xi) = L_0 h_0  \ , \qquad 
& f_{2i-1}(\xi) = M_0 \prod_{\ell =1}^{i-1} (M_\ell M_{-\ell}) L_i h_i  \ , \qquad  f_{2i}(\xi) = M_0 \prod_{\ell =1}^{i-1} (M_\ell M_{-\ell}) M_{i} L_{-i}h_{-i}  \ ,
\end{aligned}
\label{f.formula}
\end{equation}
where
$$
L_jg(\xi)=\frac{1}{2\pi}\int_0^{2\pi}{tg(\phi^t_j(\xi))dt} \ , \qquad \forall j \in \Z \ .
$$
Finally $f \in \cN_{\mu^7\rho}(\spazior{p,w^1}, \C)$, 
$\nabla f\in \cN_{\mu^7\rho}^T(B^{p,w^0}\cap\spazior{p,w^1}, \spazior{p,w^2})$, $f = O(\xi^{N+1})$ and
\begin{equation}
\label{sti.nabla.f}
\abs{\und{f}}_{\mu^7\rho} \leq 2^{10}  \epsilon_1 \mu^7 \rho, \qquad 
\tame{\und{\nabla f}}_{\mu^7\rho}\leq 2^{11} \epsilon_1\ .
\end{equation}
\end{lemma}
\begin{proof}
Denote by $\theta_j$ the time along the flow generated by $X^0_{-I_j}$,
then one has
$\di g(X^0_{-I_j}) = \frac{\partial g}{\partial \theta_j}\ ,$
so that the
equations to be solved take the form 
\begin{equation}
\label{eqf}
\frac{\partial f}{\partial \theta_j}=h_j, \qquad \forall j\in \Z \ . 
\end{equation}
Clearly $\frac{\partial}{\partial\theta_j}M_jh_j=0 $, and by \eqref{comp.1} it follows that 
$$ 
\frac{\partial}{\partial\theta_l}M_jh_j = M_j \frac{\partial h_j}{\partial\theta_l} 
=M_j \frac{\partial h_l}{\partial\theta_j} = \frac{\partial}{\partial\theta_j}M_jh_l =0, \qquad \forall l, j \in \Z \ , $$
which shows that $M_jh_j$ is independent of all the $\theta$'s, thus $M_j h_j = M h_j$. Furthermore,  by \eqref{comp.2} one has $ M h_j = 0, \,$ $\forall \, j \in \Z$.
Now, using that $\frac{\partial}{\partial\theta_j}L_j g = L_j \frac{\partial}{\partial\theta_j} g = g - M_j g $,  one verifies that  $f_i$ defined in \eqref{f.formula} satisfies
\begin{align*}
&\frac{\partial f_0}{\partial \theta_l}= h_l - M_0 h_l \ , \\
&\frac{\partial f_{2i-1}}{\partial \theta_l}=
\begin{cases}
0 & \text{if}\ |l|< i  
\\
M_0 (\prod_{\ell = 1}^{i-1} M_{\ell} M_{-\ell})  h_l & \text{if}\ l=i 
\\
M_0 (\prod_{\ell = 1}^{i-1} M_{\ell} M_{-\ell}) h_l-M_0 (\prod_{\ell = 1}^{i-1} M_{\ell} M_{-\ell})M_i h_l  & {\rm otherwise}
\end{cases} \ , \\
&\frac{\partial f_{2i}}{\partial \theta_l}=
\begin{cases}
0 & \text{if}\ |l|< i \mbox{ or } l=i 
\\
M_0 (\prod_{\ell = 1}^{i-1} M_{\ell} M_{-\ell}) M_i  h_{l} & \text{if}\ l=-i 
\\
M_0 \prod_{\ell =1}^{i-1} (M_\ell M_{-\ell}) M_{i} h_{l} -
M_0 \prod_{\ell =1}^{i} (M_\ell M_{-\ell})  h_{l}   & \text{otherwise}
\end{cases}\ .
\end{align*}
 Thus the series $f(\xi)=\sum_{i \geq 0} f_i(\xi)$, if convergent,  satisfies \eqref{eqf}. \\
We prove now the convergence of the series for $f$ and $\nabla f$.  First we define, for $\theta \in \T^\infty $,
$$\Theta_{2i-1}^\theta:=\phi_0^{\theta_0}\prod_{\ell=1}^{i-1}(\phi_\ell^{\theta_\ell}\phi_{-\ell}^{\theta_{-\ell}}) \phi_i^{\theta_i} \qquad \ , \qquad 
\Theta_{2i}^\theta:=\phi_0^{\theta_0}\prod_{\ell=1}^{i}(\phi_\ell^{\theta_\ell}\phi_{-\ell}^{\theta_{-\ell}})  
\qquad  \forall\, i \geq 0 \ ,$$
 then by \eqref{f.formula} one has 
\begin{align}
\label{f.g1}
f_{2i-1}(\xi)=\int_{\T^{2i}}\theta_i  h_i(\Theta_{2i-1}^\theta \xi)\;d\theta^{2i} \ , \qquad f_{2i}(\xi)=\int_{\T^{2i+1}}\theta_{-i}  h_{-i}(\Theta_{2i}^\theta \xi)\;d\theta^{2i+1} 
\\
\label{f.g2}
\nabla f_{2i-1}(\xi)=\int_{\T^{2i}}\Theta_{2i-1}^{-\theta} \theta_i\nabla
h_i(\Theta_{2i-1}^\theta \xi)\;d\theta^{2i}\ ,
\qquad
\nabla f_{2i}(\xi)=\int_{\T^{2i+1}}\Theta_{2i}^{-\theta} \theta_{-i}\nabla
h_{-i}(\Theta_{2i}^\theta \xi)\;d\theta^{2i+1}
\end{align}
where $\T^{2i}$ (respectively $\T^{2i+1}$) is the $2i$-dimensional torus ($2i+1$-dimensional) and the measure 
$d\theta^{2i} := \frac{d\theta_0}{2\pi}
(\prod_{\ell=1}^{i-1} \frac{d\theta_\ell}{2\pi} \frac{d\theta_{-\ell}}{2\pi}) \frac{d\theta_i}{2\pi}$ 
($d\theta^{2i+1} = \frac{d\theta_0}{2\pi}(\prod_{\ell=1}^{i} \frac{d\theta_\ell}{2\pi} \frac{d\theta_{-\ell}}{2\pi}) $).
Now, using that 
$$
h_j(\xi)=\langle W(\xi), X^0_{-I_j}(\xi) \rangle = Re(-\im W_j(\xi)\bar \xi_j) \ ,  \qquad \forall \, j \in \Z
$$ 
 one gets that $\forall i \geq 0$
$$
\und{f_{2i-1}}(|\xi|) \leq 2\pi\, \underline{h_i}(|\xi|) \leq 2\pi\, \underline{W_i}(|\xi|) |\xi_i| \  , \qquad \und{f_{2i}}(|\xi|) \leq 2\pi\, \underline{h_{-i}}(|\xi|) \leq 2\pi\, \underline{W_{-i}}(|\xi|) |\xi_{-i}|  \ , \qquad \forall i \geq 0
$$
therefore for every $1 \leq p \leq 2$ one has 
$$
\und{f}(|\xi|) \leq \sum_{k=0}^\infty \und{f_k}(|\xi|) 
\leq 2\pi \, \norm{\und{W}(|\xi|)}_{2}  \norm{\xi}_{2}
\leq 2\pi \, \norm{\und{W}(|\xi|)}_{p,w^0}  \norm{\xi}_{p,w^0}
$$
 and it follows that $\abs{\und{f}}_{\mu^7\rho} \leq 2\pi\, \abs{\und{W}}_{\mu^7\rho} \mu^7 \rho.$
This  proves the convergence of the series defining  $f$.\\
Consider now the gradient of $h_i$, whose $k^{th}$ component is given by
$$
\left[\nabla h_i(\xi)\right]_k=Re\left(-\im \frac{\partial W_i(\xi)}{\partial
  \xi_k}\bar \xi_i \right)+ \delta_{i,k} \, Re\, (-\im W_i(\xi)) \ .
$$
Inserting the formula displayed above in \eqref{f.g2} we get that $\nabla f_i$ is the sum of two
terms. We begin by estimating the second one, which we denote by $(\nabla
f_i)^{(2)}$. 
The $k^{th}$ component ($k \in \Z)$ of 
$(\nabla f)^{(2)} := \sum_l (\nabla f_l )^{(2)}$ is given by
\begin{equation}
\label{fg.3}
\left[\left( \nabla f(\xi)\right)^{(2)} \right]_k=\left[\sum_l(\nabla f_l(\xi))^{(2)}\right]_k =
\begin{cases}
\displaystyle{\int_{\T^{2k}}\Theta_{2k-1}^{-\theta}
\theta_k \, Re\, (-\im W_k(\Theta_{2k-1}^\theta \xi))\; d\theta^{2k}}\ , & k > 0 \\
 \displaystyle{\int_{\T^{2|k|+1}}\Theta_{2|k|}^{-\theta}
\theta_k \, Re\, (-\im W_k(\Theta_k^\theta \xi))\; d\theta^{2|k|+1}}\ , & k \leq 0 
\end{cases}
\end{equation}
thus, for any $\xi \in \Br^{{p,w^0}}(\mu^7\rho)$ one has
$\left[(\und{\nabla f}(|\xi|))^{(2)}\right]_k  \leq 2\pi\,
\und{W_k}(|\xi|)\ ,$ $\forall k \in \Z$,
and therefore  
$$
\tame{\und{\left(\nabla f\right)^{(2)}}}_{\mu^7 \rho}\leq 2\pi\, \tame{ \und{W}}_{\mu^7\rho}\leq \pi 2^8 \epsilon_1.
$$
We come to the other term, which we denote by $\left(\nabla f_i \right)^{(1)}$. Its $k^{th}$ component is given by
\begin{equation}
\label{fg.4}
\begin{aligned}
&  \left[ (\nabla f_{2i-1}(\xi))^{(1)}\right]_k= 
 \int_{\T^{2i}}\Theta_{2i-1}^{-\theta}
\theta_i Re\left(-\im \frac{\partial W_i}{\partial \xi_k}(\Theta_{2i-1}^\theta
\xi) \overline{\phi^{\theta_i}_i\xi_i}\right)   d\theta^{2i}  \ , \\
&  \left[ (\nabla f_{2i}(\xi))^{(1)}\right]_k= 
 \int_{\T^{2i}}\Theta_{2i}^{-\theta}
\theta_{-i} Re\left(-\im \frac{\partial W_{-i}}{\partial \xi_k}(\Theta_{2i}^\theta
\xi) \overline{\phi^{\theta_{-i}}_{-i} \xi_{-i}}\right)   d\theta^{2i+1}  \ .
\end{aligned}
\end{equation}
Then 
\begin{align*}
& \left[\und{\nabla f_{2i-1}}(|\xi|) \right]_k\leq 2\pi \und{\frac{\partial W_i}{\partial \xi_k}}(|\xi|) |\xi_i| = 2\pi  [\und{ \di W}(|\xi|)]^i_k|\xi_i|  
\ , \\
& \left[\und{\nabla f_{2i}}(|\xi|)\right]_k 
\leq 2\pi 
\und{\frac{\partial W_{-i}}{\partial \xi_k}}(|\xi|) |\xi_{-i}| 
= 2\pi  [\und{ \di W}(|\xi|)]^{-i}_k|\xi_{-i}|  \ .
\end{align*}
 It follows that the $k^{th}$ component of the function $(\nabla f)^{(1)} := \sum_{i \geq 0} (\nabla f_i)^{(1)}$ satisfies
$$
\left[ (\und{\nabla f}(|\xi|))^{(1)}\right]_k 
\leq 
\left[
   \sum_{l\geq 0}(\und{\nabla f_l}(|\xi|))^{(1)} \right]_k 
   \leq 2 \pi \,
 \sum_{l \in \Z} [\und{\di W}(|\xi|)]^l_k|\xi_l| 
  = 2\pi [\und{\di W}(|\xi|)^* |\xi| ]_k  \ .
 $$ 
 Therefore 
 $ \tame{\und{(\nabla f)^{(1)}}}_{\mu^7\rho}\leq 2\pi\, 
 \norm{W}_{\sT_{\mu^7\rho}} \leq \pi 2^8 \epsilon_1$. This is the step at
 which the control of the norm of the modulus $\und{\di W^*}$ of $\di W^*$ is
 needed.  Thus the claimed estimate for $\nabla f$ follows.
\end{proof}

We can finally apply the Darboux procedure in order to construct an
analytic change of coordinates $\psi $ which satisfies \eqref{fi.1}
and \eqref{fi.2}.
\begin{lemma}
\label{varphi.1}
There exists a map $\psi: \Br^{{p,w^0}}(\mu^9\rho) \to \spazior{p, w^0}$ which satisfies 
\eqref{fi.1}. Moreover $\psi - \uno \in \cN_{\mu^9\rho}^T(B^{p,w^0}\cap \spazior{p,w^1}, \, \spazior{p,w^2})$, $\psi-\uno = O(\xi^N)$ and  
\begin{equation}
\label{eq.varphi.1}
\tame{\underline{\psi-\uno}}_{\mu^9\rho}\leq
2^{14} \epsilon_1\ .
\end{equation}
\end{lemma}
\begin{proof}
As in \cite[Lemma 2.24]{masp_toda}.
 As anticipated just after Corollary \ref{cor.psi.1}, we apply the Darboux procedure with $\Omega_0 = \omega_0$, $\Omega_1 = \hat{\omega}_1$ and $f$ solution of \eqref{fi.3}. Then equation \eqref{darboux.eq} takes the form
\begin{equation}
\label{Yt}
Y^t=( \im +t\Upsilon_{\hat{\omega}_1})^{-1}(\nabla f-W),
\end{equation}
where $\Upsilon_{\hat{\omega}_1}$ and $W$ are defined in Corollary \ref{cor.e}.  
By Lemma \ref{moser} and Corollary \ref{cor.e}, the vector field $Y^t$ is of class $\cN_{\mu^8\rho}^T(B^{p,w^0}\cap\spazior{p,w^1}, \spazior{p,w^2})$, $Y^t(\xi) = O(\xi^N)$ 
 and
$$
\sup_{t \in [0,1]}\tame{\und{Y^t}}_{\mu^8\rho} < 2(2^{11}\epsilon_1 + 2^7\epsilon_1) < 2^{13} \epsilon_1.
$$
Thus  $Y^t$ generates a flow $\psi^t: \Br^{{p,w^0}}(\mu^9\rho)\to \spazior{p,w^0}$, defined for every $t \in [0,1]$,  which satisfies (cf. Lemma \ref{flussoinA})
$$
\tame{\und{\psi^t-\uno}}_{\mu^9\rho}\leq
2^{14} \epsilon_1, \quad \forall t \in [0,1]\ .
$$
Thus the map $\psi: = \left.\psi^t\right|_{t=1}$ exists and satisfies the claimed properties.
\end{proof}

\begin{lemma}
\label{lem.f} 
Let  $f$ be as in \eqref{f.formula} and $\psi^t$ be the flow map of the vector field $Y^t$ defined in \eqref{Yt}. Then $\forall \, l \geq 1$ one has $I_l(\psi^t(\xi))=I_l(\xi)$, for each $t \in [0,1]$.
\end{lemma}
\begin{proof}
As in \cite[Lemma 2.25]{masp_toda}.
\end{proof}

\vspace{1em}
\noindent{\em Proof of Theorem \ref{KP}.} 
It follows as in \cite{masp_toda}. 
The invertibility of $\widetilde \Psi$ and the analytic properties stated in item $v)$ follow by Lemma \ref{FGinA} $(ii)$.
\qed


\section{Proof of Lemma \ref{lem:exp}}
\label{app.exp}
	
We prove the result only for $Z^n_j(\zeta)$, since for $W^n_j(\zeta)$ the computations are analogous. We follow again the constuction of \cite{kuksinperelman, masp_toda}, adding more precise  quantitative estimates.\\
To perform the Taylor expansion at every order it is convenient to proceed in the following way.
Write $z_j(\zeta)=z_{j,1}(\zeta)+z_{j,2}(\zeta)$ where
\begin{equation}
\label{zj1ezj2}
	  z_{j,1}(\zeta):=\left( \left(L_0-\lambda_{j}^0\right) f_{j}^-(\zeta), \ \imath f_{j}^-(\zeta)\right)_\cY , 
	  \quad 
	  z_{j,2}(\zeta):=\left(V(\zeta) f_{j}^-(\zeta), \ \imath f_{j}^-(\zeta)\right)_\cY \ .
\end{equation} 
For $\varsigma= 1,2$ we will write 
$$
	z_{j,\varsigma}(\zeta) = \sum_{n=1}^\infty Z_{j,\varsigma}^n(\zeta) \ , 
	$$
	where the $Z_{j,\varsigma}^n(\zeta)$ are bounded $n$-homogeneous polynomials in $\zeta$.	
To obtain the explicit expression for the $Z^n_{j,\varsigma}$'s, we begin by  expanding  $f_{j}^+(\zeta)$ and $f_{j}^-(\zeta)$
in Taylor series. Since 
$$f_{j}^\pm(\zeta)=U_{j}(\zeta)f_{j0}^\pm=\Big(\uno-\left(P_{j}(\zeta)-P_{j0}\right)^2\Big)^{-1/2}\Big(\uno+(P_{j}(\zeta)-P_{j0})\Big)f_{j0}^\pm$$
we  expand the r.h.s. above in power series of $P_{j}(\zeta)-P_{j0}$,
getting:
\begin{equation}
\label{espansione di fj}
f_{j}^\pm(\zeta)=\sum_{m=0}^{\infty}c_m\left(P_{j}(\zeta)-P_{j0}\right)^m f_{j0}^\pm \ ,
\end{equation}
where the $c_m$'s are the coefficients of the Taylor series of the
 function $\phi(x)=\frac{1+x}{(1-x^2)^{1/2}}$.
 In particular $c_{2k+1}  = c_{2k} := (-1)^k \binom{-1/2}{k} \equiv \left( \frac{1}{4} \right)^k \binom{2k}{k}$. Using also  that  ${\sqrt {2\pi n}}\left({\frac {n}{e}}\right)^{n}<n!<{\sqrt {2\pi n}}\left({\frac {n}{e}}\right)^{n}e^{\frac {1}{12n}} $ one has $ 0 < c_{m} \leq 1$ $, \forall m \geq 0$.

 Now we expand $P_{j}(\zeta)$ in Taylor series:  for $\zeta$ sufficiently small
\begin{equation}
P_{j}(\zeta)-P_{j0} = \sum_{n=1}^\infty (-1)^n \sP^n(\zeta) \ , 
 \qquad  \sP^n(\zeta) : =-\frac{1 }{2\pi \im  } \oint_{\Gamma_j}{\left(L_0-\lambda\right)^{-1}  T^n(\zeta, \lambda)\, \di {\lambda}} 
\label{espansionePj}
\end{equation}
where the $\Gamma_j$'s are defined as in  \eqref{gamma.def},  and 
$$ T(\zeta,\lambda):=V(\zeta) \, \left(L_0-\lambda\right)^{-1}\ .$$
Substituting \eqref{espansionePj}  into \eqref{espansione di fj} we get that
\begin{equation}
\label{expressionfj}
\begin{aligned}
&f_{j}^\pm(\zeta)=f_{j0}^\pm+\sum_{n\geq 1}\sum_{1\leq m \leq n}c_m \sum_{\alpha = (\alpha_1, \ldots, \alpha_m) \in \N^m \atop |\alpha| = n}f_{j,m}^{\pm,\alpha}(\zeta),\\
& \qquad  f_{j,m}^{\pm,\alpha}(\zeta):= (-1)^{|\alpha|} \sP^{\alpha_m}(\zeta) \circ \cdots \circ \sP^{\alpha_1}(\zeta) f_{j0}^\pm
 \ . 
\end{aligned}
\end{equation}

Consider now $Z_{j,1}^n$. It is obtained by inserting \eqref{expressionfj} into $z_{j,1}(\zeta)$ and collecting the terms of order $n$. Thus $Z_{j,1}^n(\zeta) $ equals 
\begin{align}
\label{Zj1n} 
\sum_{q=(q_1,q_2) \in \N^2 \atop |q|\leq n} 
c_{q_1} c_{q_2} \sum_{\beta=(\beta_1, \ldots, \beta_{|q|}) \in \N^{|q|} \atop |\beta| = n } 
\Big(
 (L_0 - \lambda_{j}^0) f_{j,q_1}^{-, (\beta_1, \ldots, \beta_{q_1})}(\zeta) \ ,\imath  f_{j,q_2}^{-, (\beta_{|q|}, \ldots, \beta_{q_1+1})}(\zeta)
 \Big)_\cY \ . 
\end{align}
We claim that for every $\norm{\zeta}_{2} < 1/8$ one has 
\begin{equation}
\label{sP}
\imath \, \sP^n(\zeta) = \sP^n(\zeta)^* \, \imath \ . 
\end{equation}
Indeed by  Lemma \ref{key.pert},   $\Gamma_j \ni \lambda  \mapsto \left(L_0-\lambda\right)^{-1}  T^n(\zeta, \lambda) \in \cL(\cY)$ is  continuous. By
\eqref{i.comm}  
$$
 \imath \, \left(L_0-\lambda\right)^{-1}  T^n(\zeta, \lambda) = [\left(L_0-\lambda\right)^{-1}  T^n(\zeta, \lambda)]^* \, \imath \ , 
$$
thus \eqref{sP} follows by a direct computation.
Using \eqref{sP},  the scalar product of \eqref{Zj1n} becomes
\begin{align}
\label{Zj1n2}  
 (-1)^n
\Big(
 \sP^{\beta_{|q|}}(\zeta)\circ \ldots \circ \sP^{\beta_{q_1+1}}(\zeta)(L_0 - \lambda_{j}^0) \
  \sP^{\beta_{q_1}}(\zeta)\circ \ldots \circ \sP^{\beta_{1}}(\zeta)
  f_{j0}^- \ ,\imath  f_{j0}^-
 \Big)_\cY \ . 
\end{align}
To write it explicitly remark that by (H4b)
\begin{align}
\left(L_0-\lambda\right)^{-1} f_{j0}^\pm = \frac{1}{\lambda_{j}^0 - \lambda} \  f_{j0}^\pm \ , \qquad  V(\zeta) f_{j0}^\mp = \sum_{i \in \Z } ( V(\zeta) f_{j0}^\mp, f_{i0}^\pm )_\cY \ f_{i0}^\pm \ .
\end{align}
Therefore  $ \forall a \in \N$
\begin{equation}
\label{Tafor}
\begin{aligned}
\sP^a &(\zeta) f_{j0}^- =\\ 
& \sum_{i_1, \ldots, i_a  \in \Z }  \left[\frac{\im}{2 \pi } \oint_{\Gamma_j}
 \frac{1}{\lambda_{j}^0 - \lambda} 
 \frac{( V(\zeta) f_{j0}^-, f_{i_1 0}^{+} )}{\lambda_{i_1}^0 - \lambda} \ \frac{( V(\zeta) f_{i_10}^{+}, f_{i_2 0}^{-} )}{\lambda_{i_2 }^0 - \lambda} \ldots \frac{( V(\zeta) f_{i_{a-1}0}^{\sigma_{a-1}}, f_{i_a 0}^{\sigma_a} )}{\lambda_{i_a }^0 - \lambda} \di \lambda \right]  \ f_{i_a 0}^{\sigma_a}
 \end{aligned}
\end{equation}
where $\sigma_a = +$ if $a$ is odd, and $\sigma_a = -$ if $a$ is even.
Using repeatedly \eqref{Tafor} one gets 
\begin{align*}
& \eqref{Zj1n2} = 
\sum_{i_1, \ldots, i_n \in \Z}
\cS_{j,1}^{q,\beta}({\bf i})  \ ( V(\zeta) f_{j0}^-, f_{i_1 0}^{+} )_\cY \ ( V(\zeta) f_{i_10}^{+}, f_{i_2 0}^{-} )_\cY \ldots ( V(\zeta) f_{i_{n-1}0}^{\sigma_{n-1}}, f_{i_n 0}^{\sigma_n} )_\cY  \Big( f_{i_n 0}^{\sigma_n} , \imath \, f_{j0}^- \Big)_\cY \ ,
\end{align*}
with the kernel
\begin{align*}
&\cS^{q, \beta}_{j,1}({\bf i}) := \left(\frac{\im}{2\pi}\right)^{|q|}  (-1)^{n}  \oint_{\Gamma_j}\ldots \oint_{\Gamma_j}
s^{q, \beta}_{j,1}({\bf i}, \overrightarrow{\lambda}) \ \di \lambda_1 \cdots \di \lambda_{|q|}  \ , \\
& s^{q, \beta}_{j,1}({\bf i}, \overrightarrow{\lambda}):= \prod_{m=1}^{n-1} \frac{1}{\lambda(i_m) -\mu_m} 
\times \prod_{\ell = 1}^{|q| -1} \frac{1}{\lambda(i_{\sum_{r=1}^\ell\beta_{r}}) - \lambda_{\ell+1}} 
\times \frac{\lambda(i_{\beta_1 + \cdots + \beta_{q_1}}) - \lambda(j)}{\lambda(j) - \lambda_1} \times \frac{1}{\lambda(i_n) -\lambda_{|q|}} 
\end{align*}
 where we denoted  $\lambda(a) := \lambda_{a}^0$, ${\bf i} = (i_1, \ldots, i_{n})$, $\overrightarrow{\lambda} = (\lambda_1, \ldots, \lambda_{|q|})$, and   $\mu_m=\mu_m(\lambda; q,\beta) \in \Gamma_j$.

Now remark that $\imath f_{j0}^- = f_{j0}^+$, hence   $\Big( f_{i_n 0}^{\sigma_n} , \imath \, f_{j0}^- \Big)_\cY
\equiv \Big( f_{i_n 0}^{\sigma_n} ,  f_{j0}^+ \Big)_\cY 
 \neq 0$ only if $i_n = j$, $\sigma_n = +$. This implies that when $n$ is even  $\eqref{Zj1n2}  =0$, while when  $n$ is  odd 
 \begin{align}
 \label{Zj1n3}
 \eqref{Zj1n2} = \sum_{i_1, \ldots, i_{n-1} \in \Z}
\cS_{j,1}^{q,\beta}({\bf i})  \ ( V(\zeta) f_{j0}^-, f_{i_1 0}^{+} )_\cY \ ( V(\zeta) f_{i_10}^{+}, f_{i_2 0}^{-} )_\cY \ldots ( V(\zeta) f_{i_{n-1}0}^{\sigma_{n-1}}, f_{j 0}^{+} )_\cY  \ . 
 \end{align}
	Altogether one has 
	\begin{align*}
Z_{j,1}^n(\zeta)  & = \sum_{i_1, \ldots, i_{n-1} \in \Z  }
\wt\cK_{j,1}({\bf i})  \
 ( V(\zeta) f_{j0}^-, f_{i_1 0}^{+} )_\cY \ 
 ( V(\zeta) f_{i_10}^{+}, f_{i_2 0}^{-} )_\cY \ldots 
 ( V(\zeta) f_{i_{n-1}0}^{-}, f_{j 0}^{+} )_\cY  \ , 
 \\
& \wt\cK_{j,1}({\bf i}):=  \sum_{q=(q_1,q_2) \in \N^2 \atop |q|\leq n} 
c_{q_1} c_{q_2} \sum_{\beta=(\beta_1, \ldots, \beta_{|q|}) \in \N^{|q|} \atop |\beta| = n } \cS_{j,1}^{q,\beta}({\bf i}) \ . 
\end{align*}
Now consider $\cS_{j,1}^{q,\beta}({\bf i})$.
 Recall that $\lambda_1, \ldots, \lambda_{|q|}  \in \Gamma_j \equiv \{ \lambda \in \C \colon |\lambda - \lambda(j)|  = \pi/2 \}$ and that for any $\lambda \in \Gamma_j$ one has the estimate
 $$
 4 \abs{ \lambda(i) - \lambda } \geq \la \lambda(i) - \lambda(j) \ra  \geq \la i - j \ra  \ , \qquad \forall i \in \Z \ ,  \ \ \ \forall \lambda \in \Gamma_j \ ;
 $$
this  implies 
$$
\abs{\cS_{j,1}^{q,\beta}({\bf i}) } 
\leq 4^{n-1} \prod_{m=1 }^{n-1} \frac{1}{\la i_m - j \ra} \cdot
\sup_{\lambda \in \Gamma_j}\abs{ \frac{\lambda(i_{\beta_1 + \cdots + \beta_{q_1}}) - \lambda(j)}{\lambda(i_{\beta_1 + \cdots + \beta_{q_1}}) - \lambda} } 
 \leq 4^{n} \prod_{m=1 }^{n-1} \frac{1}{\la i_m - j \ra}  \ . 
$$
Since the coefficients $c_q \leq 1$ $\forall q \in \N$, it follows that\footnote{recall that any   $n \in \N$ can be written  as a sum of $k$ positive  integers in $\binom{n-1}{k-1}$ possible ways.}
\begin{align}
\notag
\abs{\wt\cK_{j,1}^n({\bf i})} & 
\leq  4^n \prod_{m=1 }^{n-1} \frac{1}{\la i_m - j \ra}  \sum_{q=(q_1,q_2) \in \N^2 \atop |q|\leq n} 
c_{q_1} c_{q_2} \sum_{\beta=(\beta_1, \ldots, \beta_{|q|}) \in \N^{|q|} \atop |\beta| = n } 1 \\
\notag
&\leq  4^n \prod_{m=1 }^{n-1} \frac{1}{\la i_m - j \ra}  \sum_{r=1}^n \sum_{q=(q_1,q_2) \in \N^2 \atop |q|= r}  \sum_{\beta=(\beta_1, \ldots, \beta_{|q|}) \in \N^{|q|} \atop |\beta| = n } 1 \\
\notag
& \leq  4^n \prod_{m=1}^{n-1} \frac{1}{\la i_m - j\ra} \sum_{r=1}^n 
\binom{r-1}{1} 
\binom{ n-1}{r-1}  
\leq  4^n \prod_{m=1}^{n-1} \frac{1}{\la i_m - j\ra} \sum_{r=0}^{n-1} 
r\binom{ n-1}{r} 
\\
\label{cKjn}
& 
\leq  4^n \prod_{m=1}^{n-1} \frac{1}{\la i_m - j\ra} \, (n-1)2^{n-2}
\leq  16^{n-1} \prod_{m=1}^{n-1} \frac{1}{\la i_m - j \ra} \ , 
\end{align}
Similar computations can be performed for $Z_{j,2}^n(\zeta)$ and one finds
$$
Z_{j,2}^n (\zeta) =
 \sum_{i_1, \ldots, i_{n-1} \in \Z}
\wt\cK_{j,2}^n({\bf i})  \ ( V(\zeta) f_{j0}^-, f_{i_1 0}^{+} )_\cY \ ( V(\zeta) f_{i_10}^{+}, f_{i_2 0}^{-} )_\cY 
\ldots 
( V(\zeta) f_{i_{n-1}0}^{-}, f_{j 0}^{+} )_\cY $$
where the kernel $\wt\cK_{j,2}({\bf i})$ fulfills the same  bound  \eqref{cKjn}.
In such a  way we proved that
\begin{equation}
\label{Z.001}
Z_j^n(\zeta) = \sum_{i_1, \ldots, i_{n-1} \in \Z }
\wt\cK_{j}^n({\bf i})  \ 
( V(\zeta) f_{j0}^-, f_{i_1 0}^{+} )_\cY \
 ( V(\zeta) f_{i_10}^{+}, f_{i_2 0}^{-} )_\cY 
 \ldots 
 ( V(\zeta) f_{i_{n-1}0}^{-}, f_{j 0}^{+} )_\cY
\end{equation}
with 
\begin{equation}
\label{ker.est0}
\abs{\wt\cK_{j}^n({\bf i})} \leq 2\cdot   16^{n-1}  \prod_{m=1 }^{n-1} \frac{1}{\la i_m - j\ra} \ . 
\end{equation}
Now remark that by  (H4b), $j+ i_1$, $i_a + i_{a+1}$ and $i_{n-1} + j$ must be even, so define $k_1, \ldots, k_n$ by
\begin{equation*}
2k_1 =  i_1  +  j \,  , \qquad 2k_\ell = (-1)^{\ell+1} (i_\ell + i_{\ell-1}) \ , \qquad 2k_n = i_{n-1} + j\ . 
\end{equation*}
In this way one has 
\begin{equation}
\label{support}
j = k_1 + \ldots + k_n
\end{equation}
 and  by (H4b), for any $1 \leq m \leq n$, $m$ odd
$$
( V(\zeta) f_{j0}^{-}, f_{i_{1} 0}^{+} )_\cY = \xi_{k_1} \ , 
\qquad
( V(\zeta) f_{i_{m-1} 0}^{-}, f_{i_{m} 0}^{+} )_\cY = \xi_{k_{m}} \ , 
\qquad
( V(\zeta) f_{i_m 0}^{+}, f_{i_{m+1} 0}^{-} )_\cY = \eta_{-k_{m+1}} 
$$
so that \eqref{Z.001} becomes 
$$
Z_j^n(\zeta) = \sum_{k_1 +k_2 + \ldots + k_n = j }
\cK_{j}^n(k_1, \cdots, k_n) \,   \xi_{k_1} \, \eta_{-k_2} \, \ldots \xi_{k_n}
 $$
with kernel 
$$
\cK_{j}^n(k_1, \cdots, k_n) := \wt\cK_j^n\Big((2k_1-j), -(2k_2 + 2k_1 - j), \ldots, -(2k_{n-1} + \ldots + 2k_1 - j)\Big) \ . 
$$
Such kernel is supported in \eqref{support} and by the estimate \eqref{ker.est0} it fulfills  \eqref{eq:K.est}.

\section{Technical results}
\label{app:weight}
First we write explicitly $\fg_{n,r}$ defined in \eqref{def:gg}.
Let  $1 \leq r \leq n$ be integers. 
  For $r$ odd we have 
\begin{equation}
\label{def:gg0}
\begin{aligned}
\fg_{n,r}(\bk  ; j) \equiv & \uno_{ \fS^{n,r}_{-j}}
\prod_{m=1\atop m \, {\rm odd}}^{r-2} \la \sum_{\ell = 1}^m k_\ell  - k_r \ra^{-1}  \la  \sum_{\ell = 1}^{m+1} k_\ell \ra^{-1}
\\
& \times 
\la \sum_{\ell = 1}^{r-1} k_\ell + j - k_r \ra^{-1} \, \la  \sum_{\ell = 1 }^{r-1} k_\ell + j + k_{r+1} \ra^{-1} \\
& \times  \prod_{m=r+1 \atop m {\rm odd}}^{n-1} \la  \sum_{\ell = 1\atop \ell \neq r}^m k_\ell  +j - k_r \ra^{-1}  \la \sum_{\ell = 1 \atop \ell \neq r}^{m+1} k_\ell + j\ra^{-1}
\end{aligned} \ , 
\end{equation}
where $\fS^{n,r}_{a}$ is defined in \eqref{set.2}.
For $r$ even
\begin{equation}
\label{def:gg1}
\begin{aligned}
\fg_{n,r}(\bk ; j) \equiv & \uno_{\fS^{n,r}_{-j}} 
\prod_{m=1\atop m \, {\rm odd}}^{r-3} 
\la \sum_{\ell = 1}^m k_\ell  - k_r \ra^{-1} 
 \la  \sum_{\ell = 1}^{m+1} k_\ell \ra^{-1} \\
 &
\times \la \sum_{\ell = 1}^{r-1} k_\ell  - k_r \ra^{-1} \, 
\la  \sum_{\ell = 1 }^{r-1} k_\ell +j  \ra^{-1} \\
&
\times  \prod_{m=r+1 \atop m {\rm odd }}^{n-1}
\la  \sum_{\ell = 1\atop \ell \neq r}^m k_\ell  +j - k_r \ra^{-1} \, \la \sum_{\ell = 1 \atop \ell \neq r}^{m+1} k_\ell + j\ra^{-1} \ .
\end{aligned}  
\end{equation}

\begin{example}
For $n = 3$, one has
\begin{align*}
&\ff_3(k_1, k_2, k_3; j) = \uno_{\{k_1+k_2+k_3 = j\}}\, \left(\la k_1 - j \ra \la k_1 + k_2 \ra \right)^{-1}  \ , \\
&\fg_{3,1}(k_1, k_2, k_3; j)= \uno_{\{-k_1+k_2+k_3 =- j\}}\,
 (\la j-k_1 \ra \, \la j+k_2  \ra)^{-1}  \ ,  \\
&\fg_{3,2}(k_1, k_2, k_3; j) =  
\uno_{\{k_1-k_2+k_3 = -j\}}
 (\la k_1 - k_2 \ra \, \la k_1 + j \ra )^{-1} \ , \\
 & \fg_{3,3}(k_1, k_2, k_3; j) = 
 \uno_{\{k_1+k_2-k_3 = -j\}}
(\la k_1 - k_3 \ra \, \la k_1 + k_2 \ra )^{-1} \ , 
\end{align*}
while for $n=5$
\begin{align*}
& \ff_5(k_1, k_2, k_3, k_4, k_5; j) = \frac{\uno_{\{k_1+k_2+k_3+k_4 + k_5 = j\}} }{\la k_1 - j \ra \la k_1 + k_2 \ra \, \la k_1 + k_2 + k_3 - j \ra \, \la k_1 + k_2 + k_3 + k_4 \ra } \\
&\fg_{5,1}(k_1, k_2, k_3, k_4, k_5; j) = 
\frac{\uno_{\{-k_1+k_2+k_3+k_4 + k_5 = -j\}} }
{\la j-k_1 \ra \, \la j+k_2  \ra \, \la j+k_2 + k_3 - k_1 \ra \, \la j+ k_2 + k_3 + k_4 \ra }  \\
&\fg_{5,2}(k_1, k_2, k_3, k_4, k_5; j) = 
\frac{\uno_{\{k_1-k_2+k_3+k_4 + k_5 = -j\}} }
{\la k_1 - k_2 \ra \, \la k_1 + j   \ra \, \la k_1 + j + k_3 - k_2 \ra \, \la k_1 + j + k_3 + k_4 \ra } \\
 &\fg_{5,3}(k_1, k_2, k_3, k_4, k_5; j) = 
 \frac{\uno_{\{k_1+k_2-k_3+k_4 + k_5 = -j\}} }
{\la k_1 - k_3 \ra \, \la k_1 + k_2  \ra \, \la k_1 + k_2 + j - k_3 \ra \, \la k_1 + k_2 +j + k_4 \ra} \\
 &\fg_{5,4}(k_1, k_2, k_3, k_4, k_5; j) = 
 \frac{\uno_{\{k_1+k_2+k_3-k_4 + k_5 = -j\}} }
{\la k_1 - k_4 \ra \, \la k_1 + k_2  \ra \, \la k_1 + k_2 + k_3 - k_4 \ra \, \la k_1 + k_2 + k_3 + j \ra } \\
 &\fg_{5,5}(k_1, k_2, k_3, k_4, k_5; j) = 
 \frac{\uno_{\{k_1+k_2+k_3+k_4 - k_5 = -j\}} }
{\la k_1 - k_5 \ra \, \la k_1+k_2  \ra \, \la k_1+k_2 + k_3 - k_5 \ra \, \la k_1+ k_2 + k_3 + k_4 \ra }
\end{align*}
\end{example}

\subsection{Proof of Lemma \ref{f.g.lp}}
\label{app:weight.1}
\paragraph*{Verification of   \eqref{w10}.} We consider $1 < p \leq 2$ and $p=1$ separately.\\
 \underline{Case $1 < p\leq 2$}. The quantity that we have to bound is the $p'$ root of 
\begin{equation}
\label{arit.prop0}
 \sum_{k_1 + \ldots + k_n = j  } 
 \prod_{m=1\atop m\,  {\rm odd}}^{n-1} \frac{1}{\la k_1 + \ldots + k_m - j\ra^{p'} } \, \cdot \, \frac{1}{\la k_1 + \ldots + k_{m+1} \ra^{p'} }  \ . 
\end{equation}
 For any $j \in \Z$ we have  (recall that $n$ is odd, thus $m \in \{ 1, 3, \ldots, n-2\}$)
\begin{align*}
\eqref{arit.prop0}
&\leq  \sum_{k_1} \frac{1}{\la k_1 - j \ra^{p'}} \sum_{k_2} \frac{1}{\la k_1 + k_2 \ra^{p'}}\cdots 
\sum_{k_{n-2}} \frac{1}{\la k_1 + \ldots +k_{n-2} - j \ra^{p'}} \sum_{k_{n-1}} \frac{1}{\la k_1 + \ldots +k_{n-2} +k_{n-1}\ra^{p'}}\\
&\leq  \Big( \sum_k \frac{1}{\la k\ra^{p'}  } \Big)^{n-1}  \equiv R_*^{p'(n-1)} \ .
\end{align*}
Thus  \eqref{w10} follows.\\
\noindent\underline{Case $p=1$}.  In this case the quantity to estimate is
\begin{align*}
 \sup_{k_1 + \ldots + k_n = j  } 
 \prod_{m=1\atop m\,  {\rm odd}}^{n-1} 
\frac{1}{\la k_1 + \ldots + k_m - j\ra } \, \cdot \, \frac{1}{\la k_1 + \ldots + k_{m+1} \ra } 
\end{align*}
which is trivially majorated by $1$,  thus \eqref{w10} follows.

\paragraph*{Verification of  \eqref{w20}.} 
\noindent\underline{Case $1 <p\leq 2$}. To majorate \eqref{w20},   it is sufficient to remark that any index $k_1, \ldots, k_{n-1}$ appears at least once in one term of $\fg_{n,r}(k_1, \ldots, k_n; j)$, as one verifies inspecting formulas \eqref{def:gg0}, \eqref{def:gg1}.
 Then one gets again that  
$\norm{\fg_{n,r}(\cdot ; j)}_{\ell^{p'}(\Z^n)}^{p'} \leq  \Big( \sum_k \frac{1}{\la k\ra^{p'}  } \Big)^{n-1}\equiv R_*^{p'(n-1)}$
and   \eqref{w20} is fulfilled. \\
\underline{Case $p=1$}. The argument is similar to the previous case, and we skip it.

\subsection{Proof of  Proposition \ref{prop:sw} $(i)$}
\label{app:weight.W}
\paragraph*{Verification of  \eqref{w1}.} 
We treat separately the case $1 < p \leq 2 $ and $p=1$. \\

\underline{Case $1<p \leq 2$}. 
For every  $j \in \Z$, $a \geq 0$, $0<b \leq 1$ the l.h.s. of \eqref{w1} is  the $p'^{th}$-root of the supremum over $j$ of
\begin{align*}
\la j \ra^{sp'} e^{p' a |j|^b }
 \sum_{k_1 + \ldots + k_n = j  }  \frac{1}{n^{p'}}
\frac{1}{ \, \prod\limits_{i=1}^n \la k_i\ra^{sp'} e^{{p'} a |k_i|^b }} \prod_{m=1\atop m\,  {\rm odd}}^{n-1} 
\frac{1}{\la k_1 + \ldots + k_m - j\ra^{p'} } \, \cdot \, \frac{1}{\la k_1 + \ldots + k_{m+1} \ra^{p'} } 
\end{align*}
which is majorated by \eqref{arit.prop0} 
since for any $j =k_1 + \cdots + k_n$
\begin{equation}
\label{exp.prop}
e^{p' a |j|^b } \leq e^{{p'} a |k_1|^b } \ \cdots \ e^{{p'} a |k_n|^b}  \ , \qquad
\la j \ra^{sp'} \leq \la k_1 \ra^{sp'} \ \cdots \ \la k_n \ra^{sp'}
 \ . 
\end{equation}
Then the result follows from Lemma \ref{f.g.lp}, with  $R_0 = 1$, $R_1 = R_*$.\\

\underline{Case $p=1$}.  In this case the quantity to estimate is
\begin{align*}
\la j \ra^{s} e^{ a |j|^b }
 \sup_{k_1 + \ldots + k_n = j  } 
\frac{1}{n} \frac{1}{\prod_{i=1}^n \la k_i \ra^s e^{a |k_i|^b }} \prod_{m=1\atop m\,  {\rm odd}}^{n-1} 
\frac{1}{\la k_1 + \ldots + k_m - j\ra } \, \cdot \, \frac{1}{\la k_1 + \ldots + k_{m+1} \ra } 
\end{align*}
which is trivially majorated by $1$, using \eqref{exp.prop}. Thus \eqref{w1} holds with $R_0 = R_1 = 1$.\\

\paragraph*{Verification of \eqref{w2}.} Using \eqref{exp.prop}, one is brought back to estimate \eqref{w20}.

\subsection{Proof of  Proposition \ref{prop:sw} (ii)}
\paragraph*{ Verification of  \eqref{w1}.} As in the previous case we treat two cases:\\

\underline{Case $1 < p \leq 2$.} 
We must  estimate the $p'$-root of 
\begin{align}
\label{schifo}
 \la j \ra^{p's}
 \sum_{k_1 + \ldots + k_n = j  } \ 
\frac{1}{\left(\sum_{l=1}^n \la k_l \ra^{s} \right)^{p'}} \ 
\prod_{m=1\atop m\,  {\rm odd}}^{n-1} \frac{1}{\la k_1 + \ldots + k_m - j\ra^{p'} } \, \cdot \, \frac{1}{\la k_1 + \ldots + k_{m+1} \ra^{p'} } 
\end{align}
For $j=k_1 + \cdots + k_n$, one has the inequalities
\begin{equation}
\label{sum.in}
 \la j \ra^{a} \leq \left(\sum_{l=1}^n \la k_l\ra \right)^a \leq n^{a-1} \ \sum_{l=1}^n \la k_l\ra^a \ , \quad \forall a \geq 1 \ ,  
\end{equation}
which yields 
\begin{align*}
 \eqref{schifo}
 & \leq 
 n^{p's-1} \sum_{k_1 + \ldots + k_n = j  }   \prod_{m=1\atop m\,  {\rm odd}}^{n-1} \frac{1}{\la k_1 + \ldots + k_m - j\ra^{p'} } \, \cdot \, \frac{1}{\la k_1 + \ldots + k_{m+1} \ra^{p'} } \\
 &
 \leq (2^{p's-1})^{n-1} \cdot \eqref{arit.prop0} 
\leq 
 \Big( 2^{p' s} \sum_k \frac{1}{\la k\ra^{p'}  } \Big)^{n-1}
 \end{align*}
 where in the last line we used that $\forall a >1$, $\forall 3 \leq n \in \N$, one has $n^a \leq 2^{a(n-1)}$.
Thus  \eqref{w1} is fulfilled with $R_0 = 1$, $R_1 = 2^s \left(\sum_{k \in \Z} \frac{1}{\la k\ra^{p'} }\right)^{1/p'}$.\\

\underline{Case $p=1$.}  One has to bound the quantity
\begin{align}
\label{schifo3}
 \la j \ra^{s}
 \sup_{k_1 + \ldots + k_n = j  } \ 
\frac{1}{\sum_{l=1}^n \la k_l \ra^{s} } \ 
\prod_{m=1\atop m\,  {\rm odd}}^{n-1} \frac{1}{\la k_1 + \ldots + k_m - j\ra } \, \cdot \, \frac{1}{\la k_1 + \ldots + k_{m+1} \ra } \ ; 
\end{align}
using \eqref{sum.in} one gets the desired bound easily.

\paragraph*{Verification of \eqref{w2}.} One  uses inequality \eqref{sum.in} and proceed as in the proof of Proposition \ref{prop:sw}$(i)$; in turn  \eqref{w2} is fulfilled with $R_0 = 1$, $R_1 = 2^s\left(\frac{1}{\la k\ra^{p'} }\right)^{1/p'}$. \\

\subsection{Proof of  Proposition \ref{prop:sw} $(iii)$}

\paragraph*{Verification of  \eqref{w1}.} As above  we treat two cases:\\

\underline{Case $1 < p \leq 2$.} 
We must  estimate the $p'$-root of the supremum over $j$ of
\begin{align}
\label{schifo2}
 \la j \ra^{p'(s+1)}
 \sum_{k_1 + \ldots + k_n = j  } \ 
\frac{1}{\left(\sum_{l=1}^n \la k_l \ra^{s} \prod_{m\neq l} \la k_m \ra \right)^{p'}} \ 
\prod_{m=1\atop m\,  {\rm odd}}^{n-1} \frac{1}{\la k_1 + \ldots + k_m - j\ra^{p'} } \, \cdot \, \frac{1}{\la k_1 + \ldots + k_{m+1} \ra^{p'} } 
\end{align}
Using \eqref{sum.in},  expression
\eqref{schifo2} is majorated by $n^{p'(s+1)-1}
\sum_{l=1}^n  \cM_{l,n}$, 
\begin{equation}
\label{arit.prop}
  \cM_{l,n}:=
  \sum_{k_1 + \ldots + k_n = j  }   
  \frac{\la k_l \ra^{p'}}{\prod_{m\neq l} \la k_m \ra^{p'}} 
  \prod_{m=1\atop m\,  {\rm odd}}^{n-1} \frac{1}{\la k_1 + \ldots + k_m - j\ra^{p'} } \, \cdot \, \frac{1}{\la k_1 + \ldots + k_{m+1} \ra^{p'} } 
\end{equation}
We claim that $ \forall 1 \leq l \leq n $
\begin{equation}
\label{palle}
\cM_{l,n} \leq  R_{\flat}^{n-1} \ ,\qquad  R_\flat := 2^{p'} R_*^{p'} \  .
\end{equation}
To prove \eqref{palle} we consider separately the case $l$ even and $l$ odd;  if $l$ is even then
\begin{align*}
\cM_{l,n} \leq \la k_l\ra^{p'} 
\sum_{k_m \colon m \neq l}
\frac{1}{\prod_{m \neq l} \la k_m \ra^{p'} } \cdot 
\frac{1}{\la k_1 + \ldots + k_{l-1} + k_l \ra^{p'} } \leq R_\flat^{n-1} \ ,
\end{align*}
which follows using repeatedly the inequalities
\begin{equation}
\label{est.sum0}
 \sum_{k\in \Z} \frac{1}{\la k\ra^{p'} \, \la k-j \ra^{p'} } \leq \frac{R_\flat}{\la j \ra^{p'}} \ , \qquad 
\sum_{k \in \Z}  \frac{1}{\la k\ra^{p'}}  < R_\flat \ .
\end{equation}
Similarly, if $l$ is odd and $l\neq n$, then
\begin{align*}
\cM_{l,n} \leq \la k_l\ra^{p'} 
\sum_{k_m \colon m \neq l}
\frac{1}{\prod_{m \neq l}\la k_m \ra^{p'} } \cdot 
\frac{1}{\la k_1 + \ldots + k_{l} + k_{l+1} \ra^{p'} }
 \leq R_\flat^{n-1} \ ,
\end{align*}
where once again we used \eqref{est.sum0} iteratively.
Finally consider the case $l=n$: using that $j-k_1 = k_2 + \ldots + k_n$,  
\begin{align*}
\cM_{n,n} \leq \la k_n\ra^{p'} 
\sum_{k_m \colon m \neq n}
\frac{1}{\prod_{m \neq n} \la k_m \ra^{p'} } \cdot 
\frac{1}{\la k_2 + \ldots + k_{n-1} + k_{n} \ra^{p'} } \leq R_\flat^{n-1} \ .
\end{align*}
All together  we proved \eqref{palle},  consequently 
$\eqref{schifo2} \leq n^{p'(s+1)-1} \sum_{l=1}^n \cM_{l,n}
\leq \left(2^{s+2} R_* \right)^{p'(n-1)}.$
 Thus \eqref{w1} holds with $R_0 = 1$, $R_1 = 2^{s+2} R_*$.\\

\underline{Case $p=1$.} One has to bound 
\begin{equation}
\label{arit.prop2}
\la j\ra^{s+1} \sup_{k_1 + \ldots + k_n = j  } 
\left| \frac{1}{\sum_{l=1}^n \la k_l\ra^s \ \prod_{m\neq l} \la k_m \ra} \prod_{m=1\atop m\,  {\rm odd}}^{n-1} \frac{1}{\la k_1 + \ldots + k_m - j\ra } \, \cdot \, \frac{1}{\la k_1 + \ldots + k_{m+1} \ra } \right| \ ,
\end{equation}
which is majorated by $n^{s}\sum_{l=1}^n \wt\cM_{l,n}$, 
\begin{equation}
\label{arit.prop20}
\wt\cM_{l,n}:= \sup_{k_1 + \ldots + k_n = j  } 
\left| \frac{\la k_l\ra}{\ \prod_{m\neq l} \la k_m \ra} \prod_{m=1\atop m\,  {\rm odd}}^{n-1} \frac{1}{\la k_1 + \ldots + k_m - j\ra } \, \cdot \, \frac{1}{\la k_1 + \ldots + k_{m+1} \ra } \right| \ .
\end{equation}
We claim that $\forall 1 \leq l \leq n$
\begin{equation}
\label{palle2}
\wt\cM_{l,n} \leq 1 \ ;
\end{equation}
this is proved exactly as \eqref{palle} using iteratively the inequality
\begin{equation}
\label{est.sum00}
\sup_{k \in \Z} \frac{1}{\la k \ra \, \la k-j \ra} \leq \frac{1}{\la j \ra} \ 
\end{equation}
in place of \eqref{est.sum0}.
Thus \eqref{w1} holds with $R_0  =1$, $R_1 = 2^{s+1}$.\\

\paragraph*{ Verification of \eqref{w2}.} We consider two cases.\\

\underline{Case $1 < p \leq 2$.} 
The quantity that we must estimate is the $p'^{th}$-root of
\begin{equation}
\label{g.est}
\la j \ra^{p'(s +1)}
\sum_{{\bf k} \in \fS^{n,r}_{-j}} 
\frac{1}{\left(\sum_{l=1}^n \la k_l\ra^{s} \prod_{m\neq l}\la k_m \ra \right)^{p'}} \ \ \fg_{n,r}(k_1, \cdots, k_n; j)^{p'}
\end{equation}
for every possible choice of $1 \leq r \leq n$ and $n \geq 3$, $n$ odd.
First remark that \eqref{g.est} is majorated by $n^{p'(s +1)-1} \sum_{l=1}^n \cN^l_{n,r}$, 
\begin{equation}
\label{g.est1}
\cN^l_{n,r}:=
\sum_{{\bf k} \in \fS^{n,r}_{-j}} 
\frac{\la k_l \ra^{p'}}{ \prod_{m\neq l}\la k_m \ra^{p'}} \ \ \fg_{n,r}(k_1, \cdots, k_n; j)^{p'} \ . 
\end{equation}
	Once again we claim that $\forall 1 \leq l,r \leq n$,
	\begin{equation}
	\label{N.est}
	\cN^l_{n,r} \leq R_\flat^{n-1} \ .
	\end{equation}
 First let  $r$ be odd. In this case $\fg_{n,r}$ is given by \eqref{def:gg0}. If 
$l \leq r-1$, then we have 
 \begin{align*}
\cN^l_{n,r}
\leq  
\la k_l \ra^{p'} \sum_{k_i \colon i \neq l } \ \ 
\frac{1}{ \prod_{i \neq l} \la k_i\ra^{p'} } \cdot 
\frac{1}{\la  \sum_{\ell=1}^{r-1}{k_\ell}  \ra^{p'} } \leq R_\flat^{n-1} \ , 
\end{align*}
 using estimates \eqref{est.sum0} iteratively.
 If $l=r$, the term $\sum_{\ell=1}^{r-1} k_\ell + j + k_{r+1}$ equals $-\sum_{\ell = r+2}^n k_\ell + k_r$ (due to the condition $\bk \in \fS_{-j}^{n,r}$), thus using again \eqref{est.sum0}
\begin{align*}
\cN^r_{n,r}
\leq  
\la k_r \ra^{p'} \sum_{k_i \colon i \neq r } \ \ 
\frac{1}{ \prod_{i \neq r} \la k_i\ra^{p'} } \cdot 
\frac{1}{\la  \sum_{\ell = r+2}^n k_\ell- k_r\ra^{p'} } \leq R_\flat^{n-1} \ .
\end{align*}
 Finally if $l \geq r+1$, we use that the term  $\sum_{\ell=1}^{r-1} k_\ell + j - k_r$ equals $-\sum_{\ell = r+1}^n k_\ell$ (due to the condition $\bk \in \fS_{-j}^{n,r}$),
thus again \eqref{est.sum0}
 \begin{align*}
\cN^l_{n,r}
\leq  
\la k_l \ra^{p'} \sum_{k_i \colon i \neq l } \ \ 
\frac{1}{ \prod_{i \neq l} \la k_i\ra^{p'} } \cdot 
\frac{1}{\la  \sum_{\ell = r+1}^n k_\ell  \ra^{p'} } \leq R_\flat^{n-1} \ .
\end{align*}
 Now  take  $r$  even;  the relevant formula  for $\fg_{n,r}$ is \eqref{def:gg1}. In case $l \leq r-1$ we have that
  \begin{align*}
\cN^l_{n,r}
\leq  
\la k_l \ra^{p'} \sum_{k_i \colon i \neq l } \ \ 
\frac{1}{ \prod_{i \neq l} \la k_i\ra^{p'} } \cdot 
\frac{1}{\la  \sum_{\ell=1}^{r-1}{k_\ell}  - k_r \ra^{p'} } \leq R_\flat^{n-1} \ .
\end{align*}
  In case $l \geq r$  we use that the term  $\sum_{\ell=1}^{r-1} k_\ell + j $ equals $k_r - 
  \sum_{\ell = r+1}^n k_\ell$ (due to the condition $\bk \in \fS_{-j}^{n,r}$),
 we have that
  \begin{align*}
\cN^l_{n,r}
\leq  
\la k_l \ra^{p'} \sum_{k_i \colon i \neq l } \ \ 
\frac{1}{ \prod_{i \neq r} \la k_i\ra^{p'} } \cdot 
\frac{1}{\la  k_r - \sum_{\ell = r+1}^n k_\ell \ra } \leq R_\flat^{n-1} \ .
\end{align*}
 All together  we have proved \eqref{N.est}, thus 
 $$
 \eqref{g.est} \leq n^{p'(s+1)-1} \sum_{l=1}^n \cN_{n,r}^l \leq (2^{p' (s+1)} R_\flat )^{n-1} \leq (2^{s+2} R_*)^{p'(n-1)} \ ,
 $$
and  \eqref{w2} follows with $R_0 = 1$, $R_1 = 2^{s+2} R_*$.\\

\underline{Case $p=1$.} One proceeds as in the case $1 < p \leq 2$ treating separately $r$ odd and even. One verifies that  \eqref{w2} holds with
$R_0 =  1$, $R_1 = 2^{s+2}$.

\end{document}